\documentclass[11pt]{amsproc}
\pdfoutput=1
\usepackage{amsmath,amssymb,amsthm,eucal, eufrak,amscd,tikz,mathtools}
\usepackage{xcolor}
\usepackage{enumitem}
\usepackage[shortalphabetic]{amsrefs}
\usepackage{tikz-cd}
\usepackage{lipsum}
\usepackage{adjustbox}
\usepackage{graphicx}
\definecolor{allrefcolors}{rgb}{0,0.2,0.5}
\usepackage[linktocpage=true,colorlinks=true,allcolors=allrefcolors,bookmarksopen,bookmarksdepth=3]{hyperref}
\usepackage[margin=1.1in]{geometry}
\usepackage[all]{xy}
\usepackage{comment}

\title{Coulomb branch algebras via symplectic cohomology}
\author{Eduardo Gonz\'alez}
\author{Cheuk Yu Mak}
\author{Dan Pomerleano}

\newtheorem{lem}{Lemma}[section]
\newtheorem{prop}[lem]{Proposition}
\newtheorem{thm}[lem]{Theorem}
\newtheorem{cor}[lem]{Corollary}

\newtheorem{defn}[lem]{Definition}
\newtheorem{lemdefn}[lem]{Lemma/Definition}

\newtheorem{rem}[lem]{Remark}
\newtheorem{example}[lem]{Example}

\newcommand{\cman}{\operatorname{cMan}_K}
\newcommand{\enr}{\operatorname{cENR}_K}
\newcommand{\ienr}{\operatorname{ENR}_K^{\infty}}
\newcommand{\co}{\mathcal{C}(G;0)}
\newcommand{\cE}{\mathcal{C}(G;E)}
\newcommand{\rco}{\mathcal{C}^{\circ}(G;0)}
\newcommand{\rcE}{\mathcal{C}^{\circ}(G;E)}

\def\Spec{\operatorname{Spec}}
\def\grad{\operatorname{grad}}
\def\graph{\operatorname{graph}}
\def\Symp{\operatorname{Symp}}
\def\Ham{\operatorname{Ham}}

\def\critp{\operatorname{critp}}

\def\vdim{\operatorname{vdim}}
\def\ol{\overline}
\def\bH{\mathbf{H}}
\def\R{\mathbb{R}}
\def\ol{\overline}

\def\cS{\mathcal{S}}
\def\cM{\mathcal{M}}
\def\cC{\mathcal{C}}

\def\bfs{\mathbf{s}}

\definecolor{cym}{rgb}{0,0.2,1}\newcommand{\cym}{\color{cym}}

\begin{document}
\maketitle

\begin{abstract}
  Let $(\bar{M},  \omega)$ be a compact symplectic manifold with convex boundary and $c_1(T\bar{M})=0$.  Suppose that $(\bar{M},  \omega)$  is equipped with a convex Hamiltonian $G$-action for some connected, compact Lie group $G$.   We construct an action of the pure Coulomb branch of $G$ on the $G$-equivariant symplectic cohomology of $\bar{M}.$ Building on work of Teleman \cite{Teleman1},  we use this construction to characterize the Coulomb branches of Braverman-Finkelberg-Nakajima \cite{BFN} in terms of equivariant symplectic cohomology.      
\end{abstract}

\section{Introduction} 
\subsection*{Background}
Fix a compact,  connected Lie group $G$ and a quaternionic representation $E$ of $G$.  Physicists associate a 3D N = 4 supersymmetric gauge theory to such a pair.  This gauge theory defines various moduli of vacua,  one of which is called the Coulomb branch, $\mathcal{M}(G, E)$.  Physical considerations predict that $\mathcal{M}(G, E)$ should have a number of remarkable properties,  for example it should be a (possibly singular) hyper-K\"{a}hler manifold with an SU(2)-action.  However,  the physical definition of $\mathcal{M}(G, E)$ involves quantum corrections which are difficult to interpret mathematically.   In \cite{BFN}, Braverman-Finkelberg-Nakajima proposed a definition of $\mathcal{M}(G,E)$  for cotangent type representations $E$,  i.e.  those representations which are isomorphic to  $\mathbb{V}\oplus \mathbb{V}^{\vee}$ for some complex representation $\mathbb{V}$.  For these representations,  they define $\mathcal{M}(G,E)$ to be the spectrum $\operatorname{Spec}(\mathcal{C}(G;E))$ of certain Poisson algebras $\mathcal{C}(G;E)$. For our purposes,  it will be useful to note that the algebras $\mathcal{C}(G;E)$ naturally arise in a one-parameter family $\rcE$  over $\mathbb{C}[\mu]$ (of which $\mathcal{C}(G;E)$ is the zero fiber) which incorporates the central rescaling of the representation (i.e. the diagonal weight one $\mathbb{C}^*$ action on $\mathbb{V}$ and weight minus one action on $\mathbb{V}^{\vee}$).  

The primordial example of a Coulomb branch is the pure Coulomb branch,  $\mathcal{M}(G,0)$,  which occurs when the ``matter representation" $E$ is trivial. The ring  $\mathcal{C}(G;0)$ is by definition the ``semi-infinite" equivariant homology  of the based polynomial loop space of $G$,  $\hat{H}_*^G(\Omega_{poly} G)$,  equipped with its Pontryagin product (see Section \ref{s:basedloop} for the comparison between polynomial loop space and smooth loop space).  An important result (\cite{BFM})  identifies $\mathcal{M}(G,0)$ with the universal centralizer of the Langlands dual group $G^{\vee}.$ The variety $\mathcal{M}(G,0)$ is an affine group scheme over $H^*(BG,\mathbb{C})$ which acts on all of the other Coulomb branches.  Moreover,  each $\mathcal{M}(G,E)$ contains a dense free $\mathcal{M}(G,0)$-orbit.   When the matter representation is non-trivial,  the BFN construction extends and unifies many known constructions in geometric representation theory --- for example slices in affine Grassmannians. 

\subsection*{Pure Coulomb branches} 
If $(M,\omega)$ is a compact (monotone) symplectic manifold with a Hamiltionian $G$ action,  then $\mathcal{C}(G;0)$ acts on the equivariant quantum cohomology,  $QH_G^*(M)$ (\cite{Teleman2,  GMP22}).  Our first result is an analog of this construction for compact symplectic manifolds $(\bar{M},\omega)$ with convex boundary and vanishing first Chern class.  Recall that having convex boundary means that in a neighborhood of the boundary, there is a primitive $\theta$ of $\omega$ such that the $\omega$-dual of $\theta$ points outward along the boundary.  A Hamiltonian $G$-action on $(\bar{M},\omega)$ is convex if $\theta$ can be chosen $G$-invariant and also to have nowhere dense spectrum.  Given $(\bar{M},\omega)$  with a (convex) Hamiltonian $G$-action,  we can consider its equivariant symplectic cohomology,  $SH_G^*(\bar{M}),$ which is the direct limit of equivariant Hamiltonian Floer groups of cylindrical Hamiltonians of positive slope.  We prove:

\begin{thm}\label{t:SeidelMapEqintro}
Let $(\bar{M},  \omega)$ be a compact symplectic manifold with convex boundary and $c_1(T\bar{M})=0$.  Suppose further that $(\bar{M},  \omega)$  is equipped with a convex Hamiltonian $G$-action.  Then there is an algebra homomorphism:   
\begin{align} \label{eq:moduleactintro}
\mathcal{S}: \mathcal{C}(G;0)=\hat{H}_*^G(\Omega_{poly} G) \to SH_{G}^*(\bar{M}). 
\end{align} 
\end{thm}

In fact,  for technical reasons discussed below,  we first construct a ring homomorphism: 
\begin{align} \label{eq:moduleactintroT}  \mathcal{S}_T : \hat{H}_*^{T}(\Omega_{poly} G) \to SH_{T}^*(\bar{M}). \end{align}
We then show that this map is equivariant with respect to the natural Weyl group actions on both sides.  The map  \eqref{eq:moduleactintro} is then obtained by passing to Weyl-invariants.  

\subsection*{Outline of construction}
We now give a brief overview of the construction of  \eqref{eq:moduleactintroT} --- a more detailed overview is  provided in Section \ref{ss:overview}. The starting point for our work is Seidel's observation that for closed symplectic manifolds,  a loop $\gamma_t$ of Hamiltonian symplectomorphisms induces an action on Hamiltonian Floer theory.  Seidel's construction has been extended to include (convex) $S^1$-actions on convex symplectic manifolds in \cite{Ritter, LJ21}.  The new feature,  compared to the closed case,  is that a Hamiltonian loop can modify the behavior of the Hamiltonian at infinity.  This means that given a Hamiltonian $H$ of some fixed slope,  there may be no continuation map between the Floer cohomology of the twisted Hamiltonian, $\gamma_t^*H$,  and the Floer cohomology of $H$ due to the failure of the maximum principle for interpolating solutions.  However,  under suitable geometric assumptions,  there is a well-defined continuation map to a Hamiltonian $H'$ of higher slope,  giving rise to an action on symplectic cohomology.  

The map \eqref{eq:moduleactintroT} is a ``parameterized" or ``family" version of Seidel's construction over cycles in $\Omega G$, the space of {\it smooth} based loops of $G$. 
The idea is to represent cycles in  $\hat{H}_*^{T}(\Omega_{poly} G)$ ``geometrically" by a triple $(B,\alpha,f)$ where $B$ is a finite dimensional smooth closed $T$-manifold, $\alpha \in H_T^*(B)$ and  $f:B \to \Omega G$ is a $T$-equivariant map.  This is an equivariant version of Baum-Douglas' ``geometric homology"(see \cite{BaumDouglas, Jak98}),   the details of which are explained in Section \ref{subsection:geomhom} (see in particular Lemma \ref{l:compareloops}).  
As a consequence, we can use $\hat{H}_*^{T}(\Omega_{poly} G)$ and $\hat{H}_*^{T}(\Omega G)$ interchangeably from now on.
Given a $T$-equivariant map $f:B \to \Omega G$,  we construct a map (cf. \eqref{eq:interSeidel2}) \begin{align} \label{eq:intro:interSeidel2}
\mathcal{S}_{f}:H_T^*(B) \times SH_{T}^*(\bar{M}) \to SH_{T}^*(\bar{M}).
\end{align}
Plugging $\alpha$ into \eqref{eq:intro:interSeidel2}  defines an action of $(B,\alpha,f)$ on $SH_{T}^*(\bar{M})$ and the map  \eqref{eq:moduleactintroT} is then defined by sending: \begin{align*} [B, \alpha,f] \to \mathcal{S}_{f}(\alpha, e_M),  \end{align*}

where $e_M$ denotes the unit in $SH_T^*(\bar{M}).$ The map $\mathcal{S}_{f}$ is in turn obtained as a composition of three maps:
\begin{align*}
&H^*_T(B) \otimes SH_T^*(\bar{M}) \simeq SH_{T \times T}^*(B \times \bar{M}) \to SH_{T}^*(B \times \bar{M})\\
&SH_{T}^*(B \times \bar{M}) \to SH_{T}^*(B \times \bar{M})\\
&SH_{T}^*(B \times \bar{M}) \to SH_{T}^*(\bar{M}).
\end{align*}
The first map is a K\"unneth type isomorphism followed by a Floer theoretic pull-back,  and the last map is a Floer theoretic push-forward.  We develop the necessary preliminaries on pushforward and pullback for family Floer theory in Sections \ref{ss:pb} and  \ref{ss:pf}. The second map is a parametrized version of the Seidel morphism,  and its construction turns out to be most technical.  The main reason is that given a $T$-invariant $B \times ET$-family of Hamiltonian function $\tilde{H}$  of some fixed slope,  the associated twisted Hamiltonian $f^*\tilde{H}$ (see Section \ref{sss:identify}) has varying (non-constant) slopes over the base $B \times ET$.  As a result,  the comparison map between $\varinjlim HF_T(f^*(\tilde{H}_n))$($\tilde{H}_n$ is a sequence of Hamiltonians of increasing slope) and $SH^*_T(B \times \bar{M})$ will involve moduli spaces for which a $C^0$ a priori estimate cannot be proved by any of the standard means (e.g. \cite{Ritter}, \cite{MU19}, \cite{Groman23}).\footnote{We also note that we cannot use Hamiltonians of quadratic growth either; see Remark \ref{rem:quadratic}. } To overcome this issue, we first homotope $f$  (see Lemma \ref{l:TeqHomotope}) to another cycle for which we can prove the required $C^0$ a priori estimate (see Section \ref{ss:tauto}).  Lemma \ref{l:TeqHomotope} is the step which requires us to work $T$-equivariantly first and then pass to $G$-equivariant cohomology by taking the Weyl-invariant part.  

To show that \eqref{eq:moduleactintroT} is well-defined, we use that the relations in the equivariant geometric homology are given by bordisms and compatibility with Gysin pushforward.
We then establish the commutative diagram \eqref{eq:BigCom} to show that \eqref{eq:moduleactintroT} is an algebra homomorphism.  Both of these steps also require careful choice of auxiliary data so that the $C^0$ estimate can be achieved. 

\begin{rem} There is a second key difference between the proof of the ring homomorphism property in the present work and the proof of the analogous claim for closed monotone symplectic manifolds \cite[Lemma 4.7]{GMP22}.  There,  once the map $\hat{H}^{T}_*(\Omega G) \to QH_{T}^*(M)$ had been constructed,  we employed a simple algebraic trick (\cite[Lemma 4.6]{GMP22}) to reduce the argument that it is a ring homomorphism to the case where $G=T$.  This algebraic trick relied on the fact that $QH_{T}^*(M)$ is free as a module over $H^*(BT)$.   As this is no longer necessarily the case for symplectic cohomology,  we are forced to give a more lengthy (though very natural) argument in Section \ref{s:algebra}.   \end{rem}

\begin{rem} \label{rem:LiouvilleC} For the discussion which follows,  it will be relevant to note that above,  $SH_G^*(\bar{M})$ was defined over the Novikov field $\Lambda$ with ground field $\mathbb{C}$.  However,  when $(\bar{M},\omega, \theta)$ is a Liouville domain (meaning $\theta$ can be extended over all of $M$),  symplectic cohomology can be also be defined over $\mathbb{C}$.  The map \eqref{eq:moduleactintro} can also be defined for this version of symplectic cohomology,  which we denote by $SH_G^*(\bar{M}, \mathbb{C})$.  \end{rem}   

\subsection*{Coulomb branches with matter}

In \cite{BFN},  the Coulomb branch algebra is defined via a convolution diagram involving a certain infinite rank vector bundle over the affine Grassmannian of the complexification of $G$.  In \cite{Teleman1},  Teleman gave an elegant and direct ``GLSM" construction of $\rcE$ as the algebra of functions on a scheme obtained by gluing two copies of $\mathring{\mathcal{M}}(G, 0)=\operatorname{Spec}(\rco)$ by group multiplication with a certain easily described ``Euler" rational section $\epsilon_\mathbb{V}$ of the Toda integrable system (see Definition \ref{defn:eulerLag} below for the precise description).  

While Teleman's construction is purely algebraic,  it was motivated by heuristics concerning 2D boundary conditions for the gauge theory with matter $E$ \cite[\S 1, \S 2]{Teleman1}.  Our second result makes these heuristics precise,  thereby providing a conceptual underpinning for this construction.  Given a Liouville domain with convex $G$-action,  it is natural to search for actions on the ordinary cohomology as opposed to the symplectic cohomology.  There is an acceleration map \begin{align} \operatorname{ac}: H^*_G(\bar{M},\mathbb{C}) \to SH_G^*(\bar{M},\mathbb{C}). \end{align}  In cases where this map is injective,  it makes sense to look for subalgebras of $\mathcal{C}(G;0)$ which preserve its image.  The example relevant to Coulomb branches is the seemingly simple case where $\bar{M}=\bar{\mathbb{V}}$ is a unit ball inside of a complex vector space $\mathbb{V}$ equipped with a unitary representation of $G$.  In this case,  we consider the $G \times S^1$ equivariant symplectic cohomology $SH^*_{G\times S^1}(\bar{\mathbb{V}})$,  where the additional $S^1$-factor corresponds to the diagonal rotation action.   The algebra $\rco:=\hat{H}_{G}(\Omega G)[\mu]$ is naturally a subalgebra of $\hat{H}_{G \times S^1}(\Omega (G\times S^1))$ and so \eqref{eq:moduleactintro} therefore induces a map: 
\begin{align} \label{eq:moduleactintro2} \rco \to SH_{G \times S^1}^*(\bar{\mathbb{V}},\mathbb{C}).  \end{align}
  We give the following symplectic interpretation of the ingredients in Teleman's construction:
 
\begin{thm} \label{thm:moduleintro} The following hold: \begin{enumerate}
\item Let $\mathcal{O}_{\epsilon_\mathbb{V}}$ be the pushforward structure sheaf associated with the Euler Lagrangian $\epsilon_\mathbb{V}$. There is an algebra isomorphism
\[ \Gamma(\mathcal{O}_{\epsilon_\mathbb{V}}) \cong SH^*_{G\times S^1}(\bar{\mathbb{V}},\mathbb{C}) \]
under which the $\mathcal{C}^\circ(G;0)$-action on $\Gamma(\mathcal{O}_{\epsilon_\mathbb{V}})$ agrees with action induced by \eqref{eq:moduleactintro2}.

\item There is a commutative diagram: \begin{equation} \label{eq:diagramintro} \begin{tikzcd}
\rcE  \arrow{r}{\mathcal{S}} \arrow[swap]{d}{i} & H^*_{G \times S^1} (\bar{\mathbb{V}},\mathbb{C})  \arrow{d}{\operatorname{ac}} \\
\co[\mu] \arrow{r}{\mathcal{S}} & SH^*_{G \times S^1}(\bar{\mathbb{V}},\mathbb{C}). \end{tikzcd} \end{equation} \end{enumerate}  
\end{thm}

These calculations together with the GLSM construction of the Coulomb branch allow us to completely characterize $\mathring{\mathcal{M}}(G,E)$ in symplectic terms ---  $\mathring{\mathcal{M}}(G,E)$ is the universal equivariant,  affine partial compactification of the group scheme $\mathring{\mathcal{M}}(G,0)$ whose coordinate ring fits into \eqref{eq:diagramintro} (see Corollary \ref{cor:characterization2}).  

\begin{rem} Theorem \ref{thm:moduleintro} was anticipated by Teleman in \cite[Remark 2.4]{Teleman1}.  The universal property from Corollary \ref{cor:characterization2} is not explicitly mentioned in Teleman's paper.  However,  it follows easily from his description of the Coulomb branch and our Theorem \ref{thm:moduleintro}. \end{rem} 


We close the introduction with a short discussion of possible directions for future work:\begin{enumerate}[label=(\roman*)] \item Comparison with \cite{GMP22} suggests a natural (and potentially easy) extension of our results here.  Namely,  \cite{GMP22} also studies actions of the quantum Coulomb branch algebra $\mathcal{C}(G;0)_\hbar:= \hat{H}_*^{S^1 \times G}(\Omega G)$ on loop-equivariant $QH^*$.  This suggests that the module action induced by \eqref{eq:moduleactintro} should lift to an action of $\mathcal{C}(G;0)_\hbar$ on the version of equivariant symplectic cohomology which also incorporates loop equivariance,  $SH_{S^1_{\operatorname{rot}} \times G}^*(\bar{M})$.  Section 7 of \cite{Teleman1} provides a similar characterization of the quantum Coulomb branch algebras $\mathcal{C}(G;E)_\hbar$ as subalgebras of the pure quantum Coulomb branch algebra $\mathcal{C}(G;0)_\hbar$. 
We remark that if $M$ is a closed monotone symplectic manifold, we can also use equivariant Floer theory as in the current paper together with the equivariant PSS map to define the Coulomb branch algebra action on $QH_G(M)$. In other words, the approach given by the current paper is compatible with the construction in \cite{GMP22}. We note that action of the quantum Coulomb branch is used to show that the set-theoretic support of $QH_G^*(M)$ in $\Spec(\mathcal{C}(G;0))$ is an algebraic Lagrangian (\cite[Corollary 1.4]{GMP22}). The argument behind \cite[Corollary 1.4]{GMP22} crucially used the finiteness of $QH_G^*(M)$ over $H^*(BG)$. By contrast, we expect that the set-theoretic support of $SH_G^*(M)$ when $M$ is a $G$-Liouville domain is no longer generally Lagrangian (see Remark \ref{rem:notLagrangian} below). 

\vskip 5 pt \item In a different direction,  constructions of Coulomb branches for representations which are not of cotangent type have recently been proposed in \cite{Raskin, Teleman3}.  Theorem \ref{thm:moduleintro} raises the natural question of whether these constructions can be interpreted in terms of quantum/symplectic cohomology of some convex symplectic manifold associated to the representation (cf.  the discussion in \cite[pp. 2-3]{Teleman3}).  \vskip 5 pt   \item Somewhat more speculatively,  we also note that Teleman develops parallel results for $K$-theoretic Coulomb branches.  Correspondingly,  we expect an analog of Theorem \ref{t:SeidelMapEqintro} for (a suitable genuinely equivariant version of) the $K$-theoretic symplectic cohomology recently constructed by Large \cite{Large}.  \end{enumerate}

\begin{rem} \label{rem:notLagrangian} For a potential example where the set-theoretic support of $SH_G^*(M)$ is not Lagrangian, consider $T^*SU(2)$ with the natural Hamiltonian lift of the $S^1$-action on $SU(2)$ by a maximal torus. View $\mathcal{C}(S^1;0)=\mathbb{C}[z^{\pm},\tau]$, where $\tau$ is the equivariant variable and $z$ corresponds to the Seidel operator. Then we expect that the support of the module $SH_{S^1}^*(T^*SU(2))$ as a subset of $\operatorname{Spec}(\mathcal{C}(S^1;0))=T^*\mathbb{C}^*$ corresponds to the locus where $z=1, \tau=0$. The expectation that $\tau$ should act nilpotently on $SH_{S^1}^*(T^*SU(2))$ comes from the freeness of the $S^1$-action and the fact that the Seidel operator acts by the identity on $SH^*(T^*SU(2))$ (i.e. non-equivariantly) corresponds to the fact that the $S^1$-action extends to an $SU(2)$ action.  However, verifying this expected calculation rigorously would require an equivariant version of Viterbo's isomorphism which is suitably compatible with the Seidel representation and would take us too far afield. \end{rem}

\subsection*{Organization of the paper}

As noted earlier, the family Seidel morphism involves families of Hamiltonians with non-constant slopes, so special care is needed to achieve compactness of moduli spaces. Section \ref{s:functorialbasic} and \ref{s:equivariant} are devoted to explaining the Floer theory over families and the equivariant Floer theory, respectively, for a large enough class of Hamiltonian functions.
We construct the map \eqref{eq:moduleactintro} in Section \ref{s:EqSeiMor} and prove that it is an algebra homomorphism in Section \ref{s:algebra}.
Finally, the proof of Theorem \ref{thm:moduleintro} is given in Section \ref{s:Coulomb}.


\subsection*{Acknowledgements}
 D.P.  would like to thank Constantin Teleman for his generous and patient explanations of \cite{Teleman1}.  The authors would also like to thank Justin Hilburn for informative discussions about Coulomb branches.  C.M. was partly supported by the Simons Collaboration on Homological Mirror Symmetry, Award \#652236, the Royal Society University Research Fellowship, and the University of Southampton FSS ECR Career Development Fund while working on this project. D.P.  was partly supported by the Simons Collaboration in Homological Mirror Symmetry,  Award \# 652299 while working on this project.





\section{Floer theory over a family, pull-back and push-forward}\label{s:functorialbasic}

In this section, we explain how to define the Floer cohomology for a Hamiltonian fibre bundle with Liouvile fibres over a smooth finite dimensional base.  It is analogous to \cite{Hutchings} and similar ideas have been used to define equivariant Floer cohomology (\cite{SS10},  \cite{BO17},  \cite{SeidelGM},  \cite{LJ20}).

Let $(\bar{M},\omega)$ be a compact symplectic manifold with a convex boundary (i.e. $\bar{M}$ is a strong convex filling of its boundary).
Let $\theta$ be a primitive one form of $\omega$ near the boundary such that the $\omega$-dual of $\theta$ points outwards along the boundary.  We will denote by $M$ the symplectic completion of $\bar{M}$.
In other words,
\[
M=\bar{M} \cup_{\partial \bar{M}} [1,\infty) \times \partial \bar{M}  
\]
where the Liouville one form on the cylindrical end $[1,\infty) \times \partial \bar{M}$ is given by $r \alpha$ for $r \in [1,\infty) $ with $\alpha:=\theta|_{\partial \bar{M}}$.  Throughout this text,  we will assume that $c_1(TM)=0$ so that we can put a $\mathbb{Z}$-grading on the Floer complexes later on (see the paragraph after Definition \ref{d:FloerComplex}).

 The Reeb vector field $R_{\alpha}$ on $(\partial \bar{M},\alpha)$ is the unique vector field characterized by $\alpha(R_{\alpha})=1$ and $\iota_{R_{\alpha}} d\alpha=0$.
The flow of the Reeb vector field is called Reeb flow.
The action spectrum of $(\partial \bar{M},\alpha)$, denoted by $\Spec(\partial \bar{M},\alpha)$, consists of the period of the orbits of the Reeb flow.


\begin{defn}
A Hamiltonian function $H \in C^{\infty}(S^1 \times M)$ is {\bf cylindrical} if there is a function $\bfs: S^1 \times \partial \bar{M}  \to \mathbb{R}_{>0}$ that is invariant under the Reeb flow and a function $c:S^1 \to \mathbb{R}$ such that $H(t,(r,x))=r \bfs(t,x)+c(t)$ in the complement of a compact set for any $t \in S^1$, $r \in [1,\infty)$ and $x \in \partial \bar{M}$.
\end{defn}

The function $\bfs$ is called the {\it slope} of $H$.
The space of cylindrical Hamiltonian is a vector space over $\mathbb{R}$ and hence convex.
We define a partial ordering on the space of cylindrical Hamiltonians by
\begin{align}
H_1 \le_{\bfs} H_2 \text{ if }\bfs_{H_1}(t,x) \le \bfs_{H_2}(t,x) \text{ for all } t,x \label{eq:partialorder}
\end{align}
where $\bfs_{H_i}$ is the slope of $H_i$ for $i=1,2$. 
We say that a Hamiltonian function $H=(H_t)_{t \in S^1} \in C^{\infty}(S^1 \times M)$ is {\it mean-normalized} if $\int_{\partial \bar{M}} H_t \alpha \wedge \omega^{n-1}=0$ for all $t \in S^1$.
Unless otherwise stated,  all Hamiltonians in the rest of the paper are mean-normalized.

\begin{defn} Let $G$ be a compact connected Lie group.  A Hamiltonian $G$-action on $(\bar{M},\omega, \theta)$ is called {\it convex} if \begin{itemize} \item $\theta$ is $G$-invariant and \item the spectrum $\Spec(\partial \bar{M},\alpha:=\theta_{|\partial \bar{M}})$ is nowhere dense.  \end{itemize} \end{defn} 
The hypothesis that $\theta$ is $G$-invariant implies that $\omega=d\theta$ is $G$-invariant.
Therefore, the Liouville vector field $X_{\theta}$ (i.e. the $\omega$-dual of $\theta$) is $G$-equivariant. 
As a result, the flow of  $X_{\theta}$ is $G$-equivariant so the $G$-action on the collar $(1-\epsilon,1] \times \partial \bar{M}$, which is defined using the flow of $X_{\theta}$,
is the product of the trivial action on $(1-\epsilon,1]$ and a $G$-action on $\partial \bar{M}$.
Therefore, we can canonically extend the $G$ action to $M$ by extending this product action to $[1,\infty) \times \partial \bar{M}$.

\begin{lem}\label{l:cylindricalFun}
Suppose that $(\bar{M},\omega,\theta)$ is equipped with a convex Hamiltonian $G$-action.


Let $\gamma \in \Omega G \subset \Omega \Ham(M)$ and $K_{\gamma}$ be a generating Hamiltonian function of $\gamma$. Then $K_{\gamma}$ is a cylindrical Hamiltonian function.  Conversely, if $H$ is a cylindrical Hamiltonian function, then outside a compact set, the Hamiltonian flow of $H$ on the cylindrical end preserves $\theta$.
\end{lem}

\begin{proof}

Let $v \in \mathfrak{g}$ and $X_v$ be the induced Hamiltonian vector field.
 Since $\theta$ is $G$-invariant, we have
\begin{align}\label{eq:Hamvf}
0=\mathcal{L}_{X_v} \theta=\iota_{X_v} \omega +d(r \alpha(X_v|_{r=1}))
\end{align}
The last equality uses that $X_v$ is independent of $r$ on the cylindrical end.
Let \begin{align*} \bfs_v:=-\alpha(X_v|_{r=1}): \partial \bar{M} \to \mathbb{R}. \end{align*} 
By \eqref{eq:Hamvf}, Hamiltonian vector field of the Hamiltonian function $r\bfs_v$ is $X_v$.
Let $R_{\alpha}$ be the Reeb vector field.
The function $\bfs_v$, which is the slope of the Hamiltonian function $r\bfs_v$, is invariant under Reeb flow because $R_{\alpha}(\bfs_v)=d\bfs_v(R_{\alpha})=\omega(X_v,R_{\alpha})|_{r=1}=-dr(X_{v})=0$.
Now, let $(X_t)_{t \in S^1}$ be the vector field which generates $\gamma$.
By the assumption that $\gamma \in \Omega G$, we know that for each $t \in S^1$, $X_t=X_{v_t}$ for some $v_t \in \mathfrak{g}$.
Therefore $(K_{\gamma})_t(r,y)=r\bfs_{v_t}(y)$ up to a constant so $K_{\gamma}$ is a cylindrical Hamiltonian function.  The converse is proved in \cite[Lemma C.3]{Ritter}.
\end{proof}



Let $B$ be a smooth finite dimensional closed manifold and $(M,\omega, \theta)$ has a convex Hamiltonian $G$-action.
\begin{defn}\label{d:admissible}
We say that $p:P \to B$ is an {\it admissible} bundle over $B$ if it is a symplectic fibre bundle with fiber $(M,\omega,\theta)$ and the structure group lies in $G$ (i.e. the clutching functions are in $G$).
\end{defn}
Since the $G$ action is convex, $p$ restricts to a subbundle $\bar{p}:\bar{P} \to B$ with fibres $\bar{M}$.
Moreover, there is a radial coordinate function $r:P \setminus \bar{P} \to [1,\infty)$ such that for any fiber $p^{-1}(b)$ , the symplectic form near the boundary equals to $\omega_b=d(r\theta_b)$.

Let $\eta:B \to \mathbb{R}$ be a Morse function, $g_{\eta}$ be a Riemannian metric on $B$ satisfying the Morse-Smale condition and $\nabla$ be a $G$-connection on $P$ such that $\nabla$ is flat near critical points of $\eta$. 
The space of such connections $\nabla$ is non-empty and contractible.
The set of critical points of $\eta$ is denoted by $\critp(\eta)$.
For every $c \in \critp(\eta)$, we can use the flatness of $\nabla$ near $c$ to obtain a neighborhood $N_c$ of $c$ together with a decomposition 
\begin{align}
p^{-1}(N_c) \simeq M \times N_c \label{eq:trivial} 
\end{align}
such that $\nabla$ is trivial with respect to the decomposition.


\begin{defn}
We call a triple $(\eta,g_{\eta},\nabla)$ as above an {\bf admissible base triple}.
\end{defn}

A Hamiltonian function of $P$ is a function $ H: S^1 \times P \to \mathbb{R}$.
For $t \in S^1$ and $b \in B$, we denote its restriction to $p^{-1}(b)$ by $H_b:S^1 \times p^{-1}(b) \to \mathbb{R}$.

\begin{defn}
A Hamiltonian function $H \in C^{\infty}(S^1 \times P)$ is {\bf cylindrical} if for every $b \in B$, $H_b:S^1 \times p^{-1}(b) \to \mathbb{R}$ is a cylindrical Hamiltonian.
\end{defn}

This is a well-defined notion because the structure group of $P$ is $G$ and $H_b$ being cylindrical is a property that is invariant under $G$-action.


\begin{defn}\label{d:eqgen}
 A cylindrical Hamiltonian $H$ is compatible with $(\eta,g_{\eta},\nabla)$ if: \begin{itemize} 
\item $H_c:=(H_{c,t})_{t \in S^1}$ is non-degenerate for every $c \in \operatorname{critp}(\eta)$
\item locally near each $c \in \operatorname{critp}(\eta)$, with respect to the decomposition \eqref{eq:trivial}, we have that $H_b$ is pulled back from the first factor.  
\item for any gradient trajectory\footnote{A gradient trajectory $\tau:\mathbb{R} \to B$ of $\eta$ is a solution to the negative gradient flow equation $ \frac{d \tau}{ds}=-\grad(\eta)(\tau(s))$.} $\tau$ of $\eta$ and two real numbers $s_1 \le s_2$, we have $H_{\tau(s_1)} \ge_{\bfs} H_{\tau(s_2)}$ with respect to the parallel transport map along $\tau$ induced by $\nabla$ (recall $\le_{\bfs}$ from \eqref{eq:partialorder}).
\item for any gradient trajectory $\tau$ of $\eta$ between two critical points $c_0$ and $c_1$, the parallel transport map does not map a Hamiltonian orbit of $H_{c_0}$ to a Hamiltonian orbit of $H_{c_1}.$
\end{itemize}   
A cylindrical Hamiltonian $H$ is compatible with $\eta$ if it is compatible with an admissible base triple $(\eta,g_{\eta},\nabla)$ for some $g_{\eta}$ and $\nabla$.
A cylindrical Hamiltonian $H$ is admissible if it is compatible with some $\eta$.
\end{defn}

\begin{lem}\label{l:existCom}
Given any admissible base triple $(\eta,g_{\eta},\nabla)$ and any constant $a \in \R_{>0} \setminus \Spec(\partial \bar{M},\alpha)$, 
there is a  cylindrical Hamiltonian $H$ that is compatible with $(\eta,g_{\eta},\nabla)$ with slope $\bfs_H=a$. 
\end{lem}

\begin{proof}

Let $a \in  \R_{>0} \setminus \Spec(\partial \bar{M},\alpha)$.
For each critical point $c$ of $\eta$, we first choose a non-degenerate cylindrical Hamiltonian function $H_c$ on $M$ with slope $a$.
Using the product decomposition \eqref{eq:trivial}, we can pull it back to a cylindrical Hamiltonian over a neighborhood $\cup_{c \in \critp(\eta)} N_c$ of the critical points of $\eta$ such that both the first and the second bullets are satisfied.
The fact that $H_c$ is non-degenerate means that all Hamiltonian orbits are lying inside a compact set.
Since there are only finitely many critical points of $\eta$ (because $B$ is compact and $\eta$ is Morse), by a generic perturbation of each individual $H_c$ over a compact set, we can further assume that the fourth bullet is also satisfied.

We can then smoothly extend the cylindrical Hamiltonian to the rest of $P$ (i.e. over the complement of $\cup_{c \in \critp(\eta)} N_c$) such that $H_{b}$ has slope $a$ for all $b \in B$.
Such an extension is always possible because the space of cylindrical Hamiltonian with a fixed slope is convex.
Since the slope is the same for all $H_{b}$, the third bullet is automatically satisfied.

\end{proof}

A pair $\vec{x}=(c,x)$ consisting of a critical point $c$ of $\eta$ and a $1$-periodic orbit $x$ of $H_{c}$ in $p^{-1}(c)$ is called a generator of $(\eta,H)$ if for any loop $g:S^1 \to G$, the torus $g \cdot x: T^2 \to p^{-1}(c)$ defined by $(r,t) \mapsto g(r) \cdot x(t)$ satisfies
\begin{align}\label{eq:zeroarea}
\int_{T^2} (g \cdot x)^*\omega =0.
\end{align}
This is true, for example, (i) when $x$ is contractible in $M$ (i.e. admits a capping), or (ii) $G$ is simply connected (or more generally, $G$ can be extended to a group action $G'$ on $M$ such that $G'$ is simply connected), or (iii) $\omega$ is exact.
Equation \eqref{eq:zeroarea} will be used to control the topological energy of Floer solutions later on (see Lemma \ref{l:zeroarea}).
 
Let $\Lambda$ be the Novikov field
\[
\Lambda :=\left \{ \sum_{i=0}^{\infty} b_i q^{r_i} | r_i \in \mathbb{R}, \lim_{i \to \infty} r_i=\infty, b_i \in \mathbb{C}  \right \}
\]


\begin{defn}\label{d:FloerComplex} The Floer cochain complex of $(\eta,H)$ is defined to be 
\begin{align} \label{eq:Z2graded} 
CF^*(P,H,\eta, \Lambda):= \oplus \Lambda \cdot \vec{x} 
\end{align} 
where the sum is taken over all generators $\vec{x}=(c,x)$ of $(\eta,H)$.
\end{defn} 
 We often suppress the coefficient field and denote the group by $CF^*(P,H,\eta)$.
We can define a $\mathbb{Z}$-grading on $CF^*(P,H,\eta)$ as follows.
Since we assume that $c_1(TM)=0$,  we can and will fix once and for all a homotopy class of trivialization of the canonical bundle of $M$.  We then choose a symplectomorphism identifying $M$ with a reference fibre of $P$. 
The structure group of the bundle $P \to B$ is $G$, which is connected,  so it acts trivially on the homotopy class of the trivialization of the canonical bundle.
Therefore,  by taking parallel transport from the reference fibre,  we get a well-defined homotopy class of trivialization of the canonical bundle on all other fibres.
We define the grading $|\vec{x}|$ of $\vec{x}=(c,x)$ to be  
\[
|\vec{x}|:=|c|_{Morse}+|x|_{CZ}
\]
where the right hand side is the sum of the Morse grading of $c$ and the Floer grading  of $x$, where the Floer grading of $x$ is the {\it cohomological} Conley-Zehnder index of $x$ with respect to the homotopy class of trivialization of the canonical bundle of the fibre $M$ over $c$ (see \cite[Definition 1.4.3 and Remark 1.4.4]{Ab15}).  Finally,  the Novikov variable $q$ will be of degree zero.

\begin{rem} The $\mathbb{Z}$-grading on $CF^*(P,H,\eta)$ depends (only) on the homotopy class of trivialization of the canonical bundle of $M$ as well as the symplectomorphism identifying $M$ with the reference fibre of $P$.  \end{rem}

Choose an $S^1$-dependent fiberwise almost complex structure $(J_b)_{b \in B}$ which is compatible with $(\omega_b)_{b \in B}$ and is of {\it contact type} in each fiber.
Recall that $J_b$ is of contact type if $\theta_b \circ J_b=dr$ outside a compact set of $p^{-1}(b)$.  
Define the moduli space $\widetilde{\mathcal{M}}(\vec{x}_0,\vec{x}_1,(J_b)_{b \in B})$ to be the space of pairs $(\tau, u)$ where \begin{itemize} \item  $\tau : \R \to B$ is  a gradient trajectory of $\eta$ such that $\lim_{s \to -\infty} \tau(s)=c_0$ and $\lim_{s \to \infty} \tau(s)=c_1$
\item $u: \mathbb{R} \times S^1 \to P$ is a smooth map which lifts $\tau$ and solves the Floer equation 
\begin{align} \label{eq:Floer}
\left\{
 \begin{array}{ll}
  (du^{vert}-X_{H_{\tau(s)}}\otimes dt)_{J_{\tau(s)}}^{0,1}=0. \\
  \lim_{s \to -\infty} u(s,t)=x_0(t)  \\
  \lim_{s \to \infty} u(s,t)=x_1(t)
 \end{array}
\right.
\end{align}
where $du^{vert}$ denotes the projection of $du$ to the fibrewise direction $T^{vert}P$ under the splitting $TP=T^{vert}P \oplus T^{hor}P$ by $\nabla$.
\end{itemize}  

We let $$ \mathcal{M}(\vec{x}_0,\vec{x}_1):= \widetilde{\mathcal{M}}(\vec{x}_0,\vec{x}_1)/\mathbb{R}. $$ 
By trivializing $P$ over $\tau$ using the connection $\nabla$,  a solution  $(\tau,u) \in \mathcal{M}(\vec{x}_0,\vec{x}_1)$ can be recast as either: \begin{itemize} \item Floer trajectories in the fiber over $c_0=c_1$ (when $\tau$ is constant) or \item continuation solutions for a Hamiltonian $(H_{s,t})_{(s,t) \in \mathbb{R} \times S^1}$ with respect to an $\omega$-compatible almost complex structure of contact type $(J_{s,t})_{(s,t) \in \mathbb{R} \times S^1}$ on $M$ (when $\tau$ is non constant).   \end{itemize} 
The moduli space $\mathcal{M}(\vec{x}_0,\vec{x}_1)$ is the zero locus of a Fredholm section 
and has a well-defined virtual dimension given by
\[
\vdim \mathcal{M}(\vec{x}_0,\vec{x}_1)=|\vec{x}_0|-|\vec{x}_1|-1.
\]


For $(\tau,u) \in \mathcal{M}(\vec{x}_0,\vec{x}_1)$, 
we define its {\it topological energy} to be
\begin{align}\label{eq:topE}
E(\tau,u):= \int_{\mathbb{R} \times S^1} (du^{vert})^*\omega + \int_{0}^1 (H_{c_0})_t(x_0(t)) dt - \int_{0}^1 (H_{c_1})_t(x_1(t)) dt.
\end{align}
The term $(du^{vert})^*\omega$ is defined to be the pull-back of the fibrewise $2$-form along the composition $T(\mathbb{R} \times S^1) \to TP \to T^{vert}P$.
In other words, if we trivialize $P$ along $\tau$ by $\nabla$ to define the symplectomorphisms
$p_{\tau,s}:P_{\tau(s)} \to P_{c_1}$, we can identify $u$ with the map $u^{triv}(s,t):=p_{\tau,s}(u(s,t))$ to $P_{c_1}$ and $E(\tau,u)$ is nothing but the topological energy of $u^{triv}$ in the literature (cf. \cite[Equation (148)]{AbSe}).\footnote{The first term of \eqref{eq:topE} is the area of $u^{triv}$.}
To justify the name of topological energy in our context, we need to understand the (in)dependence of $E(\tau,u)$ on $(\tau,u)$. If $(\tau,u') \in \mathcal{M}(\vec{x}_0,\vec{x}_1)$ is another element such that $u^{triv}_*[\Sigma]=(u')^{triv}_*[\Sigma]\in H_2(P_{c_1},p_{\tau, -\infty}(x_0) \cup x_1)$, then it is clear that $E(\tau, u)=E(\tau, u')$. However, when we vary $\tau$ in the moduli space $\mathcal{M}(c_0,c_1)$ of gradient trajectory from $c_0$ to $c_1$ modulo $\mathbb{R}$, the symplectomorphism $p_{\tau, -\infty}:P_{c_0} \to P_{c_1}$ may change up to an element of $g \in G$. To resolve this problem, we make the following observation.

\begin{lem}\label{l:zeroarea}
For any generator $\vec{x}$ and any path $g:[0,1] \to G$, the cylinder $g\cdot x:[0,1] \times S^1 \to P_c$ given by $(r,t)\mapsto g(r) \cdot x(t)$ has $\omega$-area $0$. 
\end{lem}

\begin{proof}
Let
\[
c_{\max}(\vec{x}):=\max \left\{ \int (g \cdot x)^*\omega | g:[0,1] \to G \right\}.
\]
The quantity $c_{\max}(\vec{x})$ exists because, by Equation \eqref{eq:zeroarea}, the integral  $\int (g \cdot x)^*\omega $ depends only on the points $g(0),g(1)$ and $G$ is compact.
Suppose the contrary that $c_{\max}(\vec{x})>0$ and $c_{\max}(\vec{x})$ is achieved by $g:[0,1] \to G$.
Let $\tilde{g}:[0,1] \to G$ be the concatenation of $g$ and $g(1)g(0)^{-1}g$ (so $\tilde{g}(0)=g(0)$ and $\tilde{g}(1)=g(1)g(0)^{-1}g(1)$).
Clearly, $\int (\tilde{g} \cdot x)^*\omega=2\int (g \cdot x)^*\omega=2c_{\max}(\vec{x})>c_{\max}(\vec{x})$ so we get a  contradiction.
\end{proof}

\begin{cor}\label{c:energyconstant}
The topological energy $E:\mathcal{M}(\vec{x}_0,\vec{x}_1) \to \mathbb{R}$ is a locally constant function. 
\end{cor}
\begin{proof}
Suppose we are given a one parameter family $(\tau_r,u_r)_{r \in [0,1]}$. We wish to identify the $\omega$-area of $u_0^{triv}$ and $u_1^{triv}$, where $u_i^{triv}$ is defined using the trivialization along $\tau_i$.  Note first that for any path $g(r): [0,1] \to G$ (with $g(0)=\operatorname{id}$), Lemma \ref{l:zeroarea} implies that we have an equality of $\omega$-areas for the cylinders \begin{align} (g \cdot p_{\tau_0, -\infty}(x_0)) \# u_0^{triv} \text{ and } u_0^{triv}, \end{align} where $\#$ denotes concatenation. The symplectomorphisms $p_{\tau_r, -\infty}$ differ by elements of $G$, giving rise to such a path $g(r).$ To finish the argument, we note that \begin{align} [(g \cdot p_{\tau_0, -\infty}(x_0)) \# u_0^{triv}]=[u_1^{triv}] \in H_2(P_{c_1},p_{\tau_1, -\infty}(x_0) \cup x_1). \end{align} To see this, it suffices by compactness to subdivide the interval $[0,1]$ into arbitrarily small intervals so that $g$ is an arbitrarily short path, in which case the claim is obviously true.

\end{proof}

\begin{prop}\label{p:moduliD}
Let $H$ be a cylindrical Hamiltonian function compatible with $\eta$.
\begin{itemize} \item For a generic choice of compatible $J$ that is of contact type, the space $\mathcal{M}(\vec{x}_0,\vec{x}_1)$ is a manifold of the expected dimension for any pair of generators $(\vec{x}_0,\vec{x}_1)$.
\item Moreover, if $\vdim \mathcal{M}(\vec{x}_0,\vec{x}_1) \le 1$, then for any $E \in\mathbb{R}$, the moduli space
\[
\mathcal{M}_{\le E}(\vec{x}_0,\vec{x}_1) :=  \left\{  (\tau,u) \in \mathcal{M}(\vec{x}_0,\vec{x}_1) | E(\tau,u) \le E \right\}
\]
admits a 
Gromov compactification $ \ol{\mathcal{M}}_{\le E}(\vec{x}_0,\vec{x}_1)$ making it a compact manifold with boundary.  \end{itemize} 
\end{prop}

\begin{proof}
Since $M$ is a symplectic Calabi-Yau manifold with a convex end,  the strategy of proof for both transversality (see e.g. \cite{MSJbook}, \cite{HoferSalamon}) and compactness (\cite{Viterbo},  \cite{Ritter}) is standard.

To run the transversality argument, we need to avoid solutions $(\tau,u)$ such that $\tau$ is non-constant but $u$ is $s$-independent with respect to the parallel transport map along $\tau$ induced by $\nabla$. 
These solutions are ruled out exactly by the fourth bullet of Definition \ref{d:eqgen}.

To run the argument for compactness, first note that when $\vdim \mathcal{M}(\vec{x}_0,\vec{x}_1) \le 1$, then by genericity, we can assume that $u(s,t)$ misses all $J_{s,t}$-holomorphic spheres for all 
$(s,t) \in \mathbb{R} \times S^1$ and $u \in \mathcal{M}(\vec{x}_0,\vec{x}_1)$ because the image of the evaluation map of a $1$-pointed $J_{s,t}$-holomorphic sphere is of real codimension $4$.\footnote{For each $J_{s,t}$, the moduli space of $1$-poined $J_{s,t}$-holomorphic spheres in the class $A$ has virtual dimension $2(n-3)+c_1(A)+2=2n-4$.}
In particular, we can avoid sphere bubblings.
It remains to show that we have a $C^0$  a priori estimate and a uniform geometric energy bound of solutions $(\tau,u)$.
The $C^0$ a priori estimate for $u$ is proved in \cite[Theorem C.11]{Ritter}.
More precisely, the fact that $H$ is a cylindrical Hamiltonian implies that conditions (2) and (3) in \cite[Theorem C.11]{Ritter} are satisfied (cf. \cite[Theorem C.6]{Ritter}). The condition that $(H_z,J_z)$ becomes independent of $s$ for large $|s|$ follows from the second bullet of Definition \ref{d:eqgen}.
The condition $\partial_s H_z \le 0$ in \cite[Theorem C.11]{Ritter} translates to $\partial_s H_{\tau(s)} \le 0$ in our setting, which is not necessarily true on the nose because there is a constant term $H_{\tau(s)}$.
But the Floer equation does not depend on the constant term in $H_{\tau(s)}$ so the Theorem remains valid as long as $\partial_s \bfs_{H_{\tau(s)}} \le 0$, which is precisely the third bullet of Definition \ref{d:eqgen}.

To obtain the uniform geometric energy bound on $u$,  we suppose that $u$ lies over a given gradient trajectory $\tau$ and trivialize $P$ along this trajectory.  We can then apply the standard energy estimate to give an upper bound to the geometric energy of $u^{triv}$ in terms of the topological energy of $u^{triv}$  (i.e. $E(\tau,u)$) and the Hamiltonian $H$.  
\end{proof}




\begin{rem}
The topological energy extends to a locally constant function on  $ \ol{\mathcal{M}}_{\le E}(\vec{x}_0,\vec{x}_1)$ and is additive under breaking (i.e. if $((\tau_0,u_0),(\tau_1,u_1))$ is a boundary point of $\ol{\mathcal{M}}_{\le E}(\vec{x}_0,\vec{x}_1)$, then its topological energy is the sum $E(\tau_0,u_0)+E(\tau_1,u_1)$).
\end{rem}

Choose a generic $J$ and define a differential on $CF(P,H, \eta)$ by
 \begin{align} 
\partial_{CF}(\vec{x}_1):= \sum_{\vec{x}_{0}, |\vec{x}_{0}|=|\vec{x}_{1}|+1} \sum_{(\tau,u) \in \mathcal{M}_0(\vec{x}_0,\vec{x}_1)} s(u)q^{E(\tau,u)} \vec{x}_0   \label{eq:FloerDiff} 
\end{align} 
where $s(u) \in \{-1,1\}$ is the sign of $(\tau,u)$.
The sum is well-defined because of the compactness of the moduli (Proposition \ref{p:moduliD}).
For a discussion of how the signs are determined, see \cite[Section 3 and 6]{Hutchings}.


\begin{lem}\label{l:Floercohomology}
 We have that $$\partial_{CF}^2=0$$  
Moreover, $HF(P,H, \eta):=H(CF(P,H, \eta), \partial_{CF})$ is independent of the choice of $J$.
\end{lem} 


\subsection{Continuation and invariance}

We want to discuss the natural maps arising from varying $H$ and $(\eta,g_{\eta},\nabla)$.

Let $H'$ be another cylindrical Hamiltonian on $P$ compatible with $(\eta,g_{\eta},\nabla)$ and such that the slope $\bfs_{H'_b} \ge \bfs_{H_b}$ for all $b \in B$.
In this case, we can define the continuation map 
\begin{align}
HF(P,H,\eta) \to HF(P,H',\eta) \label{eq:Gcont2}
\end{align}
as follows.
A {\it monotone homotopy} from $H$ to $H'$ is a one-parameter family of cylindrical Hamiltonian $(H_s)_{s \in \mathbb{R}}$
such that $H_s=H'$ for $s \ll 0$, $H_s=H$ for $s \gg 0$ and for every $b \in B$, we have $\partial_s \bfs_{(H_s)_b} \le 0$.
We say that a monotone homotopy is {\it compatible with $(\eta,g_{\eta},\nabla)$} if $\partial_s \bfs_{(H_s)_{\tau(s)}} \le 0$ along any gradient trajectory $\tau:\mathbb{R} \to B$ of $\eta$ (again, with respect to the parallel transport map along $\tau$ induced by $\nabla$).
Note that for a monotone increasing smooth function $\rho: \R \to [0,1]$ such that $\rho(s)=0$ for $s \ll0$ and $\rho(s)=1$ for $s \gg 0$, the family \begin{align*} H_s=\rho(s)H+(1-\rho(s))H', \quad (\text{for s} \in \R) \end{align*} is a monotone homotopy that is compatible with $(\eta,g_{\eta},\nabla)$ if both $H$ and $H'$ are compatible with $(\eta,g_{\eta},\nabla)$.

Let $\vec{x}=(c,x)$ and $\vec{x}'=(c',x')$ be a generator of $CF(P,H,\eta)$ and $CF(P,H',\eta)$ respectively.
Choose a one-parameter family of $S^1$-dependent fiberwise compatible almost complex structures $(J_s)_{s \in \mathbb{R}}$ that are of contact type.
Define the moduli space $\mathcal{M}(\vec{x}',\vec{x},(H_s)_{s \in \mathbb{R}})$ to be the space of pairs $(\tau, u)$ where \begin{itemize} \item  $\tau : \R \to B$ is  a gradient trajectory of $\eta$ such that $\lim_{s \to -\infty} \tau(s)=c'$ and $\lim_{s \to \infty} \tau(s)=c$
\item $u: \mathbb{R} \times S^1 \to P$ is a smooth map which lifts $\tau$ and solves the Floer equation 
\begin{align} \label{eq:Floer2}
\left\{
 \begin{array}{ll}
  (du^{vert}-X_{(H_s)_{\tau(s)}}\otimes dt)_{(J_s)_{\tau(s)}}^{0,1}=0. \\
  \lim_{s \to -\infty} u(s,t)=x'(t)  \\
  \lim_{s \to \infty} u(s,t)=x(t)
 \end{array}
\right.
\end{align}
\end{itemize}  
The virtual dimension of $\mathcal{M}(\vec{x}',\vec{x},(H_s)_{s \in \mathbb{R}})$ is $|\vec{x}'|-|\vec{x}|$. 
We can define the topological energy of $(\tau,u)$ using the same formula \eqref{eq:topE} with $x_0$ and $x_1$ being replaced with $x'$ and $x$, respectively.

\begin{prop}
Let $(H_s)_{s \in \mathbb{R}}$ be a monotone homotopy from $H$ to $H'$ that is compatible with $(\eta,g_{\eta},\nabla)$.
\begin{itemize} \item For a generic choice of $(J_s)_{s \in \mathbb{R}}$, the space $\mathcal{M}(\vec{x}',\vec{x},(H_s)_{s \in \mathbb{R}})$ is a manifold of the expected dimension for any pair of generators $(\vec{x},\vec{x}')$.

\item Moreover,  if its virtual dimension is $\le 1$, then for any $E \in\mathbb{R}$,
$$\mathcal{M}_{\le E}(\vec{x}',\vec{x},(H_s)_{s \in \mathbb{R}}):=\left\{ (\tau,u) \in \mathcal{M}(\vec{x}',\vec{x},(H_s)_{s \in \mathbb{R}}) | E(\tau,u ) \le E    \right\}$$
admits a 
Gromov compactification $\ol{\mathcal{M}}_{\le E}(\vec{x}',\vec{x},(H_s)_{s \in \mathbb{R}})$ making it a compact manifold with boundary.   \end{itemize} 
\end{prop}

The proof is similar to Proposition \ref{p:moduliD}. The compatibility with $(\eta,g_{\eta},\nabla)$ provides us a $C^0$ a priori estimate of the solutions as above. Therefore, we can define a linear map
\begin{align}
 \begin{array}{ll}
\kappa_{(H_s)_{s \in \R}}:CF(P,H,\eta) \to CF(P,H',\eta)\\
\kappa_{(H_s)_{s \in \R}}(\vec{x})= \sum_{\vec{x}', |\vec{x}'|=|\vec{x}|} \sum_{(\tau,u) \in \mathcal{M}(\vec{x}',\vec{x},(H_s)_{s \in \mathbb{R}})} s(\tau,u)q^{E(\tau,u)} \vec{x}'
\end{array}
\end{align}
where $s(u) \in \{-1,1\}$ is the sign of $(\tau,u)$.

\begin{lem}\label{l:kappaIndchoice}
The linear map $\kappa_{(H_s)_{s \in \R}}$ is chain map.
Moreover, the induced map on cohomology is independent of the choice of the $(\eta,g_{\eta},\nabla)$-compatible monotone homotopy and the family of cylindrical almost complex structure.
\end{lem}

\begin{proof}
The fact that  $\kappa_{(H_s)_{s \in \R}}$ is a chain map follows from looking at the boundary of the one dimensional moduli of 
$\ol{\mathcal{M}}(\vec{x}',\vec{x},(H_s)_{s \in \mathbb{R}})$.
The induced map on cohomology being independent of  $J$ can be proved by the standard homotopy argument \cite{MSJbook}.
To run the argument for the proof of the independence of choice of monotone homotopy, we need to show that the space of monotone homotopies that are compatible with $(\eta,g_{\eta},\nabla)$ is path connected.
Indeed, the space is convex so it is path connected.
\end{proof}

It is more tricky to vary $\eta$.
Let $H$ and $H'$ be cylindrical Hamiltonians that are compatible with $(\eta,g_{\eta},\nabla)$ and $(\eta',g_{\eta}',\nabla')$ respectively.
Let $(\eta_s)_{s \in \R}: B \to \mathbb{R}$ be a family of functions such that $\eta_s=\eta'$ for $s \ll 0$ and $\eta_s=\eta$ for $s \gg 0$, $(\nabla_s)_{s \in \R}$ be a family of connections for $P \to B$ such that $\nabla_s=\nabla'$ for $s \ll 0$ and $\nabla_s=\nabla$  for $s \gg 0$,
and $(H_s)_{s \in \R}: P \to \mathbb{R}$ be a family of cylindrical Hamiltonians 
such that $H_s=H'$ for $s \ll 0$ and $H_s=H$  for $s \gg 0$.
We can use $\eta_s$ to define a chain map between the Morse cochain of $\eta$ and $\eta'$ by counting solutions $\tau : \R \to B$ of
\[
\frac{d}{ds}\tau(s)=-\operatorname{grad}(\eta_s,g_{\eta_s})|_{\tau(s)}.
\]
Suppose that $(H_s)_{s \in \R}$ is a family of cylindrical Hamiltonian on $P$ which is compatible with $((\eta_s)_{s \in \R},(g_{\eta_s})_{s \in \R}, (\nabla_s)_{s \in \R})$ in the sense that
\begin{align}
\partial_s \bfs_{H_{\tau(s)}} \le 0 \text{ along  any gradient trajectory $\tau$ of $\eta_s,g_{\eta_s}$ with respect to $(\nabla_s)_{s \in \R}$.}\label{eq:non-incre-eta}
\end{align}
Then, by choosing a generic family of $S^1$-dependent cylindrical almost complex structure of contact type $(J_s)_{s \in \mathbb{R}}$ on $P$, we can define a chain map (the continuation map)
\begin{align}
 \begin{array}{ll}
\kappa_{(H_s)_{s \in \R}, (\eta_s)_{s \in \R}}: CF(P,H,\eta) \to CF(P,H',\eta')\\
\kappa_{(H_s)_{s \in \R},(\eta_s)_{s \in \R}}(\vec{x})= \sum_{\vec{x}', |\vec{x}'|=|\vec{x}|} \sum_{(\tau,u) \in \mathcal{M}(\vec{x}',\vec{x},(H_s)_{s \in \mathbb{R}},(\eta_s)_{s \in \R})} s(u)q^{E(\tau,u)} \vec{x}'
\end{array}
\end{align}
where $\mathcal{M}(\vec{x}',\vec{x},(H_s)_{s \in \mathbb{R}},(\eta_s)_{s \in \R})$ is the moduli of rigid solutions $(\tau,u)$ such that $\tau$ is a gradient trajectory of $(\eta_s,g_{\eta_s})_{s \in \R}$ and $u$ is a solution of \eqref{eq:Floer2} covering $\tau$. The term $E(u)$ and $s(u) \in \{-1,-1\}$ are again the topological energy \eqref{eq:topE} and the sign of $u$, respectively.
When  $(\eta_s,g_{\eta_s},\nabla_s)_{s \in \R}$ are independent of $s$, it recovers $\kappa_{(H_s)_{s \in \R}}$.

In general, the condition \eqref{eq:non-incre-eta} is implicit and hard to verify so we introduce the following defintion.
\begin{defn}
The relation $\le_{P}$ on the space of cylindrical Hamiltonians on $P$ is defined by
\begin{align}
H' \ge_{P} H \text{ if there is } a \notin \Spec(\partial \bar{M},\alpha) \text{ such that }\min \bfs_{H'} \ge a \ge \max \bfs_{H} \label{eq:partialorderP}
\end{align}
\end{defn}
To be crystal clear, the definitions are $\min \bfs_{H'} :=\min_{b\in B} \min_{\partial \bar{p}^{-1}(b)}\bfs_{H'_b}$
and $\max \bfs_{H}:=\max_{b \in B} \max_{\partial \bar{p}^{-1}(b)}\bfs_{H_b}$.
Since $\Spec(\partial \bar{M},\alpha)$ is nowhere dense, we can rewrite the condition \eqref{eq:partialorderP} as 
$\min \bfs_{H'} > \max \bfs_{H}$ or $\min \bfs_{H'} = \max \bfs_{H} \notin \Spec(\partial \bar{M},\alpha)$.
Note that this relation is transitive but not reflexive (cf. \eqref{eq:dir} and Corollary \ref{c:directSystem}).

\begin{lem}\label{l:inde_kappa}
Let $H$ and $H'$ be cylindrical Hamiltonians that are compatible with $(\eta,g_{\eta},\nabla)$ and $(\eta',g_{\eta}',\nabla')$ respectively.
Suppose that $H' \ge_{P} H$.
Then there exist $(\eta_s,g_{\eta_s},\nabla_s)_{s \in \R}$ and $(H_s)_{s \in \R}$ as above such that $(H_s)_{s \in \R}$ is compatible with $(\eta_s,g_{\eta_s},\nabla_s)_{s \in \R}$.
Moreover, the induced map on cohomology
\begin{align}
\kappa_{(H_s)_{s \in \R}, (\eta_s)_{s \in \R}}: HF(P,H,\eta) \to HF(P,H',\eta') \label{eq:kappaCoh}
\end{align}
is independent of the choice of $(\eta_s,g_{\eta_s},\nabla_s)_{s \in \R}$, $(H_s)_{s \in \R}$ and $(J_s)_{s \in \mathbb{R}}$.
\end{lem}

\begin{proof}

For fixed $(\eta_s,g_{\eta_s},\nabla_s)_{s \in \R}$, the space of $(H_s)_{s \in \R}$ from $H$ to $H'$ that is compatible with $(\eta_s,g_{\eta_s},\nabla_s)_{s \in \R}$ is convex so by a homotopy argument as in the proof of Lemma \ref{l:kappaIndchoice}, we conclude that it is independent of the choice of $(H_s)_{s \in \R}$ and $(J_s)_{s \in \mathbb{R}}$.

To argue that it is independent of the choice of $(\eta_s,g_{\eta_s},\nabla_s)_{s \in \R}$, we first consider the case that the slopes are constant functions that are independent of $b \in B$.
In other words, we have two real numbers $\bfs', \bfs$ such that $\bfs_{H'_b}=\bfs' $ and $\bfs_{H_b}=\bfs $ for all $b \in B$.
Moreover, we have  $\bfs' \ge \bfs$ and $\bfs', \bfs \notin \Spec(\partial \bar{M},\alpha)$.
In this case, we can choose a monotone decreasing function $f_{\bfs}:\R \to [\bfs,\bfs']$ such that $f_{\bfs}(s)=\bfs'$ for $s \ll0$, $f_{\bfs}(s)=\bfs$ for $s \gg0$ and a family
$(H_s)_{s \in \R}$ from $H$ to $H'$  such that $\bfs_{(H_s)_b}=f_{\bfs}(s)$ for all $b \in B$.
This choice of $(H_s)_{s \in \R}$ is compatible with any family $(\eta_s,g_{\eta_s},\nabla_s)_{s \in \R}$ from $(\eta,g_{\eta},\nabla)$ to $(\eta',g_{\eta}',\nabla')$.
Therefore, we can apply a homotopy argument with fixed $(H_s)_{s \in \R}$ and varying $(\eta_s, g_{\eta_s},\nabla_s)_{s \in \R}$ to conclude that \eqref{eq:kappaCoh} is independent of $(\eta_s, g_{\eta_s},\nabla_s)_{s \in \R}$ when $\bfs_{H'_b}=\bfs' $ and $\bfs_{H_b}=\bfs $ for all $b \in B$.

Now, we consider the general case.
Let $\check{\bfs} \in \R_{>0} \notin \Spec(\partial \bar{M},\alpha)$ be such that $\min \bfs_{H'} \ge \check{\bfs} \ge \max \bfs_{H}$.
Let $\check{H}$ and $\check{H}'$ be a cylindrical Hamiltonian such that $\bfs_{\check{H}_b}=\check{\bfs}=\bfs_{\check{H}'_b}$ for all $b \in B$ and is compatible with $(\eta,g_{\eta},\nabla)$ and $(\eta',g_{\eta}',\nabla')$, respectively.

We are going to show that \eqref{eq:kappaCoh} equals to the composition
\begin{align}
HF(P,H,\eta) \to HF(P,\check{H},\eta) \to HF(P,\check{H}',\eta') \to HF(P,H',\eta') \label{eq:compo3}
\end{align}
where the first and last maps are independent of choices by Lemma \ref{l:kappaIndchoice} and the second map is independent of choices by above.

First of all, by gluing the auxiliary data defining these three maps, we can show the existence of $(H_s)_{s \in \R}$ that is compatible with $(\eta_s, g_{\eta_s},\nabla_s)_{s \in \R}$.

Conversely, let $(\eta_s, g_{\eta_s},\nabla_s)_{s \in \R}$ and $(H_s)_{s \in \R}$ be a choice of data used to define \eqref{eq:kappaCoh}. By the independent of choice of $(H_s)_{s \in \R}$, we can assume that it satisfies the following property:
there exists $R_1 > R_2 > 0$ such that 
\begin{itemize}
\item $\eta_s=\eta'$ and $\nabla_s=\nabla'$ for $s<-R_2$, and $\eta_s=\eta$ and $\nabla_s=\nabla$ for $s>R_2$
\item $H_{-R_1}=\check{H}'$, $H_{R_1}=\check{H}$
\item for $s \in [-R_1,R_1]$, $\bfs_{(H_{s})_b}$ is a constant independent of $b \in B$ and $\partial_s \bfs_{H_s} \le 0$
\end{itemize}
We have shown that the map $HF(P,\check{H},\eta) \to HF(P,\check{H}',\eta')$ does not depend on auxiiliary choices. In particular, we can use $(\eta_s, g_{\eta_s},\nabla_s)_{s \in \R}$ to define it.
Therefore, by choosing $R_1 \gg R_2 \gg 0$, we can make $(\eta_s, g_{\eta_s},\nabla_s)_{s \in \R}$ and $(H_s)_{s \in \R}$ coincide with the glued data coming the composition of the three maps in \eqref{eq:compo3}, where the firs map corresponds to $s<-R_1$, the second map corresponds to $-R_1<s<R_1$ and the third map corresponds to $s>R_1$.
It shows that \eqref{eq:kappaCoh} equals to \eqref{eq:compo3}.
\end{proof}



Even though we denote the Floer cohomology as $HF(P,H,\eta)$, it also depends on $(g_{\eta},\nabla)$.
For the collection of data $\{(H,\eta,g_{\eta},\nabla)\}$ such that $H$ is cylindrical and compatible with $(\eta,g_{\eta},\nabla)$, we define a reflexive and transitive relation on it by
\begin{align}\label{eq:dir}
(H',\eta',g_{\eta}',\nabla') \ge_P (H,\eta,g_{\eta},\nabla) \text{ if } H' \ge_P H \text{ or the quadruples are identical}
\end{align}
Clearly, there is an upper bound for any two elements so it forms a directed set.

\begin{cor}\label{c:directSystem}
The collection of cohomology groups 
\[
\{HF(P,H,\eta): H \text{ is cylindrical and compatible with } (\eta,g_{\eta},\nabla)\}
\]
together with the continuation maps forms a direct system with respect to $\ge_{P}$.

In particular, when $\bfs_{H_b}$ is a constant independent of $b \in B$,  then $HF(P,H,\eta)$ is independent of the choice of $\eta$ such that $H$ is compatible with $\eta$.
\end{cor}

\begin{proof}[Sketch of proof of Corollary \ref{c:directSystem}]
For $i=0,1,2$, let $\bH_i:=HF(P,H_i,\eta_i)$ be elements in the collection such that $H_0 \le_P H_1 \le_P H_2$.
Let $\kappa_{ij}$ be the continuation map from $\bH_i$ to $\bH_j$ for $i \le j$.
By gluing the auxiliary data defining $\kappa_{01}$ and $\kappa_{12}$, we can express the composition $\kappa_{12} \circ \kappa_{01}$ as a map given by counting rigid solutions with respect to the glued auxiliary data.
Since the map $\bH_0 \to \bH_2$ is independent of choices (Lemma \ref{l:inde_kappa}), it shows that $\kappa_{01} \circ \kappa_{12}=\kappa_{02}$.

\end{proof}

As a consequence of Corollary \ref{c:directSystem}, for each $a \in \R_{>0} \setminus \Spec( \partial \bar{M},\alpha)$, the Floer cohomology
\[
HF(P,a):=HF(P,H,\eta)
\] 
is independent of admissible $H$ having $\bfs_{H_b}=a$ up to natural isomorphisms.
Let $$SH^*(P):= \varinjlim_{a} HF(P,a).$$

\subsection{Product structure on $SH^*(P)$}\label{s:product_structure}

We are going to define a product structure on $SH^*(P)$.
Roughly speaking, it comes from counting Floer pair of pants lying over a Morse gradient tree with two inputs and one output.
To ensure that the perturbation term in the Floer equation is smooth, we need to pay special attention at the trivalent point of the Morse gradient tree. The details are given in the section.

Let  $T_{2,1}$ be the unique trivalent Stasheff tree with three external edges, two of which are incoming and one
 which is outgoing. We parametrize the outgoing edge  by $e_0:(-\infty, 0] \to T_{2,1}$, and 
the two incoming edges by $e_1:[0,\infty) \to T_{2,1}$  and $e_2: [0,\infty) \to T_{2,1}$, respectively.
We identify the edges with  $(-\infty,0]$ and  $[0,\infty)$ respectively using the parametrizations.

 The classical approach to achieving transversality for these operations involves equipping each edge with a different Morse function.  Rather than doing this, we achieve transversality of gradient flow solutions for maps from $T_{2,1}$ by 
 perturbing the gradient flow equation (for a single Morse function). The main definition is the following \cite[Definition 2.6]{Abouzaidtop}:  
 \begin{defn}
    A gradient flow perturbation datum on $T_{2,1}$ is a choice, for each edge
    $ e \in E(T_{2,1})$ of a smoothly varying family of vector fields, 
    $$X_e: e \to C^{\infty}(TB) $$ 
    vanishes away from a bounded subset of $e.$
\end{defn}  

Given a gradient flow perturbation datum as above, for each edge $e \in
E(T_{2,1})$ and any map $\tau: e \to B$, one can ask for $\tau$ to solve
the perturbed gradient flow equation (for $\eta$ with respect to $X_e$):
\begin{align}\label{eq:pertgradflow} 
    \frac{d}{ds}\tau(s)= (-\grad(\eta) + X_e(s))|_{\tau(s)} \text{ for all } s \in e.
\end{align}

\begin{defn} \label{defn:Ymoduli} 
    Let $\eta :B \to \mathbb{R}$, be a Morse function and fix gradient flow
    perturbation data $\{X_e\}_{i = 0,1,2}$ as above. Suppose that $c_0, c_1, c_2$ lie in
    $\critp(\eta)$.
With respect to this data, let 
    \begin{equation}\mathcal{M}(c_0,c_1,c_2)\end{equation} 
    denote the moduli space of continuous maps $\tau: T_{2,1} \to B$ whose
    restriction to each edge is  a solution to \eqref{eq:pertgradflow}.
\end{defn}

 Note that near the vertex, the perturbation data can be arbitrary.
 It is not difficult to show that for a generic choice of
 perturbation data, our moduli spaces are cut out transversally. Somewhat
 informally, this corresponds to the fact that infinitesimally, solutions to
 the perturbed gradient flow equations correspond to intersections of the
 unstable and stable manifolds under perturbations by the diffeomorphisms
 $\phi_e$ given by integrating the
 vector fields $X_e$. As the vector fields $X_e$ can be chosen arbitrarily,
 these diffeomorphisms are essentially arbitrary (for a complete proof, see
 \cite[Section 7]{Abouzaidtop}). Futhermore, when our data is chosen generically,
 the zero dimensional components of the moduli spaces above induce maps between
 orientation lines as before, hence an operation on Morse complexes.

Let $\Sigma=\mathbb{P}^1 \setminus \{0,1,2\}$.
For $i=1,2$,  let $\epsilon_i:[0,\infty) \times S^1 \to \Sigma$ be a holomorphic embedding such that $\lim_{s \to \infty} \epsilon_i(s,t)=i \in \mathbb{P}^1$.
Also let $\epsilon_0:(-\infty,0] \times S^1 \to \Sigma$ be a holomorphic embedding such that $\lim_{s \to -\infty} \epsilon_i(s,t)=0 \in \mathbb{P}^1$.
In other words, $\epsilon_i$ is a positive cylindrical end and $\epsilon_0$ a negative cylindrical end.
We require that the images of $\epsilon_i$ are pairwise disjoint. 
Let $\pi_{\Sigma}: \Sigma \to T_{2,1}$ be the continuous map such that $\pi_{\Sigma}(\epsilon_i(s,t))=e_i(s)$ for all $i=0,1,2$, and it maps the complement of the cylindrical ends of $\Sigma$ to the trivalent vertex of $T_{2,1}$.
Let $a_0,a_1,a_2 \in \mathbb{R}_{>0} \setminus \Spec(\partial \bar{M},\alpha)$ be such that $a_0> a_1+a_2$.
Let  $(\eta,g_{\eta},\nabla)$ be an admissible base triple as before, and $H_i $ be a cylindrical  Hamiltonian of $P$  that is compatible with $(\eta,g_{\eta},\nabla)$ and has a constant slope $a_i$  (cf. Lemma \ref{l:existCom}).
We choose a perturbation data which consists of a domain and time dependent cylindrical Hamiltonian $H=(H_z)_{z \in \Sigma}$ of $P$ (i.e. $H_z$ is a $S^1$-dependent cylindrical Hamiltonian of $P$) and $\beta \in \Omega^1(\Sigma)$ such that  
\begin{itemize}
\item there is a constant $a'\in (a_1+a_2,a_0)$ such that if $z \notin \epsilon_i([1,\infty) \times S^1)$ for all $i$, then $H_z$ the slope of $(H_z)_b$ is $a'$ for all $b \in B$ 
\item for each $i=0,1,2$, there is $R \gg 0$ such that $H_{\epsilon_i(s,t)}=H_i$ when $|s|>R$
\item for all $z \in \Sigma$, $\bfs(z):=\bfs_{H_z}$ is a constant function on $\partial \bar{M}$
\item for each $i=0,1,2$, we have $\partial_s \bfs_{H_{\epsilon_i(s,t)}} \le 0$ over all $(s,t)$
\item $\beta=dt$ over the cylindrical ends, and
\item $d (\bfs(z) \beta) \le 0$
\end{itemize}
The condition that $a_0> a_1+a_2$ guarantees the existence of $H$ and $\beta$ satisfying all the conditions.
We can view $H \otimes \beta \in \Omega^1(\Sigma,C^{\infty}(S^1 \times P))$ and it will give us a perturbation term in the upcoming Floer equation we consider.

We also choose a domain and time dependent contact type fibrewise compatible almost complex structure $J=(J_z)_{z \in \Sigma}$ of $P$ such that 
\begin{itemize}
\item for each $i=0,1,2$, there is $R \gg 0$ such that $J_{\epsilon_i(s,t)}=J_i$ when $|s|>R$, where $J_i$ is an almost complex structure defining $CF(P,H_i, \eta)$
\end{itemize}
For $i=0,1,2$, let $\vec{x}_i:= (c_i, x_i)$ a generator in $CF(P,\bfs_i)$. 
We define $\mathcal{M}(\vec{x}_0,\vec{x}_1, \vec{x}_2)$ to be the moduli space of pairs $(\tau,u)$ such that
$\tau \in \mathcal{M}(c_{0}, c_{1}, c_{2})$ and $u:\Sigma  \to P$ satisfies the following statements
\begin{itemize}
\item $u(z) \in P_{\tau(\pi_{\Sigma}(z))}$
\item  $(du^{vert}-X_{(H_z)_{\tau(\pi_{\Sigma}(z))} }\otimes \beta)^{0,1}_{(J_s)_{\tau(\pi_{\Sigma}(z))}}=0$  
\item $\lim_{|s| \to \infty} u(\epsilon_i(s,t))=x_i(t)$,
\end{itemize}
In the second bullet, $du^{vert}$ is defined with respect to the connection on $P \to B$ along $\tau$.
The virtual dimension of $\mathcal{M}(\vec{x}_0,\vec{x}_1, \vec{x}_2)$ is $|\vec{x}_0|-|\vec{x}_1|-|\vec{x}_2|$.

\begin{rem}
If $x_1, x_2$ satisfies Equation \eqref{eq:zeroarea} and $u$ is a pair of pants asymptotic to $x_1, x_2$ and another loop $x_0'$ in a fibre of $P$, then it will imply that $x_0'$ also satisfies Equation \eqref{eq:zeroarea}. In other words, the condition \eqref{eq:zeroarea} is compatible with the pair of pants product.
\end{rem}

For $(\tau,u) \in \mathcal{M}(\vec{x}_0,\vec{x}_1, \vec{x}_2)$, we define its topological energy as follow (cf. \eqref{eq:topE} and the paragraph afterwards).
For each $e_i \in T_{2,1}$ and each $(s,t)$, we apply the parallel transport from $\tau(e_i(s))$ to $\tau(e_i(0))$ along $\tau(e_i)$ to move $u(\epsilon_i(s,t))$ to the fibre over $\tau(e_0(0))=\tau(e_1(0))=\tau(e_2(0))$. This gives us a map $u^{triv}:\Sigma \to P_{\tau(e_i(0))}$ with asymptotes being the parallel transport of $x_i$ from $P_{c_i}$ to $P_{\tau(e_i(0))}$ along $\tau(e_i)$. The topological energy of $(\tau,u)$ is defined to be the topological energy of $u^{triv}$, which is
\begin{align*}
E(\tau,u)
&=\int_{\Sigma} (du^{vert})^*\omega+ \int_{0}^1 (H_{c_0})_t(x_0(t)) dt - \sum_{i=1}^2 \int_{0}^1 (H_{c_i})_t(x_i(t)) dt
\end{align*}
Note that the first term equals to $\int_{\Sigma} (du^{triv})^*\omega$. Reasoning as in Corollary \ref{c:energyconstant} implies that the topological energy $E: \mathcal{M}(\vec{x}_0,\vec{x}_1, \vec{x}_2) \to \mathbb{R}$ is a locally constant function.


\begin{prop}\label{p:moduliD} The following statements hold: 
\begin{itemize} \item For a generic choice of compatible $J$ that is of contact type, the space $\mathcal{M}(\vec{x}_0,\vec{x}_1,\vec{x}_2)$ is a manifold of the expected dimension for any triple of generators $(\vec{x}_0,\vec{x}_1,\vec{x}_2)$. 

\item Moreover, if the virtual dimension is $\le 1$, then for any $E \in\mathbb{R}$, the moduli space
\[
\mathcal{M}(\vec{x}_0,\vec{x}_1,\vec{x}_2):= \left\{ (\tau,u) \in \mathcal{M}(\vec{x}_0,\vec{x}_1,\vec{x}_2) |E(\tau, u) \le E \right\}
\]
admits a 
Gromov compactification $\ol{\mathcal{M}}_{\le E}(\vec{x}_0,\vec{x}_1,\vec{x}_2)$ making it a compact manifold with boundary.  \end{itemize}
\end{prop}

\begin{proof}
The main point is again to employ a maximum principle to give a priori $C^0$ estimates of the solutions.
The maximum principle we need is in \cite[Remark C.10]{Ritter}.
The $H$ in \cite[Remark C.10]{Ritter} corresponds to $\frac{(H_z)_{\tau(\pi_{\Sigma}(z))}}{\bfs(z)}$, which equals to $r$ (up to adding a constant) over the cylindrical end, and the $\beta$ in \cite[Remark C.10]{Ritter} corresponds to our $\bfs(z) \beta$.
Since $H_z$ is time independent away from the cylindrical end,  $\partial_s \bfs_{H_{\epsilon_i(s,t)}} \le 0$ over all $(s,t)$ and $d (\bfs(z) \beta) \le 0$, \cite[Remark C.10]{Ritter} applies.
\end{proof}

  For any pair of generators, define a product
\begin{align} \vec{x}_1 \cdot \vec{x}_2:= \sum_{\vec{x}_{0}, |\vec{x}_{0}|=|\vec{x}_{1}|+|\vec{x}_{2}|} \sum_{(\tau,u) \in \mathcal{M}(\vec{x}_0,\vec{x}_1, \vec{x}_2)} s(u)q^{E(u)} \vec{x}_0    \end{align}
where $s(u) \in \{-1,1\}$ is the sign of $(\tau,u)$.
We extend it multi-linearly to the Floer complexes.

\begin{prop}\label{p:prodProperty}
Let $a_0> a_1+a_2$.
\begin{itemize} \item The product structure  $CF(P,H_1, \eta) \times CF(P,H_2, \eta) \to CF(P,H_0, \eta)$ descends to a 
product structure on cohomology.
\item The cohomological level product structure is independent of the auxiliary choices made in the construction, namely, cylindrical ends, $H=(H_z)_{z \in \Sigma}$, $J=(J_z)_{z \in \Sigma}$, and the perturbation vector field $X_e$ on $T_{2,1}$.

\item Moreover, it is compatible with continuation maps so it induces a (graded-)commutative and associative product structure on $SH^*(P)$. \end{itemize} 
\end{prop}


\begin{proof}[Sketch of Proof]

Denote the product structure by $$\mu^2:CF(P,H_1, \eta) \times CF(P,H_2, \eta) \to CF(P,H_0, \eta).$$
By considering the boundary of $1$-dimensional strata of $\ol{\mathcal{M}}(\vec{x}_0,\vec{x}_1,\vec{x}_2)$, we get $$\mu^2(d(\vec{x}_1),\vec{x}_2)+\mu^2(\vec{x}_1,d(\vec{x}_2))=d\mu^2(\vec{x}_1,\vec{x}_2)$$ so $\mu^2$ descends to cohomology $[\mu^2]:HF(P,H_1, \eta) \times HF(P,H_2, \eta) \to HF(P,H_0, \eta)$.

To see that $[\mu^2]$ is independent of choice, we apply a cobordism argument.
In other words, we choose a one-parameter family of auxiliary data connecting the two ends (it is possible because the space of auxiliary data is weakly contractible) and form the corresponding parametrized moduli space.
Counting the rigid elements in the parametrized moduli space gives us a chain homotopy we need.
It is important that the  one-parameter family of auxiliary data is chosen such that maximum principle applies as in the proof of Proposition \ref{p:moduliD}.

Once we know that it is independent of auxiliary choices, together with the standard gluing argument we can prove the (graded-)commutativity and associativity because they correspond to interpolating two different choices.

The compatibility with continuation map means that 
\[
\kappa_{a_0,a_0'}([\mu_{a_0,a_1,a_2}](\kappa_{a_1',a_1}(\vec{x}_1),\kappa_{a_2',a_2}(\vec{x}_2)))=[\mu_{a_0',a_1',a_2'}](\vec{x}_1,\vec{x}_2)
\]
where $\kappa_{j,k}:HF(P,j) \to HF(P,k)$ is the continuation map when $j \le k$
and $[\mu_{l,j,k}]:HF(P,j) \times HF(P.k) \to HF(P,l)$ is the product structure when $l>j+k$.
This compatibility follows from the gluing argument and independence of auxiliary choices.
\end{proof}

The product operation can be defined not only for cylindrical Hamiltonians with constant slopes, it can also be defined as long as the maximum principle can be achieved. The analog of Lemma \ref{l:inde_kappa} for the product operation is as follow.

\begin{lem}\label{l:inde_prod}
For $i=0,1,2$, let $H_i$ be a cylindrical Hamiltonians that is compatible with $\eta_i$.
Suppose that $\max {\bfs_{H_1}}+ \max {\bfs_{H_2}} \le \min \bfs_{H_0}$.
After making an appropriate auxiliary choice, we can define a product operation
\[
CF(P,H_1, \eta_1) \times CF(P,H_2, \eta_2) \to CF(P,H_0, \eta_0)
\]
Moreover, the induced map on cohomology
\begin{align}
HF(P,H_1, \eta_1) \times HF(P,H_2, \eta_2) \to HF(P,H_0, \eta_0) \label{eq:prodCoh}
\end{align}
is independent of the auxiliary choice in the construction.
\end{lem}

\subsection{Pull-back}\label{ss:pb}

We want to discuss the functorality between two admissible bundles with the same fibre
$(M,\omega,\theta)$ but different bases $B$ and $B'$ (both $B$ and $B'$ are smooth closed manifolds).
Let $h:B' \to B$ be a smooth map and $p':P' \to B'$ be the fiber bundle obtained by pulling back $p$ along $h$.
Let $\tilde{h}:P' \to P$ be the induced map covering $h$.
Let $\eta$ and $\eta'$ be Morse functions on $B$ and $B'$ respectively
and assume that $H$ and $H'$ are Hamiltonians on $P$ and $P'$ that are compatible with $(\eta,g_{\eta},\nabla)$ and $(\eta',g_{\eta}',\nabla')$ respectively.

Let $(H_s)_{s \in \mathbb{R}_{\ge 0}}$ and $(H'_s)_{s \in \mathbb{R}_{\le 0}}$ be two families of cylindrical Hamiltonians, on $P$ and $P'$ respectively,
such that $H_s=H$ for $s \gg 0$, $H'_s=H'$ for $s \ll 0$ and $\tilde{h}^*H_s=\tilde{h}^*H_0=H'_0=H'_s$ for $s$ in an open neighborhood of $0$.
Choose one-parameter families of $S^1$-dependent fiberwise compatible almost complex structures $(J_s)_{s \in \mathbb{R}_{\ge 0}}$ and $(J'_s)_{s \in \mathbb{R}_{\le 0}}$ that are of contact type on $P$ and $P'$, respectively.
They are chosen such that $J_s$ is independent of $s$ for $s \gg 0$, $J'_s$ is independent of $s$ for $s \ll 0$, 
and $\tilde{h}^*J_s=\tilde{h}^*J_0=J'_0=J'_s$ for $s$ in an open neighborhood of $0$.
Note that $\tilde{h}^*J_s$ is well-defined because $J_s$ is fiberwise and the restriction of $\tilde{h}$ to fibres are isomorphisms.

Given a generator $\vec{x}=(c,x)$ of $H$ and $\vec{x}'=(c',x')$ of $H'$, we can consider the moduli space $\cM_{pb}(\vec{x}',\vec{x},(H'_s)_{s \in \mathbb{R}_{\le 0}}, (H_s)_{s \in \mathbb{R}_{\ge 0}})$ which consists of $(\tau^-,\tau^+,u^-,u^+)$ such that
\begin{itemize}
\item $\tau^-:(-\infty,0] \to B'$ is a gradient trajectory of $\eta'$ with $\lim_s \tau^-(s)=c'$,
\item $\tau^+:[0,\infty) \to B$ is a gradient trajectory of $\eta$ with $\lim_s \tau^+(s)=c$,
\item $\tau^+(0)=h(\tau^-(0))$,
\item $u^-:(-\infty,0] \times S^1  \to P'$ is  covering $\tau^-$ and satisfies the Floer equation  \[(du^{vert}-X_{(H'_s)_{\tau^-(s)}}\otimes dt)^{0,1}_{(J'_s)_{\tau^-(s)}}=0\] 
with $\lim_{s \to -\infty} u^-(s,t)=x'(t)$,
\item $u^+:[0,\infty) \times S^1  \to P$ is covering $\tau^+$ and satisfies the Floer equation 
\[(du^{vert}-X_{(H_s)_{\tau^+(s)}}\otimes dt)^{0,1}_{(J_s)_{\tau^+(s)}}=0\]
 with $\lim_{s \to \infty} u^+(s,t)=x(t)$,
\item $u^+(0,t)=\tilde{h}(u^-(0,t))$,
\end{itemize}


Notice that $\tau(s):=(\tau^-(-s),\tau^+(s))$ for $s \in \mathbb{R}_{\ge 0}$ is a gradient trajectory of the function $B' \times B \to \mathbb{R}$ given by $(b',b) \mapsto \eta(b)-\eta'(b')$ starting at a point on $\graph(h) \subset B' \times B$.
Similarly, let $\pi_P:P' \times P \to P$ and $\pi_{P'}:P' \times P \to P'$ be the projection to the factors, then
$u(s,t):=(u^-(-s,t),u^+(s,t)) \in P' \times P$  for $(s,t) \in \mathbb{R}_{\ge 0} \times S^1$ is a solution to the Floer equation 
\begin{align}\label{eq:Folding}
(du^{vert}-X_{\pi_P^*(H_s)_{\tau^+(s)}-\pi_{P'}^*(H'_{-s})_{\tau^-(-s)}}\otimes dt)^{0,1}_{-(J_{-s})_{\tau^-(-s)} \oplus (J_s)_{\tau^+(s)}}=0
\end{align}
covering $\tau$ with fiberwise Lagrangian boundary condition $\graph(\tilde{h}) \subset P' \times P$ over $\tau(0)$ along $\{0\} \times S^1 \subset \mathbb{R}_{\ge 0} \times S^1$.

From this perspective (so-called `folding' in \cite{WWQuilt}), we can apply the standard Fredholm theory package to study $\cM_{pb}(\vec{x}',\vec{x},(H'_s)_{s \in \mathbb{R}_{\le 0}}, (H_s)_{s \in \mathbb{R}_{\ge 0}})$.

\begin{lem}\label{l:regularpb}
For generic choice of $(J_s)_{s \in \mathbb{R}_{\ge 0}}$ and $(J'_s)_{s \in \mathbb{R}_{\le 0}}$ as above,
every element in the moduli space $\cM_{pb}(\vec{x}',\vec{x},(H'_s)_{s \in \mathbb{R}_{\le 0}}, (H_s)_{s \in \mathbb{R}_{\ge 0}})$ is regular so it is a manifold of the expected dimension.
\end{lem}

\begin{proof}
By our assumptions on $H_s$ and $H'_s$, when $s$ is close to $0$ and $(p,p') \in \graph(\tilde{h})$, we have $H_{s,t}(p)-H_{-s,t}'(p')=0$ because $\tilde{h}^*H_s=\tilde{h}^*H_0=H'_0=H'_s$ for $s$ in an open neighborhood of $0$.
In other words, the Hamiltonian perturbation term $$X_{\pi_P^*H_{s,t}-\pi_{P'}^*H_{-s,t}'}dt|_{T^{vert}\graph(\tilde{h})}$$ is zero along the fiberwise Lagrangian boundary condition $\graph(\tilde{h})$ so it falls into the standard Floer theory package (see  \cite[(8.6)]{SeidelBook}, \cite[Section 5]{WWQuilt}, \cite[(4.70)]{Cazassus}). Therefore, the standard transversality proof applies.
\end{proof}

Next, we want to impose conditions on $(H_s)_{s \in \mathbb{R}_{\ge 0}}$ and $(H'_s)_{s \in \mathbb{R}_{\le 0}}$ to ensure compactness of moduli.
As before, we define the topological energy $E(u^-,u^+)$ of $(\tau^-,\tau^+,u^-,u^+)$ to be 
\begin{align}
E(u^-,u^+)&=E(\tau^-,u^-)+E(\tau^+,u^+)  \label{eq:Eplusminus}\\
E(\tau^-,u^-)&= \int_{(-\infty,0] \times S^1} ((du^-)^{vert})^*\omega + \int_{0}^1 (H_{c'})_t(x'(t)) dt \nonumber\\
E(\tau^+,u^+)&=\int_{[0,\infty) \times S^1} ((du^+)^{vert})^*\omega - \int_{0}^1 (H_{c})_t(x(t)) dt \nonumber
\end{align}
It only depends on the relative homology class of its folding $u:[0,\infty) \times S^1 \to P' \times P$, relative to the fibrewise Lagrangian boundary condition and the asymptotic orbit (cf. \eqref{eq:Folding}).

\begin{lem}\label{l:Gromovpb}
Suppose that $\partial_s \bfs_{(H_s)_{\tau^+(s)}} \le 0$ and $\partial_s \bfs_{(H'_s)_{\tau^-(s)}} \le 0$ for all $s$, and for all gradient trajectory $\tau^+$ and $\tau^-$ of $\eta^+$ and $\eta^-$, respectively.
If the virtual dimension of the moduli space $\cM_{pb}(\vec{x}',\vec{x},(H'_s)_{s \in \mathbb{R}_{\le 0}}, (H_s)_{s \in \mathbb{R}_{\ge 0}})$ is $\le 1$, then for any $E \in\mathbb{R}$, the subspace of solutions with topological energy at most $E$
admits a Gromov compactification to a compact manifold with boundary.
\end{lem}

\begin{proof}
It suffices to prove an a priori $C^0$ bound for the solutions.
By trivializing $P'$ over $\tau^-$ using $\nabla'$, trivializing $P$ over $\tau^+$ using $\nabla$
and identifying fibres of $P'$ to fibres of $P$ using $\tilde{h}$, we 
can regard $u^-$ and $u^+$ as smooth maps to a fibre $(M,\omega,\theta)$, say the fibre $P'_{\tau^-(0)} \simeq P_{\tau^+(0)}$. 
The map $u^- \# u^+: \R \times S^1 \to M$ defined by
\begin{align}
(u^- \# u^+)(s,t):=\left\{
 \begin{array}{ll}
  u^-(s,t)  &\text{ for }s \le 0\\
  u^+(s,t)  &\text{ for }s \ge 0
 \end{array}
\right.
\end{align}
is piecewise smooth and hence Lipschitz. 
Therefore, the first weak derivative exist \cite[5.8.2b]{EvansBook}.

Recall that $\tilde{h}^*H_s=\tilde{h}^*H_0=H'_0=H'_s$ and $\tilde{h}^*J_s=\tilde{h}^*J_0=J'_0=J'_s$ for $s$ in an open neighborhood of $0$.
Therefore, $u^- \# u^+$ satisfies a Floer equation with smooth auxiliary data (both $H$ and $J$).
The existence of the first weak derivative of $u^- \# u^+$ allows us to apply elliptic bootstrapping \cite[Theorem B.4.1]{MSJbook}, so it is actually smooth.
The required $C^0$ estimate now follows from maximum principle \cite[Theorem C.11]{Ritter} (cf. proof of Proposition \ref{p:moduliD}). 
\end{proof}




The virtual dimension of  $\cM_{pb}(\vec{x}',\vec{x},(H'_s)_{s \in \mathbb{R}_{\le 0}}, (H_s)_{s \in \mathbb{R}_{\ge 0}})$ is $|\vec{x}'|-|\vec{x}|$.
To see this, since we are using cohomological convention, the virtual dimension of the moduli of pairs of gradient trajectory $(\tau^-,\tau^+)$ with $h(\tau^-(0))=\tau^+(0)$ is precisely $|c'|-|c|$, where $|\cdot|$ is the Morse index.
Therefore, by trivializing over $\tau^-$ and $\tau^+$, we see that the virtual dimension of $(u^-,u^+)$ for a given $(\tau^-,\tau^+)$ is $|x'|-|x|$.  
Therefore, the virtual dimension is $|c'|-|c| +|x'|-|x|=|\vec{x}'|-|\vec{x}|$.

\begin{lem}\label{l:pull-back}
Suppose that $\partial_s \bfs_{(H_s)_{\tau^+(s)}} \le 0$ and $\partial_s \bfs_{(H'_s)_{\tau^-(s)}} \le 0$ for all $s$, and for all gradient trajectory $\tau^+$ and $\tau^-$ of $\eta^+$ and $\eta^-$, respectively.
The pull-back defined by 
\begin{align}
CF^*(P,H,\eta) &\to CF^*(P',H',\eta')  \nonumber \\
\vec{x} &\mapsto \sum_{\vec{x}', |\vec{x}'|=|\vec{x}|}  \sum_{(u^-,u^+) \in \cM_{pb}(\vec{x}',\vec{x},(H'_s)_{s \in \mathbb{R}_{\le 0}}, (H_s)_{s \in \mathbb{R}_{\ge 0}})} s(u^-,u^+)q^{E(u^-,u^+)}\vec{x}' \label{eq:pullback}
\end{align}
is a degree preserving chain map, where $s(u^-,u^+) \in \{-1,1\}$ is the sign.
\end{lem}



\begin{lem}\label{l:inde_pb}
Suppose that $H' \ge_{P'} \tilde{h}^*H$.
Then there exist $(H'_s)_{s \in \mathbb{R}_{\le 0}}$ and $ (H_s)_{s \in \mathbb{R}_{\ge 0}}$ as above such that 
the assumption of Lemma \ref{l:pull-back} is satisfied (and hence the pull-back map \eqref{eq:pullback} is well-defined).
Moreover, the induced map on cohomology
is independent of the choice of $(H'_s)_{s \in \mathbb{R}_{\le 0}}$, $ (H_s)_{s \in \mathbb{R}_{\ge 0}}$,  $(J_s)_{s \in \mathbb{R}_{\ge 0}}$ and $(J'_s)_{s \in \mathbb{R}_{\le 0}}$.
\end{lem}

\begin{proof}
The arguments are in parallel with the proof of Lemma \ref{l:inde_kappa} so we will omit them.
\end{proof}

The following is the analog of Corollary \ref{c:directSystem}.

\begin{cor}\label{c:natural}
Suppose that $h':B'' \to B'$ is another smooth map between smooth compact manifolds, and $P'' \to B''$ is the pull-back of $P' \to B'$ along $h'$.
Let $\tilde{h}':P'' \to P'$ be the induced map covering $h'$.
Suppose that $H'' \ge_{P''} (\tilde{h}')^*H'$ and $H' \ge_{P'} \tilde{h}^*H$.
Then the composition of the pull-back maps $HF^*(P,H,\eta) \to HF^*(P',H', \eta')$ and $HF^*(P',H',\eta') \to HF^*(P'',H'',\eta'')$ is the pull-back map
$HF^*(P,H,\eta) \to HF^*(P'',H'',\eta'')$.
\end{cor}



\begin{example}\label{e:tautquotient}
Let $h:B' \to B$ (and hence $\tilde{h}:P' \to P$) be an embedding.
Let $(\eta,g_{\eta},\nabla)$ be chosen such that there is no gradient trajectory $\tau$ of $\eta$ connecting two critical points with $\lim_{s \to -\infty} \tau(s) \in B'$ and $\lim_{s \to \infty} \tau(s) \notin B'$, and every gradient trajectory of $\eta$ connecting two critical points in $B'$ is completely lying inside $B'$.
Let $H$ be a cylindrical Hamiltonian that is compatible with $(\eta,g_{\eta},\nabla)$.
Suppose that $(\eta',g_{\eta}',\nabla')$ and $H'$ are the restriction of $(\eta,g_{\eta},\nabla)$ and $H$ to $P'$ and $B'$.
Then $CM(B',\eta',g_{\eta}')$ is naturally a quotient complex of $CM(B,\eta,g_{\eta})$ and the induced map on cohomology coincide with the Morse theoretic pull-back.
Similarly, $CF(P',H',\eta')$ is naturally a quotient complex of $CF(P,H,\eta)$ and the induced map on cohomology coincide with the Floer theoretic pull-back introduced above.
\end{example}

\begin{lem}\label{l:trivialFib}
Let $E_1 \subset E_2 \subset \dots$ be a sequence of smooth compact manifolds.
Let $P_n:=E_n \times M$ and view the projection to the first factor as an admissible bundle over $E_n$.
Then $\varprojlim_n HF^*(P_n,a)=\varprojlim_n (H^*(E_n) \otimes HF^*(M,a))=(\varprojlim_n H^*(E_n)) \otimes HF^*(M,a)$.
\end{lem}

\begin{proof}
Let $H_M$ be a non-degenerate cylindrical Hamiltonian of $M$ with constant slope $a$.
Let $\pi_M:P_n \to M$ be the projection so $\pi_M^*H_M$ is a cylindrical Hamiltonian of $P_n$.
Choose a Morse function $\eta:E_n \to \R$ and take $\nabla$ to be the trivial connection.
Let $H$ be a $C^2$ small perturbation of $\pi_M^*H_M$ in a compact set so that it is a still cylindrical Hamiltonian of $P_n$ with constant slope $a$ and it is compatible with $(\eta,g_{\eta},\nabla)$ such that $H_c=H_M$ for every critical point $c$ of $\eta$.

Suppose that $(\tau,u)$ is a solution contributing to the differential.
Then $\pi_M \circ u$ satisfies the Floer equation
\[
(d(\pi_M \circ u) -X_{H_{\tau(s)}}\otimes dt)^{0,1}_{J_{\tau(s)}}=0
\]
 When the perturbation is sufficiently small (so that $H_{\tau(s)}$ is very close to $H_M$ for all $s$), 
the virtual dimension of $\pi_M \circ u$ must be at least $0$, and when it is $0$, the input and the output are the same orbit.
This identifies $CF^*(P_n,H,\eta)$ with $CM^*(E_n,\eta) \otimes CF^*(M,H_M)$ as a cochain complex so 
$HF^*(P_n,a)=H^*(E_n) \otimes HF^*(M,a)$.

When we consider $P_n \subset P_{n+1}$, we can choose $\eta$ on $E_{n+1}$ such that the critical points of $\eta$ outside $E_n$ forms a subcomplex, and any gradient trajectory connecting two critical points in $E_n$ is contained in $E_n$.
Then the pull-back map $CF( P_{n+1},H,\eta) \to CF( P_{n},H|_{P_n},\eta|_{E_n})$ can be identified with restricting to the quotient complex obtained by killing the generators outside $P_n$.
When the perturbation is sufficiently small, it can in turn be identified with the natural map $CM^*(E_{n+1},\eta) \otimes CF^*(M,H_M) \to CM^*(E_n,\eta|_{E_n}) \otimes CF^*(M,H_M)$ so the result follows.
\end{proof}

\begin{lem}\label{l:pullbackprod}
The pull-back map $SH^*(P) \to SH^*(P')$ is an algebra map.
\end{lem}

\begin{proof}[Sketch of proof]
The proof of Lemma \ref{l:pullbackprod} combines the ideas in Section \ref{s:product_structure} and this section.
The moduli space we need to construct consists of solutions $(\tau,u)$ of a coupled equation.
The domain of $\tau$ is $T_{2,1}$ and the target is $B \sqcup B'$.
There is $r \in \mathbb{R}$ such that $\tau(\{ s \le r\}) \subset B'$ and $\tau(\{ s \ge r\}) \subset B$, where $s:T_{2,1} \to \mathbb{R}$ is the coordinate function on the edges of $T_{2,1}$.
The map $\tau$ is a negative ($X_e$-perturbed) gradient trajectory over any edge and it satisfies the matching condition $h(\tau_{B'}(\{s=r\}))=\tau_B(\{s=r\})$, where $\tau_{B'}:=\tau|_{\{ s \le r\}}$ and $\tau_B:=\tau|_{\{ s \ge r\}}$.
On the other hand, $u: \Sigma=\mathbb{P}^1 \setminus\{0,1,2\} \to P' \sqcup P$ is a lift of $\tau$ that satisfies an appropriate Floer equation.
We can compactify this moduli space.
If we consider the boundary of the compactification of the one dimensional moduli space, then the limit when $r$ goes to $-\infty$ corresponds to taking the product in $SH^*(P)$ and then pulling back to $SH^*(P')$.
Similarly, the limit when $r$ goes to $\infty$ corresponds to pulling back two classes from $SH^*(P)$ to $SH^*(P')$ and then taking the product in $SH^*(P')$.
\end{proof}


\begin{lem}\label{l:rationallyconnectedfibre}
Suppose that $h:B' \to B$ is a fibre bundle with fibres $F$ such that $H^0(F;\mathbb{C})=\mathbb{C}$, and $H^*(F;\mathbb{C})=0$ otherwise.
Then the pull-back map induces an isomorphism $HF^*(P, a) \simeq HF^*(P', a)$.
\end{lem}

\begin{proof}
We pick an admissible base triple $(\eta,g_{\eta},\nabla)$ for $P \to B$.
Let $g'$ be a Riemannian metric on $B'$ such that $h$ is a Riemannian submersion.
Let $N$ be a neighborhood of the critical points of $\eta$ such that $\nabla$ is flat over $N$.
So the bundle $P'$ over $h^{-1}(N)$ can be identified with $P'|_{h^{-1}(N)}=M \times h^{-1}(N)=M \times F \times N$.
Let $\eta':B' \to \mathbb{R}$ be a Morse function obtained by perturbing $h^*\eta$ inside $h^{-1}(N)$. 
We pick $\eta'$ such that $(\eta', g', h^*\nabla)$ is an admissible base triple.

We define a descending filtration on the chain complex $CF^*(P, a)$ and $CF^*(P', a)$ by taking $F^pCF^*(P, a)$ and $F^pCF^*(P', a)$ to be the subcomplex generated by $\vec{x}=(c,x)$ such that $|c| \ge p$, and the subcomplex generated by $\vec{x}'=(c',x')$ such that $|h(c')| \ge p$, respectively.
By our choice of admissible base triples, the pull-back map respects the fibration (i.e. it maps $F^pCF^*(P, a)$ to $F^pCF^*(P', a)$).
It induces a homomorphism of the corresponding spectral sequences.

The $E_1$-page of $CF^*(P, a)$ is given by $(\oplus_{ c \in \critp(\eta)} HF^{*-|c|}(M,a),d_1)$, where the restriction of $d_1$ to the group $HF(M,a)$ over $c$ is the sum of the continuation map over each rigid gradient trajectory with input $c$.
On the other hand, by Lemma \ref{l:trivialFib}, the $E_1$-page of $CF^*(P', a)$ is given by $(\oplus_{ c \in \critp(\eta)} H^*(F) \otimes HF^{*-|c|}(M,a),d_1')$.
By our assumption on $H^*(F;\mathbb{C})$, we have $H^*(F) \otimes HF^{*-|c|}(M,a)=HF^{*-|c|}(M,a)$.
The homomorphism between the components of the $E_1$-pages can be identified with  
$H^*(\text{point}) \otimes HF^{*-|c|}(M,a) \to H^*(F) \otimes HF^{*-|c|}(M,a)$ so it is an isomorphism.
This finishes the proof.
\end{proof}

\subsection{Push-forward}\label{ss:pf}

We now explain how to define the push-forward map in this setup by simply reversing the $s$-direction.  
To clarify,  we follow the notation set up in the first paragraph of the previous section.

This time we let $(H_s)_{s \in \mathbb{R}_{\le 0}}$ and $(H'_s)_{s \in \mathbb{R}_{\ge 0}}$ be two families of cylindrical Hamiltonians, on $P$ and $P'$ respectively,
such that $H_s=H$ for $s \ll 0$, $H'_s=H'$ for $s \gg 0$ and $\tilde{h}^*H_s=\tilde{h}^*H_0=H'_0=H'_s$ for $s$ in an open neighborhood of $0$.
Choose one-parameter families of $S^1$-dependent fiberwise compatible almost complex structures $(J_s)_{s \in \mathbb{R}_{\le 0}}$ and $(J'_s)_{s \in \mathbb{R}_{\ge 0}}$ that are of contact type on $P$ and $P'$, respectively.
They are chosen such that $J_s$ is independent of $s$ for $s \gg 0$, $J'_s$ is independent of $s$ for $s \ll 0$, 
and $\tilde{h}^*J_s=\tilde{h}^*J_0=J'_0=J'_s$ for $s$ in an open neighborhood of $0$.

Given a generator $\vec{x}=(c,x)$ of $H$ and $\vec{x}'=(c',x')$ of $H'$, we can consider the moduli space $\cM_{pf}(\vec{x};\vec{x}',(H_s)_{s \in \mathbb{R}_{\le 0}}, (H'_s)_{s \in \mathbb{R}_{\ge 0}})$ which consists of $(\tau^-,\tau^+,u^-,u^+)$ such that
\begin{itemize}
\item $\tau^-:(-\infty,0] \to B$ is a gradient trajectory of $\eta$ with $\lim_s \tau^-(s)=c$,
\item $\tau^+:[0,\infty) \to B'$ is a gradient trajectory of $\eta'$ with $\lim_s \tau^+(s)=c'$,
\item $h(\tau^+(0))=\tau^-(0)$,
\item $u^-:(-\infty,0] \times S^1  \to P$ is  covering $\tau^-$ and satisfies the Floer equation  \[(du^{vert}-X_{(H_s)_{\tau^-(s)}}\otimes dt)^{0,1}_{(J_s)_{\tau^-(s)}}=0\] 
with $\lim_{s \to -\infty} u^-(s,t)=x(t)$,
\item $u^+:[0,\infty) \times S^1  \to P'$ is covering $\tau^+$ and satisfies the Floer equation 
\[(du^{vert}-X_{(H'_s)_{\tau^+(s)}}\otimes dt)^{0,1}_{(J'_s)_{\tau^+(s)}}=0\]
 with $\lim_{s \to \infty} u^+(s,t)=x'(t)$,
\item $\tilde{h}(u^+(0,t))=u^-(0,t)$,
\end{itemize}

The analog of Lemma \ref{l:regularpb} and \ref{l:Gromovpb} are true.
Moreover, the virtual dimension of the moduli space  $\cM_{pf}(\vec{x};\vec{x}',(H_s)_{s \in \mathbb{R}_{\le 0}}, (H'_s)_{s \in \mathbb{R}_{\ge 0}})$
is $|\vec{x}|-|\vec{x}'|-\dim(B)+\dim(B')$
because the virtual dimension of $(\tau^-,\tau^+)$ is $|c|-\dim(B)+|c'|-\dim(B')$ this time.
It leads to the following analog of Lemma \ref{l:pull-back}, \ref{l:inde_pb} and Corollary \ref{c:natural}.

\begin{lem}\label{l:push-forward}
Suppose that $\partial_s \bfs_{(H'_s)_{\tau^+(s)}} \le 0$ and $\partial_s \bfs_{(H_s)_{\tau^-(s)}} \le 0$ for all $s$, and for all gradient trajectory $\tau^+$ and $\tau^-$ of $\eta^+$ and $\eta^-$, respectively. 
The push-forward defined by
\begin{align}
&CF^*(P',H',\eta') \to CF^{*+\dim(B)-\dim(B')}(P,H,\eta) \nonumber \\
&\vec{x}' \mapsto \sum_{\vec{x}, |\vec{x}|=|\vec{x}'|+\dim(B)-\dim(B')}  \sum_{(u^-,u^+) \in \cM_{pf}(\vec{x},\vec{x}',(H_s)_{s \in \mathbb{R}_{\le 0}}, (H'_s)_{s \in \mathbb{R}_{\ge 0}})} s(u^-,u^+)q^{E(u^-,u^+)} \vec{x} \label{eq:pushforward}
\end{align}
is a chain map of degree $\dim(B)-\dim(B')$, where $E(u^-,u^+)$ is defined by \eqref{eq:Eplusminus} and $s(u^-,u^+)$ is the sign.
\end{lem}

\begin{lem}\label{l:inde_pf}
Suppose that $\tilde{h}^*H \ge_{P'} H'$.
Then there exist $(H_s)_{s \in \mathbb{R}_{\le 0}}$ and $ (H'_s)_{s \in \mathbb{R}_{\ge 0}}$ as above such that 
the assumption of Lemma \ref{l:push-forward} is satisfied (and hence the push-forward map \eqref{eq:pushforward} is well-defined).
Moreover, the induced map on cohomology
is independent of the choice of $(H_s)_{s \in \mathbb{R}_{\le 0}}$, $ (H'_s)_{s \in \mathbb{R}_{\ge 0}}$,  $(J_s)_{s \in \mathbb{R}_{\le 0}}$ and $(J'_s)_{s \in \mathbb{R}_{\ge 0}}$.
\end{lem}

\begin{cor}\label{c:natural_pf}
Suppose that $h':B'' \to B'$ is another smooth map between smooth compact manifolds, and $P'' \to B''$ is the pull-back of $P' \to B'$ along $h'$.
Let $\tilde{h}':P'' \to P'$ be the induced map covering $h'$.
Suppose that $(\tilde{h}')^*H' \ge_{P''} H''$ and $ \tilde{h}^*H \ge_{P'} H'$.
Then the composition of the push-forward maps $HF^*(P'',H'',\eta'') \to HF^{*+\dim(B')-\dim(B'')}(P',H', \eta')$ and $HF^*(P',H',\eta') \to HF^{*+\dim(B)-\dim(B')}(P,H,\eta)$ is the push-forward map
$HF^*(P'',H'',\eta'') \to HF^{*+\dim(B)-\dim(B'')}(P,H,\eta)$.
\end{cor}

The analog of Lemma \ref{l:pullbackprod} is the following.

\begin{lem}\label{l:pfprod}
Denote the pushfoward and pullback by $h_*:SH^*(P') \to SH^{*+\dim(B)-\dim(B')}(P)$ and $h^*:SH^*(P) \to SH^*(P')$, respectively. Then we have $h_*(\vec{x}' h^*(\vec{x}))=h_*(\vec{x}')\vec{x}$ for all $\vec{x}' \in SH^*(P')$ and $\vec{x} \in SH^*(P)$.
\end{lem}

\begin{proof}
It is the same as Lemma \ref{l:pullbackprod} except that one of the two inputs and the output of the elements in the moduli space are swapped. 
\end{proof}



\section{Equivariant Floer theory}\label{s:equivariant}


Let $K$ be a closed subgroup of a compact connected Lie group $G$.\footnote{We mainly consider $K$ to be a maximal torus $T$ or $K=G$ or $K=N(T)$ the normalizer of a maximal torus in the paper.} 
Let $EK_1 \subset EK_2 \subset \dots $ be a $K$-equivariant smooth finite dimensional approximation of a classifying space $EK$.
In particular, each $EK_n$ is a compact smooth manifold, $\cup_n EK_n=EK$ and the connectivity of $EK_n$ goes to infinity as $n$ goes to infinity.
Let $BK_n:=EK_n/K$.
For a manfiold $X$ with a $K$ action, we denote $(X \times EK_n)/K$ by $X_{borel,n}$ and $(X \times EK)/K$ by $X_{borel}$.

Let $P \to B$ be an admissible bundle as in the previous section (see Definition \ref{d:admissible}).
We say that a $K$-action on $P$ is {\it  compatible with the admissible structure} if it covers a $K$-action on $B$ and the induced map $P_{borel,n} \to B_{borel,n}$ is an admissible bundle (so the structure group lies in $G$) for all $n$.
Notice that a choice of a symplectomorphism between $M$ and a reference fibre of $P \to B$ induces a symplectomorphism between $M$ and a reference fibre of $P_{borel,n} \to B_{borel,n}$ that is well-defined up to an element in $G$.
As a result, it induces a well-defined homotopy class of trivialization of the canonical bundle of the reference fibre so a choice of $\mathbb{Z}$-grading on $CF(P)$ determines a choice of $\mathbb{Z}$-grading on $CF(P_{borel,n})$.
For the rest of the paper, we will only use $K$-manifolds $B$ that admit $K$-orientations\footnote{In other words, $B$ admits a $K$-invariant volume form.}.

\begin{example}\label{e:P=M}
If $P=M$ and $B$ is a point, then a $K$-action on $P$ is compatible with the admissible structure if and only if 
$M_{borel,n} \to BK_n$ is admissible for all $n$.  This will hold if and only if the $K$ action on $P(=M)$ factors through the convex $G$ action on $M$.
\end{example}



\begin{defn}\label{d:equivariantFloer}
Let $a \in \R_{> 0} \setminus \Spec(\partial \bar{M},\alpha)$. Suppose that $P$ has a $K$-action that is compatible with the admissible structure.
Then we define
\begin{align}\label{eq:inverselimiT}
HF^*_{K}(P,a):=\varprojlim_n HF^*(P_{borel,n},a)
\end{align}
where 
the inverse limit is defined with respect to pull-back induced by the inclusion $P_{borel,n} \to P_{borel,n+1}$ for all $n$.
\end{defn}

This definition is well-defined thanks to Corollaries \ref{c:directSystem} and \ref{c:natural}.

\begin{lem}
The collection $\{HF^*_{K}(P,a): a \in \R_{> 0} \setminus \Spec(\partial \bar{M},\alpha)\}$ forms a direct system.
\end{lem}

\begin{proof}
It is a consequence of Corollaries \ref{c:directSystem} and \ref{c:natural}.
\end{proof}

\begin{defn} \label{defn:shkm}
The $K$-equivariant symplectic cohomology of $P$ is defined to be
\[
SH^*_{K}(P) :=\varinjlim_{a} HF^*_{K}(P,a) 
\]
\end{defn}

Continuing with Example \ref{e:P=M}, if $P=M$, then $SH^*_{K}(M)$ is the $K$-equivariant symplectic cohomology of $M$.
It is the main object of interest in Section \ref{s:EqSeiMor}.

For our later applications, we mainly consider the case $P=B \times M$ and $K$ is the diagonal action.
However, some results we need use $P$ that are not of product type (e.g. Proposition \ref{p:freeaction} and Lemma \ref{l:indepClassifyingModel}).

\begin{rem} \label{rem:quadratic} An alternative possible definition of equivariant symplectic cohomology of $M$ would be to use families of Hamiltonians of quadratic growth as opposed to the directed system considered in Definition \ref{defn:shkm}.  We caution the reader that this would produce an inequivalent theory from the version considered here --- the difference essentially arising from the algebraic fact that limits and colimits do not commute in general.  This difference is crucial for the set of examples considered in Section \ref{s:Coulomb}.  The ``quadratic version" of symplectic cohomology vanishes identically for these examples,  whereas the definition we employ is interesting enough to allow us to construct Coulomb branches.  Closely related phenomena are considered in \cite{Zhao}.  \end{rem}

\subsection{Useful properties}



\subsubsection{Free actions}

\begin{prop}\label{p:freeaction}
Suppose that $K$ acts on $P$ in a way compatible with the admissible structure such that its action on $B$ is free.
Then $P':=P/K$ is an admissible bundle over $B':=B/K$.
Moreover, for any $a \in \R_{>0} \setminus \Spec(\partial \bar{M},\alpha)$, there is an isomorphism
\begin{align}
HF^*(P', a) \simeq HF^*_{K}(P,a)
\end{align}
\end{prop}

Proposition \ref{p:freeaction} is a consequence of the following:

\begin{prop}\label{p:contractiblefibre}
Let $P' \to B'$ be an admissible bundle. Let $B_1 \subset B_2 \subset \dots$ be a sequence of closed smooth manifolds such that there is a sequence $h_n:B_n \to B'$ making $B_n$ a $EK_n$-bundle over $B'$ and $h_{n+1}|_{B_n}=h_n$.
Let $P_n \to B_n$ be the admissible bundle obtained by pulling back along $h_n$.
Then for any $a \in \mathbb{R}_{>0} \setminus \Spec(\partial \bar{M},\alpha)$,  the pull-back map $HF^*(P',a) \to HF^*(P_n,a)$ induces an isomorphism
\begin{align}\label{eq:EGntrivial}
HF^*(P',a) \simeq  \varprojlim_n HF^*(P_n,a).
\end{align} 
\end{prop}

\begin{proof}[Proof of Proposition \ref{p:freeaction} assuming Proposition \ref{p:contractiblefibre}]
For each $n$, $P_{borel,n}$ is the pull-back of $P' \to B'$ along the $EK_n$-bundle map $h_{n}:B_{borel,n} \to B'$.
The result directly follows from Proposition \ref{p:contractiblefibre}.
\end{proof}

\begin{proof}[Proof of Proposition \ref{p:contractiblefibre}]
The strategy of proof is similar to Lemma \ref{l:rationallyconnectedfibre}.
We can set up the admissible base triples $(\eta_n,g_n,\nabla_n)$ and $(\eta',g',\nabla')$ for $P_n \to B_n$  and $P' \to B'$ as in the proof of Lemma \ref{l:rationallyconnectedfibre} so we get a homomorphism of the respective $E_1$-page
\begin{align}\label{eq:e1p}
\oplus_{c' \in \critp(\eta')} HF^{*-|c'|}(M,a) \to \oplus_{c' \in \critp(\eta')} H^*(EK_n) \otimes HF^{*-|c'|}(M,a)
\end{align}
Recall from the paragraph after Defintion \ref{d:FloerComplex} that we have a $\mathbb{Z}$-grading on the Floer complexes.
With $a$ being fixed, there is $N>0$ such that $HF^{k}(M,a)=0$ for all $|k|>N$ and all $c' \in \critp(\eta')$.
We can take $n$ large enough such that $H^k(EK_n)=0$ for $1 \le k \le 2N+\dim(B')$. 
It ensures that as a spectral sequence, the RHS of \eqref{eq:e1p} has a direct summand given by $H^0(EK_n) \otimes HF^{*-|c'|}(M,a)$.
The map \eqref{eq:e1p} gives an isomorphism between the LHS and this direct summand.
When we take inverse limit over $n$ (cf. Lemma \ref{l:trivialFib}), only this summand survives so it gives the desired isomorphism \eqref{eq:EGntrivial}.
\end{proof}

\subsubsection{Classifying models}

\begin{lem}\label{l:indepClassifyingModel}
The definition $HF^*_{K}(P,a)$ is independent of the choice of the classifying space $EK$ and the finite smooth approximation $\{EK_n\}_n$.
\end{lem}

\begin{proof}
Let $E'K$ be another classifying space with finite smooth approximation $\{E'K_{n'}\}_{n'}$.
For a $K$-manifold $M$, we denote $(M \times EK_n \times E'K_{n'})/K$ by $M_{borel,n,n'}$.
We use the convention that $M_{borel,0,n'}=(M \times E'K_{n'})/K$ and $M_{borel,n,0}=(M \times EK_n)/K$.

In particular, we have the projection map $h_n:B_{borel,n,n'} \to B_{borel,0,n'}$.
The admissible bundle $P_{borel,n,n'} \to B_{borel,n,n'}$ is precisely the pull-back of the admissible bundle $P_{borel,0,n'} \to B_{borel,0,n'}$ along $h_n$.
Therefore, by Proposition \ref{p:contractiblefibre}, we have
\[
HF^*(P_{borel,0,n'},a) \simeq  \varprojlim_n HF^*(P_{borel,n,n'},a)
\]
for any $n'$.
As a result, we have
\[
\varprojlim_{n'} HF^*(P_{borel,0,n'},a) =\varprojlim_{n'} \varprojlim_n HF^*(P_{borel,n,n'},a)=\varprojlim_{n} HF^*(P_{borel,n,0},a)
\]
showing the independence of the choice of the classifying space and finite smooth approximation.
\end{proof}

\subsubsection{Product structure}

Another important property of $SH^*_K(P)$ is that it admits a product structure.
To see this, note that we have the following commutative diagram
\begin{equation}\label{eq:commuProd}
\begin{tikzcd}
HF(P_{n+1}, a_1) \times HF(P_{n+1}, a_2) \arrow{r} \arrow{d}
& HF(P_{n+1}; a_0) \arrow{d}
\\
HF(P_{n}, a_1) \times HF(P_{n}, a_2) \arrow{r} 
& HF(P_{n}; a_0)
\end{tikzcd}
\end{equation}
when $a_0 >a_1+a_2$, 
where the rows are the product operations from Proposition \ref{p:prodProperty}. 
The commutativity of \eqref{eq:commuProd} comes from the compatibility between pull-back maps and the product structure, which can be proved in the same way as Lemma \ref{l:pullbackprod} and we leave the details to readers. 
Therefore, by taking the inverse limit in $n$, we get a product operation
\[
HF_K(P, a_1) \times HF_K(P, a_2) \to HF_K(P; a_0)
\]
The compatibility with continuation maps (see Proposition \ref{p:prodProperty}) imply that we can take direct limit and get a product structure on $SH^*_K(P)$.
The product structure on  $SH^*_K(P)$ is graded-commutative and associative.


\subsubsection{Weyl group action}


Recall that $G$ is a connected compact Lie group which acts on $M$ (see  Section \ref{s:functorialbasic}).
Let $T \subset G$ be a maximal torus and $N(T)$ be the normalizer of $T$.
Let $W=N(T)/T$ be the Weyl group.
Let $P \to B$ be an admissible bundle with a compatible $G$ action.

\begin{lem}\label{l:normal}
For any constant $a \in \R_{>0} \setminus \Spec(\partial \bar{M},\alpha)$, we have
\begin{align}
HF^*_{G}(P,a) \simeq HF^*_{N(T)}(P,a)
\end{align}
\end{lem}

\begin{proof}
For each $n$, the map $(P \times EG_n)/N(T) \to (P \times EG_n)/G$ is a fibre bundle with fibre $G/N(T)$.
Since $H^*(G/N(T);\mathbb{C})=H^*(\text{point};\mathbb{C})$, we can apply Lemma \ref{l:rationallyconnectedfibre} to get the isomorphism $HF^*((P \times EG_n)/G,a) \simeq HF^*((P \times EG_n)/N(T),a)$. The result follows from passing to inverse limit.
\end{proof}

\begin{lem}\label{l:W-inv}
For any constant $a \in  \R_{>0} \setminus \Spec(\partial \bar{M},\alpha)$, we have
\begin{align}
HF^*_{T}(P,a)^{W} \simeq HF^*_{N(T)}(P,a)
\end{align}
\end{lem}

\begin{proof}
The map
\[
(P\times EG_n)/T \to (P \times EG_n)/N(T)
\]
is a covering map which covers the corresponding covering map of their bases $(B \times EG_n)/T \to (B \times EG_n)/N(T)$.
We can use an admissible base triple and Floer data on $(P \times EG_n)/T \to (B \times EG_n)/T$ that come from pulling back an admissible base triple and a compatible Floer data on $(P \times EG_n)/N(T) \to (B \times EG_n)/N(T)$.
The compatibility of Floer data on $(P \times EG_n)/N(T) \to (B \times EG_n)/N(T)$ implies that the pull-back data on  $(P \times EG_n)/T \to (B \times EG_n)/T$  is also compatible.
Moreover, there is a isomorphism of
\begin{align}
CF^*((P \times EG_n)/T,a)^{W} \simeq CF^*((P \times EG_n)/N(T),a)
\end{align}
coming from sending $(c,x) \in CF^*((P \times EG_n)/N(T),a)$ to the sum over its lifts devided by $|W|$.
It induces an isomorphism on the cohomology $H(CF^*((P \times EG_n)/T,a)^{W}) \simeq HF^*((P\times EG_n)/N(T),a)$.
Moreover, since $\Lambda$ has characteristic $0$, we have 
\[
H(CF^*((P \times EG_n)/T,a)^{W})=HF^*((P \times EG_n)/T,a)^{W}
\]
By passing to the inverse limit, we get the result.

\end{proof}



\section{Equivariant Seidel morphism}\label{s:EqSeiMor}

Let $T \subset G$ be a maximal torus.  In this section,  we are going to construct an equivariant Seidel map: \begin{align}
\cS: \hat{H}_*^{T}(\Omega_{poly} G) \otimes  SH^*_{T}(M) \to SH^*_{T}(M).  \label{eq:eqSeidelMap}
\end{align}
  The papers \cite{Sav08}, \cite{ChiHong3},  \cite{GMP22} consider similar constructions in the context of quantum cohomology for closed monotone,  symplectic manifolds, and \cite{CL22} uses an algebraic approach to give a similar construction for quantum $K$-theory of $G/P$.  As explained in the introduction,  a crucial technical difference for $SH^*_{T}(M)$ is to achieve the maximum principle for the cylindrical Hamiltonians obtained by pulling back along a family of loops (cf. the discussion after \eqref{eq:correlation} below).

\subsection{Overview}\label{ss:overview}

As the construction of \eqref{eq:eqSeidelMap} involves a number of steps,  we devote this subsection to providing an overview of our work.  The actual construction is obtained by a smooth finite dimensional approximation of the overview here.


We work with an alternative model for $\hat{H}_*^{T}(\Omega_{poly} G)$,  namely $H_*^{geo,T}(\Omega G)$,  which represents cycles as tuples $(B,\alpha,f)$ such that $B$ is a smooth oriented,  closed manifold with $T$-action,  $\alpha \in H_T^*(B,\mathbb{Z}),$ and $f:B \to \Omega G$ is a smooth,  $T$-equivariant map (see \S \ref{subsection:geomhom}).  These cycles are considered up to an appropriate notion of equivalence.  To construct \eqref{eq:eqSeidelMap},  it suffices to construct,  for any $T$-equivariant smooth map $f:B \to \Omega G$,  a map:  \begin{align} 
\mathcal{S}_{f}:H_T^*(B) \times SH_{T}^*(M) \to SH_{T}^*(M) \label{eq:interSeidel2}
\end{align}
 which respects the various equivalence relations imposed in geometric homology.  Let $$P_B:=(B \times M)_{borel}=(B \times ET \times M)/T.$$
For a cylindrical Hamiltonian $H \in C^{\infty}( S^1 \times P_B) $, let $\tilde{H}: S^1 \times B \times ET \times M \to \mathbb{R}$ be its lift.
We define
\begin{align}\label{eq:Fstar}
\begin{split}
&f^*\tilde{H}: S^1 \times B \times ET \times M \to \mathbb{R}\\
&f^*\tilde{H}_t(b,y,m)=\tilde{H}_t(b,y,(f(b)(t)) \cdot m)-K_{f(b),t}((f(b)(t)) \cdot m) 
\end{split}
\end{align}
where $K_{f(b)}:S^1 \times M \to \mathbb{R}$ is the mean-normalized Hamiltonian (see the definition after \eqref{eq:partialorder}) generating the based loop $f(b) \in \Omega G \subset \Omega \Ham(M)$.
The Hamiltonian $f^*\tilde{H}$ is $T$-invariant and it descends to a cylindrical Hamiltonian, denoted by $f^*H$, on $P_B$.





The key ingredient (beyond those already introduced) needed for \eqref{eq:interSeidel2} is the tautological isomorphism
\begin{align}
\begin{split}
\mathcal{C}: &HF^*(P_B; f^*H)  \simeq HF^*(P_B; H).  \label{eq:correlation}
\end{split}
\end{align}
A priori, even if $H$ is compatible with $\eta:B_{borel} \to \mathbb{R}$,  this doesn't imply that $f^*H$ is compatible with $\eta$.  This is because even if $H$ has a constant slope $\bfs_{H_b} \equiv a \in \R_{\ge 0} \setminus \Spec(\partial \bar{M}, \alpha)$ for all $b \in B_{borel}$, $f^*H$ does not have to have constant slope due to the second term in \eqref{eq:Fstar}.
However,  we will show in \S \ref{ss:GoodClass} and \S \ref{ss:tauto} that there are sufficiently many $H$ such that both $H$ and $f^*H$ are compatible with some $\eta$, and both $\bfs_H$ and $\bfs_{f^*H}$ are bounded functions even though $P_B$ is infinite dimensional.

Let $H'$ be a cylindrical Hamiltonian on $P_B$ with constant slope $a'$ such that 
\[
H' \le_{P_B} f^*H.
\]
We can define a map 
\begin{align}
H_T^*(B) \times HF_{T}^*(M; a') \to HF_{T}^*(P_B; f^*H) \label{eq:1sthalf}
\end{align}
 by composing the acceleration map from $H'$ to $f^*H$ 
with the Floer theoretic pull-back
\begin{align}
&H_T^*(B) \times  HF_{T}^*(M; a') \to HF^*(P_B  ; H') \label{eq:pull}
\end{align}
More precisely, the map \eqref{eq:pull} is the composition of the Floer theoretic pull-back map induced by 
$(B \times ET \times M \times ET)/T \to B_{borel} \times M_{borel}$ and the isomorphism induced by the pull-back along the map $(B \times ET \times M \times ET)/T \to  P_B$.

On the other hand, let $H''$ be a cylindrical equivariant Hamiltonian on $P_B$ with constant slope $a''$ such that 
\[
H \le_{P_B} H''.
\]
We can define a map
\begin{align}\label{eq:2ndhalf}
& HF^*(P_B; H) \to HF_{T}^*(M; a'')
\end{align}
by composing the acceleration map from $H$ to $H''$ and the Floer theoretic push-forward
\begin{align}
& HF^*(P_B; H'') \to HF_{T}^*(M; a'') \label{eq:push}
\end{align}
The map \eqref{eq:push} is induced by pushing forward along the composition map $P_B \to (\Omega G \times M)_{borel} \to M_{borel}$.  By composing \eqref{eq:1sthalf}, \eqref{eq:correlation} and \eqref{eq:2ndhalf}, we get
\begin{align}
H_T^*(B) \times HF_{T}^*(M; a') \to HF_{T}^*(M; a'') \label{eq:interSeidel}
\end{align}

Letting $a'$ and $a''$ go to infinity, \eqref{eq:interSeidel} becomes \eqref{eq:interSeidel2}.

\begin{rem} We emphasize that some of the maps above do not preserve gradings. 
The necessary grading shifts for these maps are discussed in the subsequent subsections.  \end{rem}


\subsection{Tautological isomorphism}\label{ss:Taut}

In this subsection, we are going to prove the isomorphism \eqref{eq:correlation}.

\subsubsection{Identifying Hamiltonian loops}\label{sss:identify}

We start with some basic properties.
The following lemma is standard. A proof can be found in e.g. \cite[Section 2.3]{OhBook1}.

\begin{lem}\label{l:compoHam}
Let $H,H' :S^1 \times M \to \mathbb{R}$ be Hamiltonian functions. Denote their time $t$ Hamiltonian flow by $\phi^t_{H}$ and $\phi^t_{H'}$ respectively. Then
\begin{enumerate}
\item  $(\phi^t_{H} \circ \phi_{H'}^t)_{t \in S^1}$ is generated by $(H\#H')_t(x):=H_t(x)+H'_t( (\phi_H^t)^{-1}(x))$
\item  $(\phi^{-t}_{H})_{t \in S^1}$ is generated by $\ol{H}_t(x):=-H_t(\phi_H^t(x))$
\item  for any $\psi \in \Symp(M)$, $\psi^{-1} \circ \phi^{t}_{H} \circ \psi$ is generated by $H_t(\psi(x))$
\end{enumerate}
\end{lem}

Let $\tilde{H}:S^1 \times B \times ET \times M \to \mathbb{R}$ be a Hamiltonian function.
We define $f^*\tilde{H}:S^1 \times B \times ET \times M \to \mathbb{R}$ by \eqref{eq:Fstar}.
Notice that if $\tilde{H}$ is an $T$-equivariant Hamiltonian function (i.e. $\tilde{H}_{t}(gb,gy,gm)=\tilde{H}_t(b,y,m)$ for all $g \in T$), then so is $f^*\tilde{H}$ because
\begin{align*}
(f^*\tilde{H})_t(gb,gy,gm)&=\tilde{H}_t(gb,gy,(f(gb)(t))gm)-K_{f(gb),t}((f(gb)(t))gm)\\
&=\tilde{H}_t(gb,gy,(gf(b)(t)g^{-1})gm)-K_{gf(b)g^{-1},t}((gf(b)(t)g^{-1})gm)\\
&=\tilde{H}_t(gb,gy,g(f(b)(t))m)-K_{f(b),t}(g^{-1}(gf(b)(t)g^{-1})gm)\\
&=(f^*\tilde{H})_t(b,y,m)
\end{align*}


\begin{lem}[cf. \cite{SeidelRep}, Lemma 2.3]\label{l:pullbackorbit}
Let $\tilde{H}:S^1 \times B \times ET \times M \to \mathbb{R}$ be a Hamiltonian function.
For any $(b,y) \in B \times ET$, if $(x(t))_{t \in S^1}$ is a Hamiltonian orbit of $\tilde{H}(b,y, \cdot)$, then 
 the loop
$(t \mapsto (f(b)(t))^{-1}x(t))_{t \in S^1}$ is a Hamiltonian orbit of
$(f^*\tilde{H})(b,y,\cdot)=\tilde{H}(b,y,f(b)(t)\cdot)-K_{f(b),t}(f(b)(t) \cdot)$.
Moreover, it defines a bijective correspondence between the Hamiltonian orbits.

\end{lem}

\begin{proof}

Let $\gamma(t)=f(b)(t)$.
By Lemma \ref{l:compoHam}(2), $\gamma(t)^{-1}$ is generated by
\begin{align}
(t,m) \mapsto -K_{\gamma,t}(\gamma(t) \cdot m)
\end{align}
By applying Lemma \ref{l:compoHam}(1) to the composition of $\gamma(t)^{-1}$ and $ \phi_{\tilde{H}}^t$, we know that
$\gamma(t)^{-1}\phi_{\tilde{H}}^t$ is generated by
$\tilde{H}(b,t,f(b)(t)\cdot)-K_{f(b),t}(f(b)(t) \cdot)$.
It follows that 
\begin{align}
\phi_{\tilde{H}(b,t,f(b)(t)\cdot)-K_{f(b),t}(f(b)(t) \cdot)}^t x(0)=\gamma(t)^{-1}\phi_{\tilde{H}}^t(x(0)) = \gamma(t)^{-1} x(t)
\end{align}
and hence the loop $\gamma(t)^{-1} x(t)$ is a  Hamiltonian orbit of
$\tilde{H}(b,t,f(b)(t)\cdot)-K_{f(b),t}(f(b)(t) \cdot)$.

The other direction of the bijective correspondence can be proved analogously.
\end{proof}


Let $m \in M$, $b \in B$ and consider the loop $c=(t \mapsto (f(b)(t))^{-1}m)_{t \in S^1}$.
As discussed in the paragraph after Definition \ref{d:FloerComplex}, we have chosen a homotopy class of trivialization of  the canonical bundle of $M$.
Let $\iota(f)$ be the Conley-Zehnder index of the linearization of $(f(b)(t))^{-1}$ along the loop $c$ with respect to the homotopy class of trivialization (more precisely, since $D(f(b)(1))^{-1}$ is the identity, $\iota(f)$ is defined as $2$ times the Maslov index, see \cite[Section 2.4]{SalamonLecture}).
By continuity, $\iota(f)$ is independent of $b \in B$ and the point $m \in M$.

\begin{lem}[cf. \cite{SeidelRep}, Lemma 2.6]\label{l:gradShift}
Let $(x(t))_{t \in S^1}$ be a Hamiltonian orbit of $\tilde{H}(b,y, \cdot)$.
The grading of $(t \mapsto (f(b)(t))^{-1}x(t))_{t \in S^1}$ is the sum of the grading of $(x(t))_{t \in S^1}$ and $\iota(f)$.

\end{lem}

\begin{proof}
This follows from the loop property of the Conley-Zehnder index (see \cite[Section 2.4]{SalamonLecture}).
It can be proved by noticing that $(t \mapsto (f(b)(t))^{-1}x(t))_{t \in S^1}$ is homotopic to the concatenation of $(t \mapsto x(t))_{t \in S^1}$ and $(t \mapsto (f(b)(t))^{-1}x(0))$, so the result follows from the additivity of the Conley-Zehnder index for a path of symplectic matrices and a loop of symplectic matrices.
\end{proof}

\subsubsection{A good class of admissible base triples}\label{ss:GoodClass}

Let $B$ be a finite dimensional closed smooth $T$-manifold.
A continuous map $f:B \to \Omega G$ is called {\it smooth} if the map $B \times S^1 \to G$ given by $(b,t) \mapsto (f(b))(t)$ is smooth.
Let $f:B \to \Omega G$ be a $T$-equivariant smooth map.
Let $P_{B,n}:=(B \times M)_{borel,n}=(B \times ET_n \times M)/T$ and $B_n:=B_{borel,n}$.
We want to consider a good class of admissible base triple $(\eta_n,g_n, \nabla_n)$ of $P_{B,n} \to B_n$ as follows.

Let $\eta_{BT}:BT \to \mathbb{R}$ be a Morse function (i.e a sequence of Morse functions $\eta_{BT,n}:BT_n \to \mathbb{R}$ such that $\eta_{BT,n+1}|_{BT_n}=\eta_{BT,n}$
and $\critp(\eta_{BT,n+1}) \cap BT_{n} = \critp(\eta_{BT,n})$ for all $n$).
Let $g_{BT}$ be a Riemannian metric on $BT$ (i.e. a sequence of Riemannian metrics $g_{BT,n}$ of $BT_n$ that is compatible under embeddings $BT_n \to BT_{n+1}$).
We can choose $\eta_{BT}$ and $g_{BT}$ such that any gradient trajectory of $\eta_{BT,n}$ which starts in $BT_k \subset BT_n$ and tangent to $BT_k$, for some $k<n$, coincides with the gradient trajectory of $\eta_{BT,k}$ in $BT_k$ for all time.
Moreover, there is no gradient trajectory which goes from a critical point of $\eta_{BT,n}$ in $BT_k$ to a critical point
of $\eta_{BT,n}$ in $BT_n \setminus BT_k$ for all $n>k$ (cf. Example \ref{e:tautquotient}).

Let $\nabla_T$ be a connection of the principal $T$-bundle $ET \to BT$ (i.e. a sequence of connections $\nabla_{T,n}$ of $ET_n \to BT_n$ that is compatible under embeddings $ET_n \to ET_{n+1}$)
such that it is flat near critical points of $\eta_{BT}$.
It defines a decomposition of the tangent space $T(ET_n)=T^{vert}ET_n \oplus T^{hor}ET_n$, where $T^{vert}ET_n$ is the kernel of $T(ET_n) \to T(BT_n)$ and $T^{hor}ET_n$ is the horizontal subbundle determined by $\nabla_{T,n}$.
The connection $\nabla_{T,n}$ induces a connection for the associated bundle $B_n \to BT_n$ with the holonomy group in $T$, and the horizontal subbundle $T^{hor}B_n$ is given by the projection of $T^{hor}ET_n$ to $TB_n$ under $T(B \times ET_n) \to TB_n$ (it is well-defined because $g \cdot T^{hor}_xET_n=T^{hor}_{gx} ET_n$ for all $g \in T$).
We choose a Riemannian metric $g_n$ on $B_n$ such that $T^{hor}B_n$ is orthogonal to the vertical subbundle, $(B_n,g_n) \to (BT_n,g_{BT})$ is a Riemannian submersion, and over the region $U \subset BT$ where the connection is flat, $g_n$ is the product of a $T$-invariant metric $g_B$ on $B$ and the metric $g_{BT}|_U$ on $U$.

We use $\nabla_T$ again to induce a connection $\nabla_n$ of $P_{B,n} \to B_n$ by requiring that the horizontal subbundle is the projection of $TB \oplus T^{hor}ET_n$ to $TP_{B,n}$
under
\[
T(B \times ET_n \times M) \to TP_{B,n}
\]

\begin{lem}\label{l:consistent lifts}
Let $c:\mathbb{R}/\mathbb{Z} \to BT_n$ be a loop, $c':[0,1] \to B_n$ be a horizontal lift of $c$ and $c'':[0,1] \to P_{B,n}$ be a horizontal lift of $c'$ with respect to $\nabla_n$.
Let $(B \times M)_{c(0)}$ be the fibre of the natural map $P_{B,n} \to BT_n$ over $c(0)=c(1)$.
Then $c''(0), c''(1) \in (B \times M)_{c(0)}$ lie in the same $T$-orbit, where $T$ acts diagonally.
\end{lem}

\begin{proof}
Consider the connection $\nabla$ of $P_{B,n} \to BT_n$ whose horizontal subbundle is given by the projection of $T^{hor}ET_n$ to $TP_{B_n}$ under $T(B \times ET_n \times M) \to TP_{B,n}$.
The holonomy group for this connection is $T$.
It is easy to check that $c''$ is a horizontal lift of $c$ with respect to $\nabla$ because both the connections of $B_n \to BT_n$ and $P_{B,n} \to B_n$ are induced by $\nabla_T$.
As a result, $c''(0)$ and $c''(1)$ lie in the same $T$-orbit.

\end{proof}

Let $\eta_n'$ be the pull-back of $\eta_{BT,n}$ under the Riemannian submersion $B_n \to BT_n$.
Gradient trajectories of $\eta_n'$ are horizontal lifts of gradient trajectories of $\eta_{BT,n}$.

\begin{lem}\label{l:covariant constant}
Let  $f:B \to \Omega G$ be a $T$-equivariant map and $K_{f,n}:S^1 \times P_{B,n} \to \mathbb{R}$ be the generating Hamiltonian function.
Let $\tau':\mathbb{R} \to B_n$ be a gradient trajectory of $\eta_n'$.
Then $K_{f,n}$ is covariantly constant along $\tau'$ with respect to the connection $\nabla_n$

\end{lem}

\begin{proof}
Let $\tau'':\mathbb{R} \to P_{B,n}$ be a horizontal lift of $\tau'$ with respect to $\nabla_n$.
We need to show that $K_{f,n}(t,\tau''(s))$ is independent of $s$.

Recall that the $T$-invariant lift $K_{F,n}:S^1 \times B \times ET_n \times M$ of $K_{f,n}$ is defined by
\[
K_{F,n}(t,b,x,m):=K_{f(b),t}((f(b)(t))m)
\]
and the $T$-invariance means
\begin{align*}
K_{F,n}(t,gb,gx,gm)&=K_{f(gb),t}((f(gb)(t))gm)\\
&=K_{gf(b)g^{-1},t}(g(f(b)(t))g^{-1}gm)=K_{f(b),t}(g^{-1}g(f(b)(t))g^{-1}gm)\\
&=K_{F,n}(t,b,x,m)
\end{align*}
Note that $K_{F,n}$ also $T$-invariant along the $ET_n$ direction (i.e. $K_{F,n}(t,b,x,m)=K_{F,n}(t,b,gx,m)$), which is an additional feature that is not due to coming from lifting from $P_{B,n}$.
Combining both, we have $K_{F,n}(t,b,x,m)=K_{F,n}(t,gb,x,gm)$ for all $g \in T$.

By Lemma \ref{l:consistent lifts}, $\tau''$ is actually a horizontal lift of a gradient trajectory $\tau:\mathbb{R} \to BT_n$ of $\eta_{BT,n}$.
Since $K_{F,n}$ is independent of $ET_n$ and invariant under the diagonal $T$ action on $B \times M$, together with the fact that the diagonal $T$ action is precisely the holonomy group of $P_{B,n} \to BT_n$, we conclude that $K_{f,n}(t,\tau''(s))$ is independent of $s$.
\end{proof}



Note that $\eta_n'$ is only a Morse-Bott function so $(\eta_n',g_n, \nabla_n)$ is not an admissible base triple.
We want to Morsify $\eta_n'$ to $\eta_n$ so that we get an admissible base triple, and at the same time still have some control on the derivative of $K_{f,n}$ along gradient trajectory of $\eta_n$ with respect to the connection $\nabla_n$.

Before we explain this, we need to introduce the notion of a good pair.

\begin{defn}
Let $f:B \to \Omega G$ be a $T$-equivariant map and $\eta_B':B \to \mathbb{R}$ be a $T$-invariant Morse-Bott function.
We call $(f,\eta_B')$ a good pair if $f$ is a constant function near each connected critical submanifold of $\eta_B'$.

\end{defn}

\begin{lem}\label{l:TeqHomotope}
Given a $T$-equivariant smooth map $f:B \to \Omega G$, we can homotope $f$ to another $T$-equivariant smooth map $f':B \to \Omega G$
such that there is a $T$-invariant Morse-Bott function $\eta_B'$ making $(f',\eta_B')$ a good pair.
\end{lem}

\begin{proof}
Let $\eta_B':B \to \mathbb{R}$ be a $T$-invariant Morse-Bott function such that every connected component of its critical Morse-Bott submanifolds is isomorphic to $T/H$ for some closed subgroup $H$ of $T$ (see \cite[Lemma 4.8]{Wasserman} for its existence and genericity).

Let $C \simeq T/H$ be one of the connected components.
Note that $f|_{T/H}$ factors through $\Omega C_G^0(H) \subset \Omega G$, where $C_G^0(H)$ is the identity component of the centralizer $C_G(H)$ of $H$ in $G$.
Note also that $C_G(T)=T \subset C_G^0(H)$ is a maximal torus in $C_G^0(H)$.
Therefore,  $\pi_1(T)$ surjects onto $\pi_1(C_G^0(H))$.
Let $p \in T/H$ so $f(p) \in \Omega C_G^0(H)$.
We can find a smooth map $F_p:[0,1] \to \Omega C_G^0(H)$ such that $F_p(0)=f(p)$ and $F_p(1) \in \Omega C_G(T)=\Omega T$ because of the surjectivity of $\pi_1(T) \to \pi_1(C_G^0(H))$.
Let $F:[0,1] \times T/H \to \Omega C_G^0(H)$ be the $T$-equivariant map given by $F(s,g \cdot p):=g F_p(s) g^{-1} \in \Omega C_G^0(H)$ for all $g \in T$ and $s \in [0,1]$.
In other words, $F$ is a $T$-equivariant homotopy from $F(0,\cdot)=f$ to the map $F(1,\cdot)$ which lands in $\Omega T$.
In particular, $F(1,g \cdot p)=F(1,p)$ for all $g \in T$.
We want to use this homotopy to homotope $f$ to another $T$-equivariant map such that it is constantly equal to $F(1,p)$ in a neighborhood of $C$.

To do that, let $N$ be a $T$-invariant neighborhood of $C$.
By choosing a $T$-invariant metric on $B$ and using the exponential map, we can identify $N$ as the total space of a $T$-equivariant normal bundle over $C$.
Let $\pi_N:N \to C$ be the projection map.
Let $r:N \to \mathbb{R}_{\ge 0}$ be the distance function from $C$.
Let $\rho_{\epsilon}: \mathbb{R}_{\ge 0} \to \mathbb{R}_{\ge 0}$ be a smooth function such that $\rho_{\epsilon}(s)=0$ near $s=0$ and $\rho_{\epsilon}(s)=s$ when $s > \epsilon>0$.
For $\epsilon>0$ being sufficiently small, we define a $T$-equivariant smooth map $c_{\epsilon}:N \to N$ by $b \mapsto \rho_{\epsilon}(r(b)) b$.
This map collapses a small neighborhood of $C$ to $C$.
We define $f_c:=f \circ c$, which is a $T$-equivariant smooth map that is $T$-equivariant homotopic to $f$.
Moreover, by definition, we know that $f_c$ factors through $\pi_N$ in a small neighborhood $N_C$ of $C$.
Let $\rho^C_{\delta}:[0,\delta] \to [0,1]$ be a smooth function such that $\rho^C_{\delta}(s)=1$ near $s=0$ and  $\rho^C_{\delta}(s)=0$ near $s=\delta$.
For $\delta$ being sufficiently small, we define $f':B \to \Omega G$ by $f'(b):=f_c(b)$ if $r(b) \ge \delta$, and $f'(b):=F(\rho^C_{\delta}(r(b)),\pi_N(b))$ if $r(b) \le \delta$.
Clearly, $f'$ is $T$-equivariant homotopic to $f$, and $f'$ is a constant near $C$.

By applying this procedure to every connected component of the critical submanifolds of $\eta_B'$, we get a good pair $(f',\eta_B')$ as desired.
\end{proof}

Let $(g_n,\nabla_n)$ be as above. We are now ready to introduce $\eta_n$, which is a special form of Morsification of $\eta_n'$ that makes use of $(f,\eta_B')$.

Since $(f,\eta_B')$ is a good pair, we can find a $T$-invariant neighborhood $N_{\eta}$ of $\critp(\eta_B')$ such that $f|_{N_{\eta}}$  is locally constant.
Let $\eta_B: B \to \mathbb{R}$ be a Morsification of $\eta_B'$ such that $\eta_B(b)=\eta_B'(b)$ if $b \notin Int(N_{\eta})$.

Let $U_{BT,n}$ be a small neighborhood of the critical points of $\eta_{BT,n}$ where $\nabla_T$ is flat.
Let $U_n \subset B_n$ be the preimage of $U_{BT,n}$ under $B_n \to BT_n$.
Since $\nabla_n$ and the connection of $B_n \to BT_n$ are both induced by $\nabla_T$, we can identify $U_n$ as $B \times U_{BT,n}$ and trivialize the bundle $P_{B,n}$ over $U_n$ as $B \times M \times U_{BT,n}$.

Let $\eta_n'$ be the Morse-Bott function on $B_n$ above.
The critical submanifolds of $\eta_n'$ are contained in $U_n=B \times U_{BT,n}$, and are of the form $B \times \critp(\eta_{BT,n})$.
Note that, there is a choice in the identification of $B$ coming from the choice of trivialization, and it is canonical only up to an element in $T$.
Our argument below works for any such choice.
We Morsify $\eta_n'$ inside $U_n$ by adding a function of the form $\epsilon \chi(u) \eta_B(b)$ for $(b,u) \in B \times U_{BT,n}$, where $0< \epsilon \ll 1$, $\chi$ is a bump function and $\eta_B$ is the Morse function constructed above.
We can choose $\chi:U_{BT,n} \to [0,1]$ such that over each connected component of $U_{BT,n}$, it ony depends on the distance from the corresponding critical point of $\eta_{BT,n} $, the only critical values are $\{0,1\}$ and $\chi^{-1}(1)=\critp(\eta_{BT,n})$.  
Denote the Morse funciton $\eta_n'+\epsilon \chi(u) \eta_B(b)$ by $\eta_n$.
In the next subsection, we will use the admissible base triple $(\eta_n,g_n,\nabla_n)$.

We end this subsection with the following observation.

\begin{lem}\label{l:equivariantS}
Let $(f,\eta_B')$ be a good pair and $\eta_B,g_B$ be defined as above.
For any constant $c \in \mathbb{R}$, there is a $T$-invariant function $\bfs \in C^{\infty}(S^1 \times B)$ such that $\bfs > c$, $\bfs(b) \notin \Spec(\partial \bar{M},\alpha)$ for any $b \in \critp(\eta_B)$, and 
for any gradient trajectory $\tau_B$ of $\eta_B$ with respect to $g_B$, we have
\begin{align}\label{eq:grad_on_B}
\frac{d}{ds} \bfs (t,\tau_B(s)) \le - \max_{m\in \partial \bar{M}} \left|\frac{d}{ds}\bfs_{K_{f(\tau_B(s))},t}(m) \right|
\end{align}
for all $t \in S^1$.
\end{lem}

\begin{proof}

To see that $\bfs$ exists even though $\eta_B$ is {\it not} $T$-invariant, we argue as follows.
Recall that $(f,\eta_B')$ is a good pair and $f|_{N_{\eta}}$ is locally constant.
It implies that $K_{f(b),t}(m)$ is locally independent of $b \in N_{\eta}$.
Therefore, the RHS of \eqref{eq:grad_on_B} is $0$ when $\tau_B(s) \in N_{\eta}$.
When $\tau_B(s) \notin N_{\eta}$, $\tau_B(s)$ is also a gradient trajectory of $\eta_B'$, and  $\eta_B'$ is $T$-invariant.
Therefore, it suffices to find a $T$-invariant function $\bfs \in C^{\infty}(S^1 \times B)$ such that it is locally constant in $N_{\eta}$, $\bfs >  c$, $\bfs(b) \notin \Spec(\partial \bar{M},\alpha)$ for any $b \in \critp(\eta_B)$, and 
for any gradient trajectory $\tau_B'$ of $\eta_B'$, we have
\begin{align}\label{eq:grad_on_B'}
\frac{d}{ds} \bfs (t,\tau_B'(s)) \le - \max_{m\in \partial \bar{M}} \left|\frac{d}{ds}\bfs_{K_{f(\tau_B'(s))},t}(m) \right|
\end{align}
but this is easy because $\eta_B'$ is $T$-invariant, $N_{\eta}$ is $T$-invariant and $g_B$ is also $T$-invariant.
\end{proof}

\subsubsection{A good class of admissible Hamiltonians}\label{ss:tauto}

Let $f^*0$ be the Hamiltonian function on $P_B$ given by \eqref{eq:Fstar} with $H=0$.
It is descended from $-K_{F,n}(t,b,y,m)$ and $K_{F,n}(t,b,y,m):=K_{f(b),t}((f(b)(t))m)$ is independent of $ET_n$.
Therefore, for any $b \in B_{borel}$, there is $b'  \in B$ such that
\[
\max_{(t,m) \in S^1 \times \partial \bar{M}}|\bfs_{(f^*0)_b}(m)|=\max_{(t,m) \in S^1 \times \partial \bar{M}}|\bfs_{K_{f(b'),t}}(m)|
\]
Since $ S^1 \times B \times \partial \bar{M}$ is compact, the value 
\begin{align}\label{eq:cfK}
c_{f,K}:=\sup_{b \in B_{borel}} \max_{(t,m) \in S^1 \times \partial \bar{M}}|\bfs_{(f^*0)_b}(m)|=\sup_{b' \in B} \max_{(t,m) \in S^1 \times \partial \bar{M}}|\bfs_{K_{f(b'),t}}(m)|
\end{align}
 is finite.

Given a $T$-invariant function $\bfs \in C^{\infty}(S^1 \times B)$, we can pull it back to get a $T$-invariant function on $S^1 \times B \times ET$. It descends to a function on $S^1 \times B_{borel}$ which we denote by $\bfs_{borel}$.

\begin{prop}\label{p:accelerator}
Let $(f,\eta_B')$ be a good pair.
Then there is a constant $C_f >0$ depending only on $(f,\eta_B')$ with the following property.
There is an admissible base triple $(\eta,g,\nabla)$ of $P_{B} \to B_{borel}$ (i.e. a sequence of admissible base triples $(\eta_n,g_n, \nabla_n)$ of $P_{B,n} \to B_n$)
such that for any $c \in \mathbb{R}_{>0}$, there is a cylindrical Hamiltonian $A_{f} \in C^{\infty}(S^1 \times P_B)$ (i.e a sequence of cylindrical Hamiltonian $A_{f,n}$ on $P_{B,n}$ such that $A_{f,n+1}|_{P_{B,n}}=A_{f,n}$) 
and a $T$-invariant function $\bfs$ obtained by Lemma \ref{l:equivariantS}
such that
\begin{itemize}
\item $\bfs >c_{f,K}+c$, and
\item $\bfs_{(A_f)_b}=\bfs_{borel}(b)$ for all $b \in B_{borel}$, and
\item both  $f^*A_{f,n}$ and $A_{f,n}$ are compatible with $(\eta_n,g_n,\nabla_n)$ for all $n$, and 
\item $c \le \bfs_{A_f}, \bfs_{f^*A_f} \le c+C_f$.
\end{itemize}
\end{prop}

\begin{proof}

Recall that $g_B$ is a $T$-invariant metric on $B$ such that $g_n|_{U_n}=g_B+g_{BT}|_{U_{BT,n}}$. 
By Lemma \ref{l:equivariantS}, we can pick a  $T$-invariant function $\bfs \in C^{\infty}(S^1 \times B)$ such that $\bfs > c_{f,K}+c$, and 
for any gradient trajectory $\tau_B$ of $\eta_B$ with respect to $g_B$, we have
\begin{align}\label{eq:grad_on_B}
\frac{d}{ds} \bfs (t,\tau_B(s)) \le - \max_{m\in \partial \bar{M}} \left|\frac{d}{ds}\bfs_{K_{f(\tau_B(s))},t}(m) \right|
\end{align}
for all $t \in S^1$. 

Since  $\bfs $ is $T$-invariant, it induces a function $\bfs_{borel} \in C^{\infty}(S^1 \times B_{borel})$.
We claim that
\begin{align}\label{eq:grad_on_B_borel}
\frac{d}{ds}\bfs_{borel}(t,\tau(s)) \le - \max_{m\in \partial \bar{M}} \left |\frac{d}{ds}\bfs_{(f^*0)_{\tau(s)}}(t,m) \right|
\end{align}
 for all gradient trajectory $\tau:\mathbb{R} \to B_n$ of $\eta_{n}$, for all $t \in S^1$, and for all $n$.
To see why, note that outside $U_n$, $\tau$ is also a gradient trajectory of $\eta_n'$ because $\epsilon \chi(u) \eta_B(b)$ is supported inside $U_n$.
Moreover, gradient trajectories of $\eta_n'$ are horizontal lifts of gradient trajectory of $\eta_{BT,n}$.
By Lemma \ref{l:covariant constant}, the RHS is $0$. The LHS is also $0$ because $\bfs_{borel}$ is also covariantly constant.
Inside $U_n$, we use the trivialization $U_n=B \times U_{BT,n}$ so for $(b,u) \in B \times U_{BT,n}$, we have 
\begin{align}\label{eq:sborel}
\bfs_{borel}(t,(b,u))=\bfs(t,b)
\end{align}
and
\begin{align}\label{eq:sf0}
\bfs_{(f^*0)_{(b,u)}}(t,m)=\bfs_{K_{f(b),t}}(m)
\end{align}
Both of them are independent of $u \in U_{BT,n}$.
On the other hand, we have
\[
d\eta_n=d\eta_n'(u)+\epsilon \chi(u) d\eta_B(b)+\epsilon d \chi(u) \eta_B(b)
\]
The gradient $grad(\eta_n)$ of $\eta_n$ is therefore lying inside $\epsilon \chi(u) grad(\eta_B(b)) + TU_{BT,n}$.
By \eqref{eq:grad_on_B}, \eqref{eq:sborel} and \eqref{eq:sf0}, we conclude that \eqref{eq:grad_on_B_borel} is true.

Now, we would like to choose $A_f$ such that $\bfs_{(A_f)_b}(m)=\bfs_{borel}(b)$ for all  $ m \in \partial \bar{M}$.
This choice of $A_f$ trivially satisfies the second bullet of the proposition.

For the last bullet, note that 
\[
c_{f,K}+c \le \min_B \bfs  \le \bfs_{A_f}  \le \max_B \bfs 
\]
Therefore, we have 
\[
c \le \bfs_{f^*A_f}  \le \max_B \bfs +c_{f,K}
\]
So  the last bullet is satisfied with $C_f=\max_B \bfs +c_{f,K}-c$.
Recall that, $c_{f,K}$ is a constant which only depends on $(f,\eta_B')$.
Therefore, in order to find a $C_f$ which depends only on $(f,\eta_B')$ we need to give a uniform upper bound for $\max_B \bfs -c$ that is independent of $c$, and only depends on $(f,\eta_B')$.
This uniform upper bound exists because, if $\bfs > c_{f,K}+c$ is obtained from Lemma \ref{l:equivariantS}, then for any other $c'>0$, 
$\bfs-c+c' >c_{f,K}+c'$ almost satisfy all the conditions in Lemma \ref{l:equivariantS} except possibly that 
$(\bfs-c+c' )(b)$ might lie in the spectrum $\Spec(\partial \bar{M},\alpha)$ for some $b \in \critp(\eta_B)$.
Therefore, one can $T$-equivariantly perturb $\bfs-c+c' $ to get a $T$-invariant function $\bfs'>c_{f,K}+c'$ which satisfy all the conditions in Lemma \ref{l:equivariantS}.
The perturbation can be as small as we want so $\max_B \bfs -c$ and $\max_B \bfs' -c'$ can be made as close to each other as we want.
This implies that there is a uniform upper bound for $\max_B \bfs -c$ that is independent of $c$, which in turn implies that we can find a $C_f$ we want which depends only on $(f,\eta_B')$.

It remains to show that it also fulfills the third bullet.
Recall the condition of compatiblity with $(\eta_n,g_n,\nabla_n)$ from Definition \ref{d:eqgen}.
Our choice of $\bfs$ is intensionally chosen such that the third bullet of Definition \ref{d:eqgen} is satisfied
for both $A_f$ and $f^*A_f$  (see \eqref{eq:grad_on_B_borel}).

The first bullet of Definition \ref{d:eqgen} can be achieved because $\bfs:S^1 \times B \to \mathbb{R}$ is chosen such that over the critical points of $\eta_{B}$, $\bfs$ does not lie in the spectrum of the Reeb flow, so a generic choice of cylindrical $A_f$ with $\bfs_{(A_f)_b}=\bfs_{borel}(b)$ will be non-degenerate over the critical points of $\eta$.
When  $\bfs_{(A_f)_b}$ is non-degenerate over  the critical points of $\eta$, so is true for $\bfs_{(f^*A_f)_b}$.

The second bullet of Definition \ref{d:eqgen} can be achieved because $\bfs_{borel}$ is locally constant in $N_{\eta} \times U_{BT,n} \subset B \times U_{BT,n}=U_n$, so we can pick $A_{f,n}$ which is locally constant over $N_{\eta} \times U_{BT,n}$, which contains all the critical points of $\eta_n$.
Since $f$ is also locally constant over $N_{\eta}$, $f^*A_{f,n}$ will also be locally constant over $N_{\eta} \times U_{BT,n}$.

The last bullet of Definition \ref{d:eqgen}  can be achieved by generic fibrewise-compactly supported perturbation of $A_f$ outside $U_n$. It will imply that so is true for $f^*A_f$.
\end{proof}


We have the following tautological isomorphism. 

\begin{prop}\label{p:correlationproof}
Let $B$ be a finite dimensional closed smooth $T$-manifold.
Let $H$ be a cylindrical Hamiltonian of $P_B$ such that both $f^*H$ and $H$ are compatible with $(\eta_n,g_n,\nabla_n)$.
Then we have an isomorphism 
\begin{align}\label{eq:correlation2}
\cC_f:HF^*(P_{B,n}, f^*H)  \simeq HF^{*-\iota(f)}(P_{B,n};H)
\end{align}
By taking inverse limit in $n$, we get the  isomorphism \eqref{eq:correlation}.
\end{prop}

\begin{rem}
In fact, Proposition \ref{p:correlationproof} is true even if $H$ is compatible with $(\eta_n,g_n,\nabla_n)$ but $f^*H$ is not.
The only difference is that when $f^*H$ is not compatible with  $(\eta_n,g_n,\nabla_n)$, we cannot apply Lemma \ref{l:Floercohomology} to get the well-definedness of $HF(P_{B,n}, f^*H)$.
But as we shall see from the proof below, the identification $CF(P_{B,n}, f^*H)  \simeq CF(P_{B,n};H)$ proves that $HF(P_{B,n}, f^*H)$ is well-defined in a posterori.
\end{rem}

\begin{proof}[Proof of Proposition \ref{p:correlationproof}]


The strategy to prove \eqref{eq:correlation2} is to establish a cochain level isomorphism with respect to appropriate auxilary data.
Let $\tilde{\Phi}_t: B \times ET \times M \to B \times ET \times M$ be
$\tilde{\Phi}_t(b,y,m)=(b,y,f(b)(t)m)$.
It satisfies
\begin{align}
\tilde{\Phi}_{t}(gb,gy,gm)=(gb,gy,f(gb)(t)gm)=(gb,gy,gf(b)(t)m)
\end{align}
so it descends to a map $\Phi_t:P_B \to P_B$ for all $t \in S^1$.

By Lemma \ref{l:pullbackorbit}, $\Phi_t$ provides a bijection between the generators of 
$CF(P_{B,n}, f^*H)$ and $CF(P_{B,n}, H)$.

Let $(J_b)_{b \in B_{n}}$ be a generic $S^1$-dependent fibrewise almost complex structure on $P_{B,n}$ that is compatible with the fibrewise symplectic form and is of contact type.
The differential in the Floer complex  $CF(P_{B,n};H)$ is defined by counting solutions of coupled equations with gradient trajectories on $BT_n$ with respect to $(\eta_n,g_n)$, and Floer solution liftings of the gradient trajectories with respect to $(\nabla_n,H,J)$.

If we use $(\Phi_t^*J_b)_{b \in B_{n}}$ to define the Floer differential of $CF(P_{B,n}, f^*H)$, then we will have a bijective correspondence between the solutions contributing to the differential of $CF(P_{B,n}, f^*H)$ and solutions contributing to the differential of $CF(P_{B,n};H)$.
More precisely, if $(\tau(s),u(s,t))$ is a solution contributing to the differential of $CF(P_{B,n}, H)$, then 
$(\tau(s),(\Phi_t)^{-1} \circ u(s,t))$ will be a solution contributing to the differential of $CF(P_{B,n}, f^*H)$.

Moreover, by trivializing $P_B$ along $\tau$ and applying the Lemma \ref{l:topEnergy} below, we see that the correspondence preserves the topological energy $E(\tau, (\Phi_t)^{-1} \circ u(s,t))=E(\tau, u(s,t))$.
Therefore, we get the isomorphism (see Lemma \ref{l:gradShift} for the grading shift).
\end{proof}

\begin{lem}\label{l:topEnergy}
Let $u:\mathbb{R} \times S^1 \to M$ be a Floer solution with respect to $H=(H_{s,t})_{(s,t) \in \mathbb{R} \times S^1} \in C^{\infty}(\mathbb{R} \times S^1 \times M)$. 
Let $\gamma \in \Omega G$ and $v(s,t)=\gamma(t)^{-1}u(s,t)$.
Then $E(v)=E(u)$.
\end{lem}

\begin{proof}
We have
\begin{align*}
\partial_sv(s,t)&=(D\gamma(t))^{-1} \partial_su(s,t) \\
\partial_tv(s,t)&=(D\gamma(t))^{-1} \partial_tu(s,t) + X_{\gamma(t)^{-1}}(v(s,t)) 
\end{align*}
where $X_{\gamma(t)^{-1}}$ is the Hamiltonian vector field generating $\gamma^{-1}$ at time $t$ (recall from Lemma \ref{l:compoHam}(2) the generating Hamiltonian of $\gamma^{-1}$).
Therefore, 
\begin{align*}
&\int_{\mathbb{R} \times S^1} v^* \omega\\
=&\int_{-\infty}^{\infty} \int_0^1 \omega((D\gamma(t))^{-1} \partial_su(s,t), (D\gamma(t))^{-1} \partial_tu(s,t) + X_{\gamma(t)^{-1}}(v(s,t) ) dt ds\\
=&\int_{\mathbb{R} \times S^1} u^* \omega + \int_{-\infty}^{\infty} \int_0^1 \omega(\partial_su(s,t), (D\gamma(t)) X_{\gamma(t)^{-1}}(v(s,t))dt ds\\
=&\int_{\mathbb{R} \times S^1} u^* \omega + \int_{-\infty}^{\infty} \int_0^1 d(-K_{\gamma,t}(u(s,t)))(\partial_su(s,t))dt ds\\
=&\int_{\mathbb{R} \times S^1} u^* \omega + \int_0^1 K_{\gamma,t}(u(-\infty,t)) dt - \int_0^1 K_{\gamma,t}(u(\infty,t)) dt
\end{align*}
where $u(\pm \infty,t):= \lim_{s \to \pm \infty} u(s,t)$.
As a result,
\begin{align*}
&E(v)\\
=&\int_{\mathbb{R} \times S^1} v^* \omega +  \int_0^1 (H_{-\infty,t}-K_{\gamma,t})(\gamma(t)v(-\infty,t)) dt - \int_0^1 (H_{\infty,t}-K_{\gamma,t})(\gamma(t)v(\infty,t)) dt\\
=& \int_{\mathbb{R} \times S^1} u^* \omega +  \int_0^1 H_{-\infty,t}(u(-\infty,t)) dt - \int_0^1 H_{\infty,t}(u(\infty,t)) dt\\
=&E(u)
\end{align*}

\end{proof}

The isomorphism \eqref{eq:correlation2} is compatible with the product structure as follows.

\begin{prop}\label{p:correlationprod}
Let $B$ be a finite dimensional closed smooth $T$-manifold.
Let $H$ be a cylindrical Hamiltonian of $P_B$ such that both $f^*H$ and $H$ are compatible with $(\eta_n,g_n,\nabla_n)$.
Let $H'$ be another cylindrical Hamiltonian of $P_B$ that is of constant slope $\bfs'$.
Let $H''$ be a cylindrical Hamiltonian of $P_B$ such that both $f^*H''$ and $H''$ are compatible with $(\eta_n,g_n,\nabla_n)$,
 $\max \bfs_H +\max \bfs_{H'} \le \min \bfs_{H''}$ and $\max \bfs_{f^*H} +\max \bfs_{H'} \le \min \bfs_{f^*H''}$.
Then the following diagram commutes
\begin{equation*}
\begin{tikzcd}
HF(P_{B,n}, f^*H) \times HF(P_{B,n},H') \arrow{r} \arrow{d}
& HF(P_{B,n}; f^*H'') \arrow{d}
\\
HF(P_{B,n}, H) \times HF(P_{B,n},H') \arrow{r} 
& HF(P_{B,n}; H'')
\end{tikzcd}
\end{equation*}
where the horizontal maps are pair-of-pants products, and vertical maps are continuation maps.
\end{prop}

\begin{proof}[Sketch of proof]
Similar to the proof of Proposition \ref{p:correlationproof}, we want to construct $\Phi=(\Phi_{z,t}: P_B \to P_B)_{z \in \Sigma, t \in S^1}$ to give a chain level identification of the product operation on the first row and on the second row. 
More precisely, over the first positive cylindrical end and the negative cylindrical end of $\Sigma$, we choose 
$\tilde{\Phi}_{z,t}: B \times ET \times M \to B \times ET \times M$ to be
$\tilde{\Phi}_{\epsilon_i(s,t),t}(b,y,m)=(b,y,f(b)(t)m)$.
Over the second positive cylindrical end, we choose $\tilde{\Phi}_{\epsilon_2(s,t),t}(b,y,m)=(b,y,m)$.
There is no homotopical obstruction to extends the definition of $\tilde{\Phi}_{z,t}$ over the complement of cylindrical ends and at the same time having $\tilde{\Phi}_{z,t}$ to be $T$-equivariant for all $z,t$.
For example, we can start with a $s$-invariant $\tilde{\Phi}$ over a cylinder and then isotope it such that it is the identity in some $s$-invariant neighborhood. Adding a puncture in this  $s$-invariant neighborhood and introducing a cylindrical end $\epsilon_2$ near the puncture will give what we want.
The $T$-equivariant $\tilde{\Phi}$ descends to $\Phi$.
We can choose all the auxiliary data in the definition of the two product operations to be intertwined by $\Phi$.


\end{proof}

\begin{defn}
Let $(f,\eta_B')$ be a good pair.
For each $a \in \mathbb{R}_{>0} \setminus \Spec(\partial \bar{M},\alpha)$, we can apply Proposition \ref{p:accelerator} to find an $A_f$ with the listed properties.
Together with Proposition \ref{p:correlationproof} and the continuation maps \ref{l:inde_kappa}, we define the following composition map
\begin{align}\label{eq:CsfA}
\cC_{a,f,A_f}:&HF(P_B, a) \to HF(P_B, f^*A_f)  \simeq HF(P_B; A_{f}) \to HF(P_B; a +C_f+\epsilon_a)
\end{align}
where $\epsilon_a \ge 0$ is a small number such that $a +C_f+\epsilon_a \in \mathbb{R}_{>0} \setminus \Spec(\partial \bar{M},\alpha)$.
\end{defn}
In Equation \ref{eq:CsfA}, the first map exists because $\min f^*A_f \ge a$. The second map is the isomorphism in  Proposition \ref{p:correlationproof}. The last map exists because $\max A_f \le  a +C_{f}+\epsilon_a$.

\subsubsection{Independence of choice and functorality}

\begin{lem}\label{l:indepChoiceC}
Let $(f_0, \eta_{B,0}')$ and $(f_1, \eta_{B,1}')$ be good pairs.
Suppose that $c:\R \times B \to \Omega G$ is a $T$-equivariant smooth map such that $c(s,\cdot)=f_0$ for $s \ll 0$ and $c(s,\cdot)=f_1$ for $s \gg 0$.
For $i=0,1$, let $C_{f_i}$ be as in Proposition \ref{p:accelerator}.
For $a_1 \in \mathbb{R}_{>0} \setminus \Spec(\partial \bar{M},\alpha)$ and $a_1+C_{f_1} +\epsilon_{a_1}<a_2 \in \mathbb{R}_{>0} \setminus \Spec(\partial \bar{M},\alpha)$, we let
$A_{f_1}$ and $A_{f_0}$ be as in Proposition \ref{p:accelerator} such that
\[
a_1 \le \bfs_{A_{f_1}}, \bfs_{f_1^*A_{f_1}} \le a_1+C_{f_1}, \; a_2 \le \bfs_{A_{f_0}}, \bfs_{f_0^*A_{f_0}} \le a_2+C_{f_0}
\]
Then we have
\[
\kappa_{a_1+C_{f_1}+\epsilon_{a_1}, a_2+C_{f_0}+\epsilon_{a_2}} \circ \cC_{a_1,f_1,A_{f_1}}=  \cC_{a_2,f_0,A_{f_0}} \circ \kappa_{a_1, a_2}
\]
\end{lem}

\begin{proof}
The content of the lemma is that the following diagram commutes
\begin{equation*}
\begin{tikzcd}
HF(P_B, a_1) \arrow{r} \arrow{d}
& HF(P_B, f_1^*A_{f_1}) \arrow{r} \arrow{d}
& HF(P_B; A_{f_1}) \arrow{r} \arrow{d}
& HF(P_B; a_1+C_{f_1}+\epsilon_{a_1}) \arrow{d}
\\
HF(P_B, a_2) \arrow{r} 
& HF(P_B, f_0^*A_{f_0}) \arrow{r}
& HF(P_B; A_{f_0}) \arrow{r} 
& HF(P_B; a_2+C_{f_0}+\epsilon_{a_2}) 
\end{tikzcd}
\end{equation*}
In the first row, we are using admissible base triple coming from applying Proposition \ref{p:accelerator} to $(f_1,\eta_{B,1}')$ and in the second row, we are using admissible base triple coming from applying Proposition \ref{p:accelerator} to $(f_0,\eta_{B,0}')$.

Note that, we have
\[
f_1^*A_{f_1} \le_{P_B} f_0^*A_{f_0} \text{ and } A_{f_1} \le_{P_B} A_{f_0}
\]
so the second the third vertical maps are well-defined (see Lemma \ref{l:inde_kappa}).

The left and right squares commute because of the functorality of compatible Hamiltonians (Lemma \ref{l:inde_kappa} and Corollary \ref{c:directSystem}).
The middle square commutes because the two vertical maps can be tautologically identified by identifying the respective moduli spaces as in the proof of Proposition \ref{p:correlationproof}.
\end{proof}

As a consequence of Lemma \ref{l:indepChoiceC}, we obtain the following corollary

\begin{cor}
Suppose that $(f,\eta_B')$ is a good pair.
Then the maps $\{\cC_{c,f,A_f}\}_{c \in \mathbb{R}}$ induce a well-defined map
\begin{align}
\cC_{f}: SH^*(P_B) \to SH^*(P_B). \label{eq:directLim}
\end{align}
which is independent of the choice of admissible base triples, $A_f$ and the representative in the $T$-equivariant homotopy class of $f$.
\end{cor}

\begin{proof}
Recall that $SH^*(P_B):=\varinjlim_{a} HF(P_B, a)$.
To define the map, we take $f_0=f_1=f$ in Lemma \ref{l:indepChoiceC}.
We can take a sequence $(a_k)_{k \in \mathbb{N}}$ such that $a_{k+1}> a_k +C_f+\epsilon_{a_k}$ for all $k$.
Applying Lemma \ref{l:indepChoiceC} to $a_k$ for all $k$, we get a sequence of commutative diagrams which induces a map $\cC_{f}$ on the direct limit \eqref{eq:directLim}.
To see that $\cC_f$ is independent of choices, we use the fact that Lemma \ref{l:indepChoiceC} is true for any $a_2 >a_1+C_{f}+\epsilon_{a_1}$.
Therefore, we can apply Lemma \ref{l:indepChoiceC} to compare $\cC_f$ with respect to two different choices. The commutativity of the diagrams and the functorality of continuation maps  (Lemma \ref{l:inde_kappa} and Corollary \ref{c:directSystem}) imply that $\cC_f$ is independent of choices as claimed.

\end{proof}

\begin{lem}\label{l:cfprod}
For any $z_1,z_2 \in SH^*(P_B)$, we have
$\cC_f(z_1z_2)=\cC_f(z_1)z_2$.
\end{lem}

\begin{proof}
It follows from Proposition \ref{p:correlationprod} and the compatibility with continuation maps.
\end{proof}

\subsection{Construction over a $T$-equivariant cycle}\label{ss:Gcycle}

The goal of this section is to complete the construction of \eqref{eq:interSeidel2} following the outline in Section \ref{ss:overview}.
The main ingredients are Proposition \ref{p:correlationproof} and the pull-back/push-forward maps in Section \ref{ss:pb}, \ref{ss:pf}.

\subsubsection{Pull-back}

To define the pull-back \eqref{eq:pull}, the commutative diagram to keep in mind is
\begin{equation*}
\begin{tikzcd}
P_{B,n}  \arrow[swap]{d}
& (B \times ET_n \times M \times ET)/T \arrow{l} \arrow{r} \arrow[swap]{d}
& (B \times ET_n)/T \times (M \times ET)/T =B_{borel,n} \times M_{borel} \arrow[swap]{d}
\\
B_n 
&  (B \times ET_n \times ET)/T \arrow{l} \arrow{r}
&(B \times ET_n)/T \times ET/T=B_{borel,n}  \times  BT
\end{tikzcd}
\end{equation*}

We first consider the square on the LHS.
By applying Proposition \ref{p:freeaction} to the admissible bundle $B \times ET_n \times M$ with respect to the diagonal $T$ action, we have
\begin{align}
HF(P_{B,n}, a) \simeq HF^*_{T}(B \times ET_n \times M,a) \label{eq:pullbackreal1}
\end{align}

Now we consider the square on the RHS. For each $n' \in \mathbb{N}$, we have a commutative diagram 
\begin{equation}\label{eq:compullback}
\begin{tikzcd}
 (B \times ET_n \times M \times ET_{n'})/T  \arrow{r} \arrow[swap]{d}
&B_{borel,n} \times M_{borel,n'} \arrow[swap]{d}
\\
 (B \times ET_{n} \times M \times ET_{n'+1})/T  \arrow{r} 
& B_{borel,n} \times M_{borel,n'+1}
\end{tikzcd}
\end{equation}
Let $\bH_n:=HF((B \times ET_n \times M )_{borel,n'}, a)$.
By the commutative diagram \eqref{eq:compullback} and the functorality of pullback (see Corollary \ref{c:natural}), the following diagram commutes
\begin{equation*}
\begin{tikzcd}
\bH_{n'}
& H^*(B_{borel,n}) \times HF(M_{borel,n'},a) \arrow{l} 
\\
\bH_{n'+1} \arrow[swap]{u}
& H^*(B_{borel,n}) \times HF(M_{borel,n'+1},a)  \arrow{l} \arrow[swap]{u}
\end{tikzcd}
\end{equation*}
Passing to the inverse limit over $n'$, we get a map
\begin{align}
H^*(B_{borel,n}) \otimes HF^*_{T}(M, a ) \to HF^*_{T}(B \times ET_n \times M,a) \label{eq:pullbackreal2}
\end{align}
By composing it with \eqref{eq:pullbackreal1} and passing to the inverse limit over $n$, we get a map
\begin{align}
H^*_T(B) \otimes HF^*_{T}(M, a ) \to HF(P_B, a)  \label{eq:pullbackreal3}
\end{align}
Composing it with the acceleration map, it gives the desired map in Equation \eqref{eq:pull}.

Another commutative diagram which can be easily verified is
\begin{equation}\label{eq:pullbackreal4}
\begin{tikzcd}
H^*_T(B) \otimes HF^*_{T}(M, a_0 )  \arrow{r} \arrow[swap]{d}
&HF(P_B, a_0) \arrow[swap]{d}
\\
H^*_T(B) \otimes HF^*_{T}(M, a_1 ) \arrow{r} 
& HF(P_B, a_1)
\end{tikzcd}
\end{equation}
when $a_0 \le a_1$.
As a result, we have
\begin{align}
H^*_T(B) \otimes SH^*_{T}(M) \to SH^*(P_B)  \label{eq:pullbackreal5}
\end{align}


\subsubsection{Pushforward}

On the other hand, as explained in the overview, we consider the Floer theoretic push-forward associated to the following diagram
\begin{equation*}
\begin{tikzcd}
 P_{B,n}=(B \times ET_n \times M)/T  \arrow{r} \arrow[swap]{d}
& (\Omega G \times M \times ET_n)/T  \arrow{r} \arrow[swap]{d}
&  (M \times ET_n)/T=M_{borel,n} \arrow[swap]{d}
\\
B_n \arrow{r} 
&(\Omega G \times ET_{n})/T  \arrow{r} 
& BT_n
\end{tikzcd}
\end{equation*}
Note that both the square on the left and right are fibre product squares so the outer square is also a fibre product square and Lemma \ref{l:push-forward} gives us the push-forward map
\begin{align}\label{eq:new3}
HF(P_{B,n}; a) \to HF(M_{borel,n}, a).
\end{align}

\begin{lem}\label{l:pushwforwardcommute}
The following diagram commutes
\begin{equation*}
\begin{tikzcd}
HF(P_{B,n+1}; a) \arrow{r} \arrow{d}
& HF(M_{borel,n+1}, a) \arrow{d} 
\\
HF(P_{B,n}; a) \arrow{r}
& HF(M_{borel,n}, a) 
\end{tikzcd}
\end{equation*}
\end{lem}

\begin{proof}
The first vertical map is induced by a chain level short exact sequence
\[
0 \to S_P  \to CF(P_{B,n+1}; a) \to CF(P_{B,n}; a)  \to 0
\]
where $S_P$ is the subcomplex of $CF(P_{B,n+1}; a)$ generated by generators of $CF(P_{B,n+1}; a)$ that do not lie inside $P_{B,n}$.
Similarly, the other vertical map is induced by a chain level short exact sequence
\[
0 \to S_M  \to CF(M_{borel,n+1}; a) \to CF(M_{borel,n}; a)  \to 0
\]
where $S_M$ is the subcomplex of $CF(M_{borel,n+1}; a)$ generated by generators of $CF(M_{borel,n+1}; a)$ that does not lie inside $M_{borel,n}$.
Therefore, to prove the result, it suffices to show that the chain level pushforward maps sit inside the following commutative diagram
\begin{equation*}
\begin{tikzcd}
0 \arrow{r} &S_P  \arrow{r} \arrow{d} &CF(P_{B,n+1}; a)  \arrow{r} \arrow{d} &  CF(P_{B,n}; a)  \arrow{d}  \arrow{r} &0\\
0 \arrow{r} &S_M  \arrow{r}  & CF(M_{borel,n+1}; a)  \arrow{r}  & CF(M_{borel,n}; a)  \arrow{r} &0
\end{tikzcd}
\end{equation*}
In other words, we need two results.
The first result is that the image of $S_P$ is in $S_M$.
It in turn follows easily from the fact that for our choice of Morse data on the bases, if the positive asymptote (i.e. input) of the underlying gradient trajectory of a solution which contributes to the pushforward map is lying in $B_{n+1} \setminus B_{n}$, then its negative asymptote (i.e. output) is lying in $BT_{n+1} \setminus BT_n$.

The second result we need is that if the input is in $P_{B,n}$ and the output is not in $M_{borel,n}$, then the moduli space of solutions computed with respect to $P_{B,n+1} \to M_{borel,n+1}$ and with respect to $P_{B,n} \to M_{borel,n}$ are the same.
This is true because the underlying gradient trajectory of a solution with input in $B_{n}$ and output in $BT_n$ is entirely contained in $B_n$ and $BT_n$.
It allows us to choose auxiliary data such that the two moduli spaces completely coincide.

\end{proof}

By Lemma \ref{l:pushwforwardcommute}, we can take to the inverse limit with respect to $n$, then we obtain the pushforward map
\begin{align}
HF(P_{B}; a) \to  HF(M_{borel}, a) \label{eq:pfreal1}
\end{align}

Another commutative diagram which can be easily verified is
\begin{equation}\label{eq:pfreal2}
\begin{tikzcd}
HF(P_{B}; a_0 )  \arrow{r} \arrow[swap]{d}
&HF(M_{borel}, a_0) \arrow[swap]{d}
\\
HF(P_{B}; a_1 ) \arrow{r} 
&HF(M_{borel}, a_1)
\end{tikzcd}
\end{equation}
when $a_0 \le a_1$.
As a result, we have
\begin{align}
SH^*(P_B) \to SH^{*-\dim(B)}_T(M)  \label{eq:pfreal3}
\end{align}
The map $\cS_f$ (see \eqref{eq:interSeidel2}) is defined to be the composition of  \eqref{eq:pullbackreal5}, \eqref{eq:directLim} and \eqref{eq:pfreal3}.

\begin{lem}\label{l:moduleType}
Let $(f,\eta_B')$ be a good pair.
For any $a_0, a_1 \in SH^*_T(M)$ and $\alpha \in \hat{H}_T^* (B)$, we have
\begin{align}
\cS_f(\alpha,a_0) a_1=\cS_f(\alpha,a_0a_1)
\end{align}
\end{lem}

\begin{proof}
Denote the pull-back  \eqref{eq:pullbackreal5} and push-forward map \eqref{eq:pfreal3} by
$h^*$ and $h_*$, respectively.
We have
\begin{align*}
\cS_f(\alpha,a_0a_1) &=h_* (\cC_f (h^*(\alpha,a_0a_1)))  \\
&=h_* (\cC_f (h^*(\alpha,a_0) h^*(e_B,a_1)))\\
&=h_* (\cC_f (h^*(\alpha,a_0)) h^*(e_B,a_1))\\
&=h_* (\cC_f (h^*(\alpha,a_0))) a_1 \\
&=\cS_f(\alpha,a_0) a_1
\end{align*}
 The first and last equality follow from the definition of $\cS_f$. The second equality uses that pull-back maps are algebra maps (Lemma \ref{l:pullbackprod}). The third equality comes from Lemma \ref{l:cfprod}. The fourth equality uses Lemma \ref{l:pfprod} and the fact that $h^*(e_B,a_1)$ equals to the pull-back of $a_1$ to $SH^*(P_B)$. 
\end{proof}

\subsection{Equivariant geometric homology and Semi-infinite homology} \label{subsection:geomhom}
In this subsection,  we first construct an equivariant version of the geometric homology theory developed in \cite{BaumDouglas,  Jak98}.  Then we recall the definition of 
semi-infinite homology and prove a comparison result between these two theories.  Closely related results have been obtained in \cite{Baum, GuoMathai}.  Throughout this section we let $K$ denote a compact, connected Lie group.    \vskip 5 pt
\subsubsection{Equivariant geometric homology} 
 \vskip 5 pt

 \begin{defn} We say that a smooth,  closed manifold $B$  is $K$-oriented if it is oriented and equipped with a smooth (necessarily orientation-preserving) $K$-action.   We let $\cman$ denote the collection of closed $K$-oriented manifolds.  \end{defn} 

Let $N$ be a topological space with $K$-action.  We consider triples $(B,\alpha,f)$ such that $B \in \cman$,  $\alpha \in H_K^*(B,\mathbb{Z}),$ and $f:B \to N$ is a continuous,  $K$-equivariant map.
Two such triples $(B,\alpha,f)$ and $(B',\alpha',f')$ are called {\it equivalent} if there is an orientation preserving, $K$-equivariant diffeomorphism
$\phi:B' \to B$ such that $\phi^* \alpha=\alpha'$ and $f'=f \circ \phi$. Let $Z^{geo,K}(N)$ denote the free abelian group generated by equivalence classes of triples.

\begin{defn}\label{d:geomHo} 
  The geometric homology $H^{geo,K}_*(N)$ is the quotient of $Z^{geo,K}(N)$ by the following relations:
\begin{enumerate}
\item If $B=B_1 \sqcup B_2$, then $(B,\alpha,f)=(B_1, \alpha|_{B_1},f|_{B_1})+(B_2, \alpha|_{B_2},f|_{B_2})$
\item $(B,\alpha_1+\alpha_2,f)=(B,\alpha_1,f)+(B,\alpha_2,f).$
\item Suppose there is a $K$-equivariant map $F: B \times [0,1] \to N$ then \begin{align*} (B, \alpha,  F_{B \times \lbrace 0 \rbrace}) = (B, \alpha,  F_{B \times \lbrace1 \rbrace}). \end{align*} 
\item Let $B_0,  B_1$ be in $\cman$ and $j: B_0 \to B_1$ be a smooth, equivariant embedding.  Let $f: B_1 \to N$ be an equivariant map and $\alpha \in H_K^{i}(B_0)$.  Then
\begin{align} \label{eq:Gysinrel} (B_0, \alpha,  f_{|B_0})=(B_1, j_{!}(\alpha),  f),  \end{align} where $j_{!}$ denotes the equivariant Gysin morphism (\cite{Arabia} or  \eqref{eq:GysinPD} below).  
\end{enumerate} 
\end{defn} 

We denote the class of $(B,\alpha,f)$ by $[B,\alpha,f]$.

\begin{rem} The reader may notice that compared to the axioms of \cite{Jak98},  we have weakened Axiom (3) but strengthened Axiom (4).  This is because the product cobordisms from Axiom (3) are especially easy to incorporate into our setup (though we could also allow for general cobordisms without too much difficulty).  On the other hand,  concerning Axiom (4),  the special sphere bundle suspensions from \cite{Jak98} do not play a special role in our Floer theoretic setup.     \end{rem}

Geometric homology is covariantly functorial.  It is also homotopy invariant:

\begin{lem} \label{lem:homotopy} If $h^0,h^1: N_0 \to N_1$ are $K$-homotopic (equivariant) maps between $K$-spaces,  then  $h^0_*= h^1_*:  H^{geo,K}_*(N_0) \to H^{geo,K}_*(N_1).$ \end{lem} 
\begin{proof} Let $h^t: N_0 \times [0,1] \to N_1$ denote the homotopy between $h^0$ and $h^1$.  Suppose we have a tuple $(B,\alpha, f)$ representing a class in $H^{geo,K}_*(N_0)$.  We then push it forward by $h^0, h^1$ to obtain $Z_0:= (B, \alpha, h^0 \circ f),  Z_1:= (B, \alpha,  h^1 \circ f)$ representing classes in $H^{geo,K}_*(N_1).$ Then $h^t \circ f: B \times [0,1]$ gives a homotopy $h^0 \circ f$ and $h^1 \circ f$.  Thus by Axiom (3) above of geometric homology,  $Z_0=Z_1 \in H^{geo,K}_*(N_1).$   \end{proof} 

There is also an Eilenberg-Zilber map: \begin{align}\label{eq:prod} \mathcal{EZ}^{geo}:  H_*^{geo,K}(N_0) \otimes H_*^{geo,K}(N_1) \to   H_*^{geo,  K\times K}(N_0 \times N_1)\\ [B_0, \alpha_0,f_0]\otimes [B_1, \alpha_1,f_1] \to [B_0 \times B_1,  \pi_0^*(\alpha_0) \cup \pi_1^*(\alpha_1),  f_0 \times f_1] \nonumber \end{align}
where $\pi_i:B_0 \times B_1 \to B_i$ is the projection map for $i=0,1$.

Finally,  we let $T \subset K$ be a maximal torus inside $K$ and let $W=N(T)/T$ denote the Weyl group of $T.$ Given a $K$- space $N$,  we want to construct a Weyl-group action on $H_*^{geo,T}(N)$.  To this end,  let $\rho_N: T \times N \to N$ denote the induced $T$-action on $N$ and let \begin{align}  \phi_{w}: T \to T \\ t \to  wtw^{-1} \nonumber \end{align}  denote the homomorphism associated to some Weyl group element.    Then $ \rho_N \circ \phi_w^{-1}$ (shorthand for $\rho_N \circ (\phi_w^{-1} \times 1_N)$) gives a new $T$-action on $N$.  Given $(B,  \alpha,  f) $ representing a class in $H^{geo,T}(N)$,  we set \begin{align} \phi_w^*[B,  \alpha,  f]=[B_{\rho_{B} \circ \phi_{w}^{-1}},  \phi_w^*(\alpha),  f] \in H^{geo,T}(N_{\rho_{N} \circ \phi_{w}^{-1}}). \end{align} Here $B_{\rho_{B} \circ \phi_{w}^{-1}}$ denotes $B$ with the $w$-twisted group action (so that $f$ remains equivariant) and $\phi_w^*(\alpha) \in H_T^*(B_{\rho_{B} \circ \phi_{w}^{-1}}) $ is the pull-back of $\alpha$ along
\begin{align}\label{eq:pullbackalpha} (B \times ET)/T \to (B \times ET) /T \\
(b,y) \to (b,w^{-1}y). \nonumber \end{align} 
  
Next,  we note that multiplication by (a representative of) a Weyl group element $w$ induces a map \begin{align*} w_*: H^{geo,T}(N_{\rho_{N} \circ \phi_{w}^{-1}}) \to H^{geo,T}(N) \end{align*} We define the Weyl group action on $H^{geo,T}(N)$ by \begin{align} \label{eq:weylgeom1} w_* \circ \phi_w^*: H^{geo,T}(N) \to H^{geo,T}(N). \end{align}  At the level of cycles we have: \begin{align} \label{eq:weylgeom2} w \cdot (B,\alpha,f)=(B_{\rho_{B} \circ \phi_{w}^{-1}}, \phi_w^*(\alpha),w \circ f).  \end{align} 

 \vskip 5 pt

\subsubsection{Semi-infinite homology} Unless otherwise stated,  all homology groups in this subsection are taken with $\mathbb{Z}$-coefficients.  \begin{defn} A closed $K$-equivariant subspace $N$ of a finite dimensional $K$-representation $\mathbb{V}$ is said to be a $K$-ENR if there is an equivariant open set $U \subset \mathbb{V}$ which equivariantly retracts onto $N$.  We let $\enr$ denote the collection of compact $K$-ENRs.   \end{defn}  

\begin{rem} Using the main result of \cite{Gleason}, the property of being a compact $K$-ENR can be shown to be independent of the representation $\mathbb{V}$ in the sense that if $N'$ is a $K$-equivariant subspace of another $K$-representation $\mathbb{V}'$ and $N$ is $K$-equivariantly homeomorphic to $N'$, then $N'$ also admits an  an equivariant open set $U' \subset \mathbb{V}'$ which equivariantly retracts onto $N'$ (see \cite[pg 143-144]{Jaworowski} for discussion and a sharper result).  By combing the results of \cite[Theorem 2.1]{Jaworowski} and  \cite{Mostow},  any finite $K$-CW complex is a compact $K$-ENR.   \end{rem} 

For any $N \in \enr$,   we let $\hat{H}_*^{K}(N)$ denote the semi-infinite homology  \cite{MR1614555, MR1786481,Graham}.  Let us recall the definitions, which are given in terms of finite dimensional approximations to classifying spaces of compact Lie groups.  A unitary embedding $K \hookrightarrow U(n)$ gives a model $EK$ for the classifying space of $K$.   These models come with finite dimensional approximations $EK_n$ and we let $BK_n:= EK_{n}/K$ to be the corresponding finite dimensional approximations of $BK$.  

Let $N_{borel,n}:= (N \times EK_n) /K$ be the finite dimensional approximations to the Borel mixing space $N_{borel} := (N \times EK) /K.  $  The semi-infinite homology of $N$ is defined to be the limit 
\begin{align} \label{eq: inverselimit}   
  \hat{H}_{\ast}^K(N) := \varprojlim_n H_ {\operatorname{dim} BK_n+ \ast}(N_{borel,n}) 
\end{align} 
where the maps in the inverse system are defined using certain Gysin pull-back maps (\cite[\S  3]{GMP22}). 

\begin{rem} One can also consider homology groups with coefficients in a field $\mathbb{K}$.  We will denote these by $\hat{H}_*^K(N,\mathbb{K}).$ \end{rem}

 In any fixed degree,  the inverse limit \eqref{eq: inverselimit} stabilizes (see e.g.  \cite[page 79]{MR1786481} or \cite[page 601]{Graham}).  This implies that,  if $N_0, N_1 \in \enr$,  there are Eilenberg-Zilber maps: \begin{align} \label{eq:Kunneth2} \mathcal{EZ}: \hat{H}_*^K(N_0) \otimes \hat{H}_*^K(N_1) \to   \hat{H}_*^{K\times K}(N_0 \times N_1) \end{align} 
which come from their non-equivariant counterparts on the finite dimensional approximations.   

Let us briefly discuss the case of a closed,  compact $K$-oriented manifold $B$.  In this case,   $\hat{H}_{\ast}^K(B)$ carries an equivariant fundamental class $[B]_K \in  \hat{H}_{\operatorname{dim}(B)}^K(B)$ and there is a Poincar\'e duality isomorphism: 
\begin{align} PD_B:H_K^{\operatorname{dim}(B)-\ast}(B) \cong \hat{H}_ {\ast}^K(B)   .  \end{align}
 If $j: B_0 \to B_1$ is an equivariant map between closed, compact K-oriented manifolds,  the Gysin map on equivariant cohomology  $j_!: H_K^i(B^0) \to H_K^{i+r}(B^1)$ ($r=\operatorname{dim}(B^1)-\operatorname{dim}(B^0)$) can be expressed as: \begin{align} \label{eq:GysinPD} \alpha \to PD_{B_1}^{-1}( j_*(PD_{B_0}( \alpha))). \end{align} 

We will need to consider certain infinite dimensional spaces as well.  The precise definition is as follows: 
\begin{defn} A topological space $N$ will be called an infinite type $K$-ENR if $N$ is an ascending union of compact K-ENRs.  Namely,  we require that $N= \operatorname{colim}_i N_i$ where $N_i \in \enr$ and \begin{align} N_1 \subset N_2 \subset  \cdots N_i \subset \cdots \end{align} are all closed $K$-equivariant inclusions.  We let $\ienr$ denote the collection of all infinite type $K$-ENRs.  \end{defn}

For any $N \in \ienr$,  we define \begin{align} \hat{H}_*^{K}(N):= \operatorname{colim}_i \hat{H}_*^{K}(N_i).\end{align}

Note that, this definition is independent of the choice of $N_i$ because if $N_i' \in \enr$ is another sequnece such that 
 $N= \operatorname{colim}_i N_i'$ and  \begin{align} N_1' \subset N_2' \subset  \cdots N_i' \subset \cdots \end{align} are all closed $K$-equivariant inclusions, then for each $i$, by the compactness of $N_i$ we can find a $N_j'$ such that $N_i \subset N_j'$.
Therefore, we have an induced map  $\operatorname{colim}_i \hat{H}_*^{K}(N_i) \to  \operatorname{colim}_i \hat{H}_*^{K}(N_i')$ and an induced map in the opposite direction.
These two maps are inverse of each other by the naturality of induced maps by inclusions.

The Eilenberg-Zilber map \eqref{eq:Kunneth2} also exists when both spaces are infinite type $K$-ENRs. \vskip 5 pt
\subsubsection{Comparison}

There is a natural transformation  $\psi: H^{geo,K}_*(-) \to \hat{H}_*^{K}(-)$ which for any $N \in \enr$ is  defined by
\begin{align} \label{eq: mapnoneq}
\psi: H^{geo,K}_*(N) \to \hat{H}_*^{K}(N)\\
\psi: [B, \alpha,f] \mapsto f_*(PD_B^{-1}(\alpha)) \nonumber.
\end{align}

It follows easily from the definition that this is well-defined.  For example, to see that this map respects \eqref{eq:Gysinrel},  we have that \begin{align} f_*(PD_{B_1}(j_{!}(\alpha)) = f_*\circ PD_{B_1} \circ PD_{B_1}^{-1}\circ j_*\circ PD_{B_0}(\alpha)=f_{|B_0}(PD_{B_{0}}(\alpha)) \end{align} 

Also note that $H_{*}^{geo,K}(N)$ also commutes with direct limits and thus we can define a comparison map $\psi$ extending \eqref{eq: mapnoneq} for any $N \in \ienr.$ We will need the following observation in our main argument. 

\begin{lem}[Lemma 2.1 of \cite{Baum}]  \label{lem:retract} Let $N$ be in $\enr$,  then $N$ is an equivariant retract of $M \in \cman$.   \end{lem} 
\begin{proof} Let $U$ be an equivariant open subset of the representation $\mathbb{V}$ which retracts on to $N$.  Fix a $K$-invariant metric on $\mathbb{V}$ and let $f: U \to \mathbb{R}$ be a $C^0$-close $K$-equivariant smoothing of the distance function to $N$.  A suitable level set of $f$ bounds a smooth compact manifold with boundary $\bar{U}$ which retracts onto $N$. By doubling $\bar{U}$,  we obtain our manifold $M$.    \end{proof} 

The main result of this section is the following: 
\begin{thm} \label{thm:maincomp} For any $N \in \ienr$,  the canonical map \begin{align} \psi: H_{*}^{geo,K}(N) \to \hat{H}_*^K(N) \end{align} is an isomorphism.  
\end{thm} \begin{proof} It clearly suffices to consider the case $N \in \enr$ which we do for the rest of the argument.  \vskip 5 pt

\emph{Surjectivity:} In view of  Lemma \ref{lem:retract},  it suffices to prove surjectivity when $N \in \cman$.  In this case,  consider cycles of the form $(N, \alpha, id).$ These surject onto homology by Poincare duality.   \vskip 10 pt

\emph{Injectivity:}  Again we can suppose that $N$ is a closed oriented smooth manifold.  Now fix a $(B, \alpha,f)$ representing a class of $H^{geo,K}_*(N)$ so that $f_*(PD_B(\alpha))=0.$  By taking the representation $\mathbb{V}$ in the proof of Lemma \ref{lem:retract} sufficiently large (i.e.  to include sufficiently many irreducible representations with high enough multiplicity), \cite[Corollary 1.10]{Wasserman} implies that we may embed $N$ as a retract of a higher dimensional manifold $N'$ so that $f:B \to N'$ is $K$-homotopic to an equivariant embedding $f': B \to N'$.  In view of \eqref{eq:Gysinrel},  the fact that $f'_*(PD_B(\alpha))=0$ implies that \begin{align} [B,\alpha,f]=[B,\alpha,f']=[N',  0,  id]=0 \in H^{geo,K}(N'). \end{align} It follows that $[B,\alpha,f]=0 \in H^{geo,K}(N). $  \end{proof} 

\begin{lem} \label{lem:properties} The comparison map satisfies the following additional properties \begin{enumerate}  \item The comparison map is compatible with Eilenberg-Zilber maps.  More precisely,  the following diagram commutes \[
\begin{tikzcd}
H_{*}^{geo,K}(N_0)\otimes H_{*}^{geo,K}(N_1)  \arrow{r}{\psi \otimes \psi} \arrow[swap]{d}{\mathcal{EZ}^{geo}} & \hat{H}_{*}^{K}(N_0)\otimes \hat{H}_{*}^{K}(N_1) \arrow{d}{\mathcal{EZ}} \\
H_{*}^{geo,K \times K}(N_0 \times N_1) \arrow{r}{\psi} & \hat{H}_{*}^{K \times K}(N_0 \times N_1)
\end{tikzcd}
\] \item Let $\Delta: K \subset K \times K$ be the diagonal subgroup.  Then the following diagram commutes:\[ \begin{tikzcd}
H_{*}^{geo,K \times K}(N_0 \times N_1)  \arrow{r}{\psi} \arrow[swap]{d}{r_{\Delta}^{geo}} & \hat{H}_{*}^{K \times K}(N_0 \times N_1)  \arrow{d}{r_{\Delta}} \\
H_{*}^{geo,K}(N_0 \times N_1) \arrow{r}{\psi} & \hat{H}_{*}^{K}(N_0 \times N_1)
\end{tikzcd} \] 
where $r_{\Delta}^{geo}$ and $r_{\Delta}$ denote the restriction to the diagonal subgroup. 
\item For any $N \in \ienr$ and $T \subset K$ a maximal torus,  the comparison map $\psi$ intertwines the Weyl group action on $H_*^{geo,T}(N)$ from \eqref{eq:weylgeom1} with the natural action on $\hat{H}_*^{T}(N).$ 
\end{enumerate} \end{lem} 
\begin{proof}

\emph{Claim (1):} We again take $N \in \enr.$  We need to check that \begin{align} (f_0 \times f_1)_*(PD_{B_{0} \times B_{1}}(\alpha_0 \cup \alpha_1))=\mathcal{EZ}((f_0)_*(PD_{B_{0}}(\alpha_0)\otimes (f_1)_*(PD_{B_{1}}(\alpha_1) ). \end{align} In view of the naturality of $\mathcal{EZ}(-\otimes-)$ in the two arguments (which can be proved on finite dimensional approximations),  it suffices to prove  \begin{align} PD_{B_{0} \times B_{1}}(\alpha_0 \cup \alpha_1))=\mathcal{EZ}((PD_{B_{0}}(\alpha_0)\otimes (PD_{B_{1}}(\alpha_1) ) \end{align} which also can be checked on finite dimensional approximations.  \vskip 10 pt

\emph{Claim (2):} Immediate from the definitions.  \vskip 10 pt
\emph{Claim (3):} The comparison map $\psi$ is natural and it therefore suffices to show that it intertwines $\phi_w^*: H_*^{geo,T}(N) \to H_*^{geo,T}(N)$ with the corresponding map $\phi_w^*: \hat{H}_*^{T}(N) \to \hat{H}_*^{T}(N)$ on semi-infinite homology.  Without a loss of generality,  we can assume that $N \in \cman$ and that geometric cycles are represented by $(N, \alpha, id).$ The claim then follows from the fact that the pull-back map $\phi_w^*: \hat{H}_*^{T}(N) \to \hat{H}_*^{T}(N)$ induced by the automorphism coincides with the induced map on equivariant cohomology under Poincare duality.  
\end{proof}

\subsubsection{Based loop spaces}\label{s:basedloop}

Let $G$ be a compact connected Lie group and let $\Omega G$ denote the space of smooth loops in $G$ based at $\operatorname{id} \in G$.  This admits a natural action of $G$ given by \begin{align} \label{eq:conjug} G \times \Omega G \to \Omega
  G \\ g \cdot \gamma(t) = g\gamma(t)g^{-1}.  \nonumber
\end{align} 

The space $\Omega G$ also admits a Pontryagin product given by pointwise multiplication:
\begin{align} \label{eq:mG} m_{\Omega G}: \Omega G \times \Omega G \to \Omega G. \end{align}
This map \eqref{eq:mG} is manifestly $G$-equivariant if $\Omega G \times \Omega G$ is given the diagonal $G$-action.  Thus for any connected,  closed subgroup $K \subset G$,  it therefore induces a map \begin{align} \label{eq: Ghatmg} m_{K}: H_*^{geo, K}(\Omega G) \otimes H_*^{geo,K}(\Omega G) \to H_*^{geo, K}(\Omega G)  \end{align} 
where $m_{K}$ is the composition of $m_{\Omega G,*}$ with the restriction along the diagonal subgroup and the Eilenberg-Zilber map.  This product equips  $H_*^{geo,K}(\Omega G)$ with the structure of an associative (in fact commutative) algebra. 

Now let $\Omega_{poly}G \subset \Omega G$ denote the space of based loops $S^1 \to G$ which extend to an algebraic map $\mathbb{C}^* \to G_{\mathbb{C}}$,  where $G_\mathbb{C}$ is the complexification of $G$ (\cite[3.5]{PS86} or \cite[Definition 2.1]{Atiyah-Pressley}).  We have that $\Omega_{poly}G$ is a colimit of finite dimensional $G_\mathbb{C}$-projective varieties $\Omega_{poly}G_{\leq \lambda}$ each of which is a compact $G$-ENR by \cite[Theorem 2.1]{Jaworowski}.   The $G$-action \eqref{eq:conjug} and Pontryagin product \eqref{eq:mG} restrict to $\Omega_{poly}G.$ As a consequence,  we have a product \begin{align} \label{eq: Ghatmg2} m_{K}: \hat{H}_*^{K}(\Omega_{poly} G) \otimes \hat{H}_*^{K}(\Omega_{poly} G) \to \hat{H}_*^{K}(\Omega_{poly} G).  \end{align} 

We have the following comparison result:   

\begin{lem} \label{l:compareloops}
 For any connected,  closed subgroup $K \subseteq G$,  there is a natural isomorphism of rings $H_*^{geo,K}(\Omega G) \cong \hat{H}_*^{K}(\Omega_{poly} G).$  \end{lem} 
\begin{proof} 
Note that the Pontryagin product also gives $H_*^{geo,K}(\Omega_{poly} G)$ the structure of an associative algebra. The inclusion $\Omega_{poly}G \to \Omega G$ is a $G$-equivariant homotopy equivalence (see e.g. \cite[Lemma 2.6]{GMP22}).  Hence,  by Lemma \ref{lem:homotopy},  $H_*^{geo,K}(\Omega_{poly} G) \cong H_*^{geo,K}(\Omega G).$ Finally,  because $\Omega_{poly}G \in \ienr$,  Theorem \ref{thm:maincomp} and Lemma  \ref{lem:properties} (1)-(2) imply that $H_*^{geo,K}(\Omega_{poly} G) \cong \hat{H}_*^{K}(\Omega_{poly} G)$ as rings as well.  \end{proof}

\subsection{Concluding the construction}



 Given $[B,\alpha,f] \in H^{geo,T}_*(\Omega G)$, we define
\begin{align}
&\cS([B,\alpha,f], \cdot): SH^*_T(M) \to SH^*_T(M) \nonumber \\
&\cS([B,\alpha,f], z):=\cS_f( \alpha, z) \label{eq:finalcheck}
\end{align}

\begin{prop}
The map  $\cS: H^{geo,T}_*(\Omega G) \times SH^*_T(M) \to SH^*_T(M)$ given by \eqref{eq:finalcheck} is well-defined.
\end{prop}

\begin{proof}
We need to check that $\cS$ is independent of the four equivalence relations in Definition \ref{d:geomHo}.
The independence of the first two is obvious.
The independence of the third follows from a cobordism argument. 


For the last relation, let $B_0,  B_1$ be in $\cman$ and $j: B_0 \to B_1$ be a smooth, equivariant embedding. Using $j$, we identify $B_0$ as a submanifold of $B_1$. Let $f: B_1 \to N$ be an equivariant map and $\alpha \in H_T^{i}(B_0)$.  
Let $\eta_{B_1}':B_1 \to \mathbb{R}$ be a $T$-invariant Morse-Bott function and $g_{B_1}$ be a $T$-invariant Riemannian metric such that any gradient trajectory from a Morse-Bott manifold in $B_0$ to another  Morse-Bott manifold in $B_0$ stays in $B_0$ for all time, and there is no gradient trajectory which starts from a Morse-Bott manifold in $B_1 \setminus B_0$ to a Morse-Bott manifold in $B_0$. In other words, in terms of Morse-Bott complex, we want the generators from $B_0$ forms a subcomplex that can be identified with the Morse-Bott complex of $B_0$.
Moreover, we require that every connected critical submanifold of $\eta_{B_1}'$ is of the form $T/H$.
By Lemma \ref{l:TeqHomotope}, at the cost of possibly a $T$-equivariant homotopy of $f$ to another map, we can assume that $(f,\eta_{B_1}')$ is a good pair (and hence $(f|_{B_0},\eta_{B_1}'|_{B_0})$ is also a good pair).

Let $k:=\dim(B_1)-\dim(B_0)$. The commutative diagram we want is
\begin{equation}\label{eq:4th condition}
\resizebox{\displaywidth}{!}{
\xymatrix{
H^*_T(B_0) \otimes HF^*_T(M,a_0)   \ar[r]\ar[d]
& HF^*(P_{B_0}, (f)^*A_{f}|_{P_{B_0}}) \ar[d] \ar[r]
&  HF^{*-\iota(f)}(P_{B_0}, A_{f}|_{P_{B_0}}) \ar[d]\ar[r]
&  HF^{*-\iota(f)-\dim(B_0)}_T( M,a_1) \ar[d]
\\
H^{*+k}_T(B_1) \otimes HF^*_T(M,a_0)   \ar[r]
& HF^{*+k}(P_{B_1}, f^*A_{f}) \ar[r]
&  HF^{*+k-\iota(f)}(P_{B_1}, A_{f}) \ar[r]
&  HF^{*-\iota(f)-\dim(B_0)}_T( M,a_1) 
}}
\end{equation}
With our setup, the chain level groups in the first row are subcomplexes of the ones in the second row up to grading shifts.
These induce the vertical maps in the diagram. 
Equivalently, these can also be identified with the push-forward maps.
The commutativity of the diagram follows from the fact that the chain level maps send the subcomplexes to the subcomplexes.
Applying the composition of maps in the first row to $(\alpha,a)$ gives $\cS_{f|_{B_0}}(\alpha,a)$, while applying the composition of maps on the other side of the diagram gives $\cS_{f}(j_!(\alpha),a)$.
This finishes the proof.
\end{proof}

\subsection{Weyl-equivariance}

\begin{prop}\label{prop:SeidelWeyleq}
The map  $\cS(\cdot, e): H^{geo,T}_*(\Omega G)  \to SH^*_T(M)$  is Weyl-equivariant.
\end{prop}

\begin{proof}

Let $EG$ be a classifying space of $G$.
For any $w \in N(T)$,  we define $\phi_w: T \to T$ to be $\phi_w(g)=wgw^{-1}$.
Let $(B,\alpha,f)$ be a cycle in $H^{geo,T}_*(\Omega G)$.
At the cost of possibly homotoping $f$ (Lemma \ref{l:TeqHomotope}), we assume that there is an $\eta'_B$ making $(f,\eta_B')$ a good pair.
We denote the $T$ action on $B$, $M$ and $EG$ by $\rho_B$, $\rho_M$ and $\rho_E$, respectively.
We define $w^*\rho_{B}:= \rho_B \circ \phi_w$, and similarly for $w^*\rho_{M}$ and $w^*\rho_{E}$.
Recall that the class $w(B,\alpha,f)$ is represented by $(B,\phi_w^*\alpha, wf)$ with the $\rho_B \circ \phi_{w}^{-1}$ action \eqref{eq:weylgeom2}. 

The map $M \to M$ given by $m \mapsto w \cdot m$ is $(\rho_M,w^*\rho_M)$-equivariant (i.e. it is equivariant with respect to the $\rho_M$ action on the domain and the $w^*\rho_M$ action on the target).
Similarly, the map $EG \to EG$ given by $y \mapsto w \cdot y$ is $(\rho_E, w^*\rho_E)$-equivariant.
For a space $X$ and a $T$-action $\rho_X$ on $X$, we denote the quotient $X/T$ by $X/\rho_X$ to emphsise the action of $T$ on $X$. For two $T$ actions $\rho_{X_1}$ and $\rho_{X_2}$, we use $\rho_{X_1} \otimes \rho_{X_2}$ to denote the diagonal action by $T$.

Fix $w \in N(T)$, we consider the following commutative diagram
\begin{equation}\label{eq:wActionSpace}
\resizebox{\displaywidth}{!}{
\xymatrix{
(B \times EG)/(\rho_B \otimes \rho_E) \times (M \times EG)/(\rho_M \otimes \rho_E)   \ar[d]^{\simeq}
& (B \times EG \times M \times EG)/(\rho_B \otimes \rho_E \otimes \rho_M \otimes \rho_E) \ar[d]^{\simeq} \ar[l]
\\
(B \times EG)/(\rho_B \otimes w^*\rho_E) \times (M \times EG)/(w^*\rho_M \otimes w^*\rho_E)     
& (B \times EG \times M \times EG)/(\rho_B \otimes w^*\rho_E \otimes w^*\rho_M \otimes w^*\rho_E)  \ar[l]
}}
\end{equation}
where the first vertical map sends $((b,y_1),(m,y_2))$ to $((b,wy_1),(wm,wy_2))$, the second vertical map
sends $(b,y_1,m,y_2)$ to $(b,wy_1,wm,wy_2)$ and the horizontal maps are the natural maps.



Let $A_f \in C^{\infty}(S^1 \times (B \times EG \times M \times EG)/(\rho_B \otimes \rho_E \otimes \rho_M \otimes \rho_E))$
be a Hamiltonian given by Proposition \ref{p:accelerator}
and 
\[
\tilde{A}_f:S^1 \times B \times EG \times M \times EG \to \mathbb{R}
\] be its lift. 
We define 
\begin{align*}
\tilde{A}_{wf}: &S^1 \times B \times EG \times M \times EG \to \mathbb{R}\\
&(t,b,y_1,m,y_2) \mapsto  \tilde{A}_f(t,b,w^{-1}y_1,w^{-1}m,w^{-1}y_2) 
\end{align*}
which is $\rho_B \otimes w^*\rho_E \otimes w^*\rho_M \otimes w^*\rho_E$-invariant.
As a result, $\tilde{A}_{wf}$ descends to a function, denoted by $A_{wf}$, on $S^1 \times (B \times EG \times M \times EG)/(\rho_B \otimes w^*\rho_E \otimes w^*\rho_M \otimes w^*\rho_E)$.
It is direct to check that
\begin{align*}
&((wf)^*\tilde{A}_{wf})_t(b,wy_1,wm,wy_2) \\
=&(\tilde{A}_{wf})_t(b,wy_1,wf(b)(t)w^{-1} wm,wy_2)-K_{wf(b)w^{-1},t}(wf(b)(t)w^{-1} wm)\\
=&(f^*\tilde{A}_{f})_t(b,y_1,m,y_2).
\end{align*}
Moreover, both $A_{wf}$ and $(wf)^*A_{wf}$ are cylindrical Hamiltonians compatible with the admissible base triple on $(B \times EG \times EG)/(\rho_B \otimes w^*\rho_E \otimes  w^*\rho_E)$
obtained by pushing forward the admissible base triple on $(B \times EG \times EG)/(\rho_B \otimes \rho_E \otimes  \rho_E)$.

As a result, we can use the commutative diagram \eqref{eq:wActionSpace} to deduce the following commutative diagram
\begin{equation}\label{eq:wAction1}
\xymatrix{
H((B \times EG)/(\rho_B \otimes \rho_E)) \otimes HF^*_T(M,a_0)   \ar[r]\ar[d]^{\phi_w^* \otimes w_*}
& HF(P_{B}, (f)^*A_{f}) \ar[d] \ar[d]
\\
H((B \times EG)/(\rho_B \otimes w^*\rho_E)) \otimes HF^*_{T}(M,a_0)   \ar[r]
& HF(P_{B}^w, (wf)^*A_{wf}) 
}
\end{equation}
where $P_{B}^w:=(B \times EG \times M \times EG)/(\rho_B \otimes w^*\rho_E \otimes w^*\rho_M \otimes w^*\rho_E)$.
Note that $(M \times EG)/(w^*\rho_M \otimes w^*\rho_E)$ is canonically the same as 
$(M \times EG)/(\rho_M \otimes \rho_E) $ so the second factor in the top left and bottom left tensor products in \eqref{eq:wAction1} can be canonically identified and we just denote it by $HF^*_{T}(M,a_0) $.
Moreover, the map on $HF^*_T(M,a_0)$ is precisely the action by $w$ so we denote it by $w_*$.
On the other hand, the map between the first factor in the top left and bottom left tensor products is precisely the $\phi_w^*$ described in \eqref{eq:pullbackalpha}. 

Similarly, we also have the following commutative diagram of maps between spaces
\begin{equation}\label{eq:wActionSpace2}
\resizebox{\displaywidth}{!}{
\xymatrix{
(B \times EG \times M \times EG)/(\rho_B \otimes \rho_E \otimes \rho_M \otimes \rho_E)  \ar[d]^{\simeq} \ar[r]
& (EG \times M \times EG)/( \rho_E \otimes \rho_M \otimes \rho_E) \ar[d]^{\simeq}
\\
(B \times EG \times M \times EG)/(\rho_B \otimes w^*\rho_E \otimes w^*\rho_M \otimes w^*\rho_E)   \ar[r]
& ( EG \times M \times EG)/( w^*\rho_E \otimes w^*\rho_M \otimes w^*\rho_E)  
}}
\end{equation}
where the first vertical map sends $(b,y_1,m,y_2)$ to $(b,wy_1,wm,wy_2)$
and the second vertical map sends $(y_1,m,y_2)$ to $(wy_1,wm,wy_2)$.
The corresponding commutative diagram of maps between cohomology is
\begin{equation}\label{eq:wAction2}
\xymatrix{
 HF(P_{B}, (f)^*A_{f}) \ar[d] \ar[r]
&  HF(P_{B}, A_{f}) \ar[d]\ar[r]
&  HF_T( M,a_1) \ar[d]^{w_*}
\\
 HF(P_{B}^w, (wf)^*A_{wf}) \ar[r]
&  HF(P_{B}^w, A_{wf}) \ar[r]
&  HF_{T}( M,a_1) 
}
\end{equation}
In the last column, we use the canonical isomorphism  between 
$(EG \times M \times EG)/( \rho_E \otimes \rho_M \otimes \rho_E)$ and 
$( EG \times M \times EG)/( w^*\rho_E \otimes w^*\rho_M \otimes w^*\rho_E)  $ again.
Moreover, the last vertical map $w_*$ is precisely the action on $HF_{T}( M,a_1) $ by $w$.


By combining \eqref{eq:wAction1}, \eqref{eq:wAction2} and passing to the direct limit, we show that
$w_* \cS_f(\alpha,e)$ (the output of $(\alpha,e)$ under the composition of the maps in the first row and the last vertical map) equals to $\cS_{wf}( \phi_w^*\alpha, w_*  e)$ (the output of $(\alpha,e)$ under the composition of the first vertiical map and the maps in the second row).
Moreover, we know that $\cS_{wf}( \phi_w^*\alpha, w_* e)=\cS_{wf}( \phi_w^*\alpha, e)$ so it proves that $\cS(\cdot, e)$ is Weyl-equivariant.
\end{proof}


\section{An algebra homomorphism}\label{s:algebra}

In this section, we are going to show that $\cS(\cdot, e): H^{geo,T}_*(\Omega G) \to SH^*_T(M)$ is an algebra homomorphism as promised in \eqref{eq:moduleactintroT}.


First note that for $[B_1,\alpha_1,f_1], [B_2,\alpha_2,f_2] \in H^{geo,T}_*(\Omega G)$, their product is represented by
$(B_2 \times B_1, \pi_2^*\alpha_2 \cup \pi_1^*\alpha_1, f)$, where $f:B_2 \times B_1 \to \Omega G$ is given by $f(b_2,b_1)=f_2(b_2)f_1(b_1)$ (see \eqref{eq:prod}).

\begin{prop}\label{p:double_taut}
For $i=1,2$, let $(f_i:B_i \to \Omega G,\eta_{B_i}')$ be a good pair.
Let $B:=B_2 \times B_1$, $F_i:B \to \mathbb{R}$ be $F_i(b_2,b_1)=f_i(b_i)$ and $F:B \to \mathbb{R}$ be $F(b_2,b_1)=f_2(b_2)f_1(b_1)$.
Let $\eta':B \to \mathbb{R}$ be $\eta_{B_2 \times B_1}'(b_2,b_1)=\eta'_{B_1}(b_1)+\eta'_{B_2}(b_2)$
Then $(F_i,\eta')$ and $(F,\eta')$ are good pairs, $F_1^*F_2^*H=F^*H$ for any cylindrical Hamiltonian $H \in C^{\infty}(S^1 \times P_B)$ and
\begin{align}\label{eq:F1F2}
\cC_F \simeq \cC_{F_2} \circ \cC_{F_1} :HF(P_B,F^*H) \to HF(P_B,H)
\end{align}
\end{prop}

\begin{proof}
Clearly, $\eta_{B_2 \times B_1}'$ is $T$-invariant (in fact, $T \times T$-invariant).
A critical submanifold of $\eta'$ is a product of critical submanifolds of $\eta'_{B_1}$ and $\eta'_{B_2}$.
Let $N_{\eta_i}$ be a $T$-invariant neighborhood of $\critp(\eta'_{B_i})$ such that $f_i$ is locally constant.
Then $F_i$ and $F$ are locally constant near the $T$-invariant set $N_{\eta_1} \times N_{\eta_2}$ so $(F_i,\eta_{B_2 \times B_1}')$ and $(F,\eta_{B_2 \times B_1}')$ are good pairs.

Let $\tilde{H}:S^1 \times B_1 \times B_2 \times ET \times M \to \mathbb{R}$ be the lift of $H$.
Now we check that
\begin{align*}
&F_1^*F_2^*\tilde{H}_t(b_1,b_2,y,m)\\
=&F_2^*\tilde{H}_t(b_1,b_2,y,(f_1(b_1)(t)) \cdot m)-K_{f_1(b_1),t}((f_1(b_1)(t)) \cdot m) \\
=&\tilde{H}_t(b_1,b_2,y,(f_2(b_2)(t))(f_1(b_1)(t)) \cdot m)-K_{f_2(b_2),t}((f_2(b_2)(t))(f_1(b_1)(t)) \cdot m)-K_{f_1(b_1),t}((f_1(b_1)(t)) \cdot m)\\
=&\tilde{H}_t(b_1,b_2,y,(f_2(b_2)(t))(f_1(b_1)(t)) \cdot m)-K_{f_2(b_2)f_1(b_1),t}((f_2(b_2)(t))(f_1(b_1)(t)) \cdot m)\\
=&F^*\tilde{H}_t(b_1,b_2,y,m)
\end{align*}
Here the third equality follows from Lemma \ref{l:compoHam}(1).

To see \eqref{eq:F1F2}, we simply observe that in the proof of Proposition \ref{p:correlationproof}, the $\tilde{\Phi}$ associated to $F$ is the composition of the  $\tilde{\Phi}$ associated to $F_2$ and $F_1$.
\end{proof}

The commutative diagram we want to establish is the following, and Proposition \ref{p:double_taut} gives the commutativity of the triangle in the middle.
\begin{equation}\label{eq:BigCom}
\resizebox{\displaywidth}{!}{%
\xymatrix{
H^*_T(B_2) \otimes H^*_T(B_1) \otimes HF^*_T(M,a_0)   \ar[r]^{id \otimes Q_1} \ar[d]^{Q_2 \otimes id} 
& H^*_T(B_2) \otimes HF^*(P_{B_1}, f_1^*A_{f_1}) \ar[d]^{Q_3}\ar[r]^{id \otimes \cC_{f_1}}
&  H^*_T(B_2) \otimes HF^*(P_{B_1}, A_{f_1}) \ar[d]^{Q_4}\ar[r]^{id \otimes \Pi_1}
& H^*_T(B_2) \otimes HF^*_T( M,a_1) \ar[d]^{Q_5}
\\
H^*_T(B_2 \times B_1) \otimes HF^*_T(M, a_0)   \ar[r]^{Q_6} 
&HF^*(P_{B_2 \times B_1},F^*A_F)   \ar[r]^{\cC_{F_1}} \ar[dr]^{\cC_F} 
& HF^*(P_{B_2 \times B_1},F_2^*A_F) \ar[r]^{\Pi_2} \ar[d]^{\cC_{F_2}} 
& HF^*(P_{B_2},f_2^*A_{f_2}) \ar[d]^{\cC_{f_2}} 
\\
&  
&HF^*(P_{B_2 \times B_1},A_F) \ar[dr]^{\Pi_5} \ar[r]^{\Pi_3}
& HF^*(P_{B_2},A_{f_2})  \ar[d]^{\Pi_4}
\\
&
&
& HF^*_T(M,a_2) 
}}
\end{equation}
The rest of the subsection is devoted to explaining all the maps in the commutative diagram and proving their commutativity.

Once the commutativity of the diagram is proved and after passing to the direct limit with respect to slopes, the composition of the maps in the first row and last column is $(\alpha_2,\alpha_1,z) \mapsto \cS_{f_2}(\alpha_2, \cS_{f_1}(\alpha_1,z))$, while the composition of the maps on `the other side' is 
$(\alpha_2,\alpha_1,z) \mapsto \cS_{F}(\pi_2^*\alpha_2 \cup \pi_1^*\alpha_1,z)$.
When $z$ is the unit $e \in SH^*_T(M)$, we obtain
\[
\cS_{F}(\pi_2^*\alpha_2 \cup \pi_1^*\alpha_1,e)= \cS_{f_2}(\alpha_2, \cS_{f_1}(\alpha_1,e))=\cS_{f_2}(\alpha_2, e) \cS_{f_1}(\alpha_1,e)
\]
so $\cS(\cdot,e)$ is an algebra map. Here, the second equality comes from Lemma \ref{l:moduleType}.

Our first task is to explain the maps $Q_3$ and $Q_4$.
To fix the ideas, the maps $Q_3$ and $Q_4$ are pull-back maps (in the sense of Section \ref{ss:pb}) with respect to the following  fibre product diagram.
\begin{equation}\label{eq:B21Total}
\begin{tikzcd}
(B_2 \times ET \times B_1 \times M \times ET)/T \arrow{r} \arrow{d}
& (B_2 \times ET)/T  \times (B_1 \times M \times ET)/T \arrow{d}
\\
(B_2 \times ET \times B_1 \times ET)/T \arrow{r} 
& (B_2 \times ET)/T  \times (B_1 \times ET)/T
\end{tikzcd}
\end{equation}
However, the slopes of the Hamiltonians are not constant functions so we have to be careful with the auxiliary choices so that $C^0$ a priori estimates can be achieved to give well-defined pull-back maps. The auxiliary choices are explained below.

For $i=1,2$, let $(\eta_i,g_i,\nabla_i)$ be the admissible base triple obtained by applying Proposition \ref{p:accelerator} to the good pairs $(f_i,\eta_{B_i}')$.
We assume that the connections on $(B_i \times ET)/T \to BT$ for both $i=1,2$ are induced by the same connection $\nabla_T$ for the bundle $ET \to BT$.
In particular, over the region $U$ where $\nabla_T$ is flat, we have $((B_i)_{borel}|_U,g_i) \simeq (B_i \times U,g_{B_i} \oplus g_{BT}|_U)$, where $g_{B_i}$ is a $T$-equivariant metric on $B_i$.
However, for regularity reason, we don't want to build $\eta_i$ using the same Morse function $\eta_{BT}: BT \to \mathbb{R}$. 
 Instead, let $\eta_{BT,i}:BT \to \mathbb{R}$ be a Morse function that is a small perturbation of $\eta_{BT}$ so that it satisfies the properties as required in Section \ref{ss:GoodClass}.
Let  $\eta_i':(B_i \times ET)/T \to \mathbb{R}$ be the pull-back of the Morse function $\eta_{BT,i}$.
 We assume that $\eta_i=\eta_i'+\epsilon_i \chi_i(u) \eta_{B_i}(b_i) $, where $\eta_{B_i}:B_i \to \mathbb{R}$ is a Morsification of $\eta_{B_i}'$ as in Section \ref{ss:GoodClass}.


It is easy to check that the following diagram is a fibre product diagram
\begin{equation*}
\begin{tikzcd}
(B_2 \times ET \times B_1 \times ET)/T \arrow{r} \arrow{d}
& (B_2 \times ET)/T  \times (B_1 \times ET)/T \arrow{d}
\\
(ET \times ET)/T \arrow{r} 
& (ET)/T  \times  (ET)/T 
\end{tikzcd}
\end{equation*}
The map from the top-left to bottom-right, the second horizontal map and the second vertical map
are $T \times T$-bundles with fibres $T \times B_2 \times B_1$, $T$ and $B_2 \times B_1$, respectively.
We equip all of them with the connection induced by $\nabla_T \oplus \nabla_T$.

We choose a Riemannian metric $g_{BBT}$ on $(ET \times ET)/T$ such that the map to $((ET)/T  \times  (ET)/T, g_{BT} \oplus g_{BT}) $ is a Riemannian submersion, and the orthogonal complement of the fibres argees with the horizontal subspace $T^{hor} \subset T((ET \times ET)/T)$ of the connection.
The second vertical map is also a Riemannian submersion, where the source is equipped with the metric $g_2 \oplus g_1$.
The metric $g_{BBT}$ and $g_2 \oplus g_1$ together induces a Riemannian metric $g_{21}$ on $(B_2 \times ET \times B_1 \times ET)/T$ such that all the four maps are Riemannian submersions.
Indeed, the tangent spaces of $(B_2 \times ET \times B_1 \times ET)/T$, $(B_2 \times ET)/T  \times (B_1 \times ET)/T$ and $(ET \times ET)/T$ are given by $T(T) \oplus T(B_2 \times B_1) \oplus T^{hor}$, $T(B_2 \times B_1) \oplus T^{hor}$ and $T(T) \oplus T^{hor}$ respectively, and the Riemannian metrics $g_{BBT}$ and $g_2 \oplus g_1$ respect the product decomposition and agree on $T^{hor}$.
Therefore, they uniquely determine a metric $g_{21}$ making all the four maps Riemannian submersions.
Moreover, if we equip the first vertical map the connection induced by $\nabla_T$, then its horizontal subspaces agree with
the $g_{21}$-orthgonal complement of the fibres.

Let $\eta_{BBT}': (ET \times ET)/T \to \mathbb{R}$ be the pull-back of $\eta_{BT}$.
Over the flat region $U$, we have $((ET \times ET)/T|_U, g_{BBT}) \simeq (T \times U, g_T \times g_{BT}|_U)$.
To Morsify $\eta_{BBT}'$, we choose a Morse function $\eta_T: T \to \mathbb{R}$ and define $\eta_{BBT}=\eta_{BBT}'+ \chi(u) \eta_T(g)$.
Let $\eta_{21}':(B_2 \times ET \times B_1 \times ET)/T \to \mathbb{R}$ be the pull-back of $\eta_{BBT}$
and define its Morsification $\eta_{21}:=\eta_{21}'+\chi(u)(\eta_{B_1}(b_1)+\eta_{B_2}(b_2))$.

Let $\nabla_{21}$ be the connection for the first vertical map in \eqref{eq:B21Total} induced by $\nabla_T$.
Then $(\eta_{21},g_{21},\nabla_{21})$ is an admissible base triple, and is of the form obtained by applying Proposition \ref{p:accelerator} to  $(f_2f_1,\pi_1^*\eta_{B_1}'+\pi_2^*\eta_{B_2}')$.

Following Proposition \ref{p:accelerator}, let $A_{f_i} \in C^{\infty}(S^1 \times P_{B_i})$ be a cylindrical Hamiltonian such that both $A_{f_i}$ and $f_i^*A_{f_i}$ are compatible with $(\eta_i,g_i,\nabla_i)$ and $A_{f_i}$ has slope $\bfs_{A_{f_i}}= (\bfs_i)_{borel}$. Here, $\bfs_i \in C^{\infty}(S^1 \times B_i)$ is a $T$-invariant function such that it is locally constant on $N_{\eta_i}$, 
$\bfs_i > c_{f_i,K}$ and \eqref{eq:grad_on_B} is satisfied ($\eta_B$, $\tau_B$ and $f$ are replaced with $\eta_{B_i}$, $\tau_{B_i}$ and $f_i$, respectively).
We assume furthermore that $A_{f_i}$ are chosen such that there are constants $a_0,a_1,a_2$ with 
\begin{align}\label{eq:SlopeBound}
0<a_0 \le \bfs_{f_1^*A_{f_1}},   \bfs_{A_{f_1}} \le a_1 \le  \bfs_{f_2^*A_{f_2}},   \bfs_{A_{f_2}} \le a_2
\end{align}
The maps in the first row and last column of \eqref{eq:BigCom} are precisely those maps introduced in Section \ref{ss:Taut} and \ref{ss:Gcycle} to define $\cS_{f_i}$.

Similarly, we are going to pick a cylindrical Hamiltonian $A_F \in C^{\infty}(S^1 \times P_{B_2 \times B_1})$  such that $A_F$, $F_2^*A_F$ and $F^*A_F$ are all compatible with $(\eta_{21},g_{21},\nabla_{21})$.
To do that, let $\bfs_2' \in  C^{\infty}(S^1 \times B_2)$ be another $T$-invariant function such that
 it is locally constant in $N_{\eta_2}$, $\bfs_2' > c_{f_2,K}$ and \eqref{eq:grad_on_B} is satisfied.
The following lemma explains the choice of cylindrical Hamiltonians we use for the target of the maps $Q_3$ and $Q_4$.

\begin{lem}\label{l:SlopeH2}
If the values of $(\pi_1^*\bfs_1+\pi_2^*\bfs_2')|_{N_{\eta_2} \times N_{\eta_1}}$ do not lie in the action spectrum of the contact boundary $\partial \bar{M}$, then there is a cylindrical Hamiltonian $A_F \in C^{\infty}(S^1 \times P_{B_2 \times B_1})$ with slope $\bfs_{A_F}=(\pi_1^*\bfs_1+\pi_2^*\bfs_2')_{borel}$ such that 
$A_F$, $F_2^*A_F$ and $F^*A_F$ are all compatible with $(\eta_{21},g_{21},\nabla_{21})$.
\end{lem}

\begin{proof}
The proof is in parallel to the proof of Proposition \ref{p:accelerator}. The key statement to check is 
$\frac{d}{ds}\bfs_{F^*A_F}(t,\tau(s)), \frac{d}{ds}\bfs_{F_2^*A_F}(t,\tau(s)) \le 0$ for any gradient trajectory $\tau: \mathbb{R} \to (B_2 \times ET_{n_2} \times B_1 \times ET_{n_1})/T$.
Note that, before taking the Borel construction, for any gradient trajectory $\tau_{B_2 \times B_1}=(\tau_{B_2},\tau_{B_1}): \mathbb{R} \to B_2 \times B_1$, we have
\begin{align*}
\frac{d}{ds} (\pi_1^*\bfs_1+\pi_2^*\bfs_2' )(t,\tau_{B_2 \times B_1}(s)) 
\le& \frac{d}{ds} \bfs_1(t,\tau_{B_1}(s))  + \frac{d}{ds} \bfs_2'(t,\tau_{B_2}(s))\\
\le &  - \max_{m\in \partial \bar{M}} \left|\frac{d}{ds}\bfs_{K_{f_1(\tau_{B_1}(s))},t}(m) \right|
- \max_{m\in \partial \bar{M}} \left|\frac{d}{ds}\bfs_{K_{f_2(\tau_{B_2}(s))},t}(m) \right| \\
\le  &- \max_{m\in \partial \bar{M}} \left|\frac{d}{ds}\bfs_{K_{f_2f_1(\tau_{B_2 \times B_1}(s))},t}(m) \right|
\end{align*}
where the last inequality comes from Lemma \ref{l:compoHam}(1).

It in turn follows that
the analog of \eqref{eq:grad_on_B_borel} is true in our case. More precisely, we have
\begin{align*}
\frac{d}{ds}(\pi_1^*\bfs_1+\pi_2^*\bfs_2')_{borel}(t,\tau(s)) \le - \max_{m\in \partial \bar{M}} \left |\frac{d}{ds}\bfs_{(F^*0)_{\tau(s)}}(t,m),  \right|, \text{ and } \\
\frac{d}{ds}(\pi_1^*\bfs_1+\pi_2^*\bfs_2')_{borel}(t,\tau(s)) \le - \max_{m\in \partial \bar{M}} \left |\frac{d}{ds}\bfs_{(F_2^*0)_{\tau(s)}}(t,m),  \right|
\end{align*}
which show that $\frac{d}{ds}\bfs_{F^*A_F}(t,\tau(s)), \frac{d}{ds}\bfs_{F_2^*A_F}(t,\tau(s)) \le 0$.
\end{proof}

For a fixed $a_2'$ as in Lemma \ref{l:SlopeH2}, by possibly choosing a larger $a_1$ and $a_2$, we can assume that 
\begin{align}\label{eq:SlopeBound2}
a_0 \le \bfs_{F^*A_F}, \bfs_{F_2^*A_F},   \bfs_{A_F},  \le a_1
\end{align}
This is not needed at the moment to define the maps $Q_3$ and $Q_4$, but it will be needed when we define $\Pi_2$ and $\Pi_3$ in \eqref{eq:BigCom}.

We are now ready to define $Q_3$ and $Q_4$.

\begin{lemdefn}\label{l:UpMidSq}
Let $A_{f_1}$ and $A_F$ be as above. Then there are well-defined pull-back maps
\begin{align}
&Q_3: H^*_T(B_2) \otimes HF^*(P_{B_1}, f_1^*A_{f_1})  \to HF^*(P_{B_2 \times B_1},F^*A_F)\\
&Q_4:  H^*_T(B_2) \otimes HF^*(P_{B_1}, A_{f_1}) \to  HF^*(P_{B_2 \times B_1},F_2^*A_F)
\end{align}
such that $\cC_{F_1} \circ Q_3 =Q_4 \circ (id \otimes \cC_{f_1})$.
\end{lemdefn}

\begin{proof}
We are going to explain the definition of $Q_3$. The definition of $Q_4$ is similar.

Following Section \ref{ss:pb}, for $n_1,n_2 \in \mathbb{N}$, let $B'=(B_2 \times ET_{n_2} \times B_1 \times ET_{n_1})/T$, $B=(B_2 \times ET_{n_2})/T \times (B_1 \times ET_{n_1})/T$ and $h:B' \to B$ be the obvious map.
We also let $P'=(B_2 \times ET_{n_2} \times B_1 \times M \times ET_{n_1})/T$, $P=(B_2 \times ET_{n_2})/T \times (B_1 \times M \times ET_{n_1})/T$ and $\tilde{h}:P' \to P$ be the obvious map.
Note that $P'$ is the pull-back of $P \to B$ along $h$.

Let $\rho:(-\infty,0] \to [0,1]$ be a smooth monotone increasing function that is $0$ for $s \ll 0$ and equals to $1$ near $s=0$. 
We define $H_s:=f_1^*A_{f_1}$ for all $s \ge 0$, and $H'_s:=(1-\rho(s))F^*A_F+ \rho(s) \tilde{h}^*H_0$ for $s \le 0$.
We claim that the condition in Lemma \ref{l:pull-back} is satisfied so that the pull-back is well-defined.
To see this, note that $\tilde{h}^*H_0=\tilde{h}^*f_1^*A_{f_1}=F_1^* (\tilde{h}^*A_{f_1})$ and $\bfs_{\tilde{h}^*A_{f_1}}= (\pi_1^* \bfs_1)_{borel}$.
Recall that  $\bfs_{A_{F}}= (\pi_1^* \bfs_1+\pi_2^* \bfs_2')_{borel}$ and $F^*A_F=F_1^*F_2^*A_F$ so
\[H'_s=F_1^*((1-\rho(s))F_2^*A_F +\rho(s)\tilde{h}^*A_{f_1})=F_1^*(\tilde{h}^*A_{f_1}+(1-\rho(s))(F_2^*A_F - \tilde{h}^*A_{f_1})).\]
Let $H''_s:=\tilde{h}^*A_{f_1}+(1-\rho(s))(F_2^*A_F - \tilde{h}^*A_{f_1})$.
For a gradient trajectory $\tau^-: (-\infty,0] \to B'$, we denote its projection to $(B_i \times ET_{n_i})/T$ by $\tau_i^-$ for $i=1,2$.
Then we have
\begin{align*}
&\frac{d}{ds} (\bfs_{(H'_s)_{\tau^-(s)},t}(m))\\
\le& \max_{m} \frac{d}{ds} (\bfs_{(H''_s)_{\tau^-(s)},t}(m)) + \max_{m} \left | \frac{d}{ds} \bfs_{(F_1^*0)_{\tau^-(s)}}(t,m) \right | \\
\le &\max_{m} \frac{d}{ds} (\bfs_{(\tilde{h}^*A_{f_1})_{\tau^-(s)},t}(m)) 
+\max_{m} \frac{d}{ds} (\bfs_{((1-\rho(s))(F_2^*A_F - \tilde{h}^*A_{f_1}))_{\tau^-(s)},t}(m)) 
+ \max_{m} \left | \frac{d}{ds} \bfs_{(F_1^*0)_{\tau^-(s)}}(t,m) \right |\\
\le &\left ( \frac{d}{ds} (\bfs_1)_{borel}(t,\tau_1^-(s)) + \max_{m} \left | \frac{d}{ds} \bfs_{(F_1^*0)_{\tau^-(s)}}(t,m) \right | \right ) \\
&+(1-\rho(s)) \left ( \frac{d}{ds} (\bfs'_2)_{borel}(t,\tau_2^-(s)) + \max_{m} \left | \frac{d}{ds} \bfs_{(F_2^*0)_{\tau^-(s)}}(t,m) \right | \right ) \\
&-\rho'(s) \left ((\bfs'_2)_{borel}(t,\tau_2^-(s)) - \max_{m} \left |\bfs_{(F_2^*0)_{\tau^-(s)}}(t,m) \right | \right)
\end{align*}
On the RHS, the first term and the second term are non-positive because $\bfs_1$ and $\bfs_2$ are chosen to satisfy this property (see \eqref{eq:grad_on_B_borel}).
The third term is also non-positive because $\rho' \ge 0$ and $\bfs_2' > c_{f_2,K}=c_{F_2,K} \ge \max_{m} \left |\bfs_{F_2^*0}(t,m) \right |$ for all $t \in S^1$ (see \eqref{eq:cfK}).
As a result, we have $\frac{d}{ds} (\bfs_{(H'_s)_{\tau^-(s)},t}(m)) \le 0$ so  the condition in Lemma \ref{l:pull-back} is satisfied.


The definition of $Q_4$ is similar.
Moreover, if we use the same cutoff function $\rho(s)$ for both $Q_3$ and $Q_4$ and appropriate almost complex structures, then they are tautologically identified with each other via the map $\Phi$ in Proposition \ref{p:correlationproof}.
This gives the commutativity   $\cC_{F_1} \circ Q_3 =Q_4 \circ (id \otimes \cC_{f_1})$.
\end{proof}

Lemma \ref{l:UpMidSq} gives us the commutativity of the upper middle square in \eqref{eq:BigCom}.
Now we turn to the upper left square.

\begin{lem}\label{l:trivialCom}
The following diagram commutes
\begin{equation}\label{eq:1st-square}
\begin{tikzcd}
H^*_T(B_2) \otimes H^*_T(B_1) \otimes HF^*_T(M,a_0)   \arrow{r} \arrow[swap]{d}
&H^*_T(B_2) \otimes H^*(P_{B_1},a_0)  \arrow[swap]{d}
\\
HF^*_T(B_2 \times B_1) \otimes HF^*_T(M,a_0)  \arrow{r} 
&HF(P_{B_2 \times B_1}, a_0)
\end{tikzcd}
\end{equation}

\end{lem}

\begin{proof}
The commutativity of 
\begin{equation*}
\resizebox{\displaywidth}{!}{
\xymatrix{
(B_2 \times ET_{n_2})/T \times (B_1 \times ET_{n_1})/T \times (M \times ET_{m})/T   
&(B_2 \times ET_{n_2})/T \times (B_1 \times ET_{n_1} \times M \times ET_{m})/T    \ar[l]
\\
(B_2 \times ET_{n_2} \times B_1 \times ET_{n_1})/T \times (M \times ET_{m})/T    \ar[u] 
&(B_2 \times ET_{n_2} \times B_1 \times ET_{n_1} \times M \times ET_{m})/T \ar[l] \ar[u]  
}}
\end{equation*}
and the naturality (Corollary \ref{c:natural}) imply a commutative diagram of the corresponding pull-back maps.
By passing to the inverse limit and using the independence of the model of the classifying space (Lemma \ref{l:indepClassifyingModel}), we obtain the result.
\end{proof}

Note that, the Hamiltonians in the domain and the target of the second vertical map of \eqref{eq:1st-square} have constant slopes so the groups and the map are independent of the model of the classifying space (cf. Lemma \ref{l:indepClassifyingModel}).
To complete the upper left square in \eqref{eq:BigCom}. We need to show the following commutativity.

\begin{lem}\label{l:1st-square}
The following diagram commutes
\begin{equation}\label{eq:1stt-square}
\begin{tikzcd}
H^*_T(B_2) \otimes HF^*(P_{B_1},a_0)   \arrow{r} \arrow[swap]{d}
&H^*_T(B_2) \otimes HF^*(P_{B_1}, f_1^*A_{f_1})  \arrow[swap]{d}{Q_3}
\\
HF(P_{B_2 \times B_1}, a_0)  \arrow{r} 
&HF^*(P_{B_2 \times B_1},F^*A_F)
\end{tikzcd}
\end{equation}

\end{lem}

\begin{proof}
We use the admissible base triple $(\eta_1,g_1,\nabla_1)$ for $H^*(P_{B_1},a_0)$ and $HF^*(P_{B_1}, f_1^*A_{f_1})$, and the admissible base triple  $(\eta_{21},g_{21},\nabla_{21})$ for $HF(P_{B_2 \times B_1}, a_0) $ and  $HF^*(P_{B_2 \times B_1},F^*A_F)$.

Similar to Lemma \ref{l:kappaIndchoice} and \ref{l:inde_kappa}, it suffices to find a homotopy between the Floer data corresponding to the two different compositions such that the maximum principle can be achieved.
The Hamiltonian $(H_s)_{s \in \mathbb{R}}$ defining 
$ H^*(P_{B_1},a_0) \to HF^*(P_{B_1}, f_1^*A_{f_1})  $ is of the form
\[
H_s=(1-\rho(s))f_1^*A_{f_1} + \rho(s) A_{a_0}
\]
for some Hamiltonian $A_{a_0} \in C^{\infty}(S^1 \times P_{B_1})$ that is of constant slope $a_0$, and some monotone increasing function $\rho:\mathbb{R} \to [0,1]$ such that $\rho(s)=0$ for $s \ll 0$
and $\rho(s)=1$ for $s \gg 0$.\footnote{More precisely, we need to use the restriction of $H_s$ to the finite approximations to define the continuation maps and then pass to direct limit.}
We can regard $(H_s)_{s \in \mathbb{R}}$ as a family of elements in $C^{\infty}(S^1 \times (B_2)_{borel} \times P_{B_1})$ which defines
$H^*_T(B_2) \otimes H^*(P_{B_1},a_0) \to H^*_T(B_2) \otimes HF^*(P_{B_1}, f_1^*A_{f_1}) $.
The Hamiltonian defining $Q_3$ is explained in Lemma/Definition \ref{l:UpMidSq}.

On the other side, the Hamiltonians $(H'_s)_{s \in \mathbb{R}_{\le 0}}$ and $(H_s)_{s \in \mathbb{R}_{\ge 0}}$ defining 
 $H^*_T(B_2) \otimes H^*(P_{B_1},a_0) \to HF(P_{B_2 \times B_1}, a_0)$
are family of Hamiltonians of constant slope $a_0$ in 
 $C^{\infty}(S^1 \times P_{B_2 \times B_1})$ and
 $C^{\infty}(S^1 \times (B_2)_{borel} \times P_{B_1})$ respectively.
Finally, the Hamiltonian  $(H_s)_{s \in \mathbb{R}}$ defining 
$ HF(P_{B_2 \times B_1}, a_0) \to HF^*(P_{B_2 \times B_1},F^*A_F) $ is of the form
\[
H_s=(1-\rho(s))F^*A_F + \rho(s) A_{a_0}
\]
for some Hamiltonian $A_{a_0} \in C^{\infty}(S^1 \times P_{B_2 \times B_1})$ that is of constant slope $a_0$, and some monotone increasing function $\rho:\mathbb{R} \to [0,1]$ such that $\rho(s)=0$ for $s \ll 0$
and $\rho(s)=1$ for $s \gg 0$.

Let $\rho^+:\mathbb{R}_{\ge 0} \to [0,1]$,
$\rho^-:\mathbb{R}_{\le 0} \to [0,1]$
and $\rho:\mathbb{R} \to [0,1]$ be monotone increasing functions such that $\rho^+(s), \rho^-(s), \rho(s)=0$ when $s$ near the left end,
and $\rho^+(s), \rho^-(s), \rho(s)=1$ when $s$ near the right end.
To construct a homotopy between the two concatenation of Floer data coming from the two compositions, we consider $(H_{s,r})_{s \in \mathbb{R}_{\ge 0}, r \in \mathbb{R}} \in C^{\infty}(S^1 \times P_{B_1})$ of the form
\[
H_{s,r}=(1-\rho(r))A_{a_0}+\rho(r) \left ((1-\rho^+(s)) f_1^*A_{f_1} +\rho^+(s)A_{a_0} \right)
\]
and $(H'_{s,r})_{s \in \mathbb{R}_{\le 0}, r \in \mathbb{R}} \in C^{\infty}(S^1 \times P_{B_2 \times B_1})$ of the form
\[
H'_{s,r}=\rho^-(s) \left( (1-\rho(r) \tilde{h}^*A_{a_0} +\rho(r)  \tilde{h}^*f_1^*A_{f_1}  \right) +(1-\rho^-(s))F^*A_F
\]
where $\tilde{h}: P_{B_2 \times B_1} \to P_{B_1}$ is the obvious map (as in Lemma/Definition \ref{l:UpMidSq}).
Notice that for any fixed $r$, $H_{s,r}=A_{a_0}$ when $s \gg 0$, $H'_{s,r}=F^*A_F$ when $s \ll 0$, and $\tilde{h}^*H_{s,r}=H'_{s,r}$ when $s$ is close to $0$.
It is straightforward to check that for all $r$, the slope of $H_{s,r}$ and $H'_{s,r}$ along any gradient trajectory is decreasing.
Therefore, the maximum principle can be applied.
Moreover, for appropriate $\rho^+, \rho^-, \rho$, the limit as $r \to \pm \infty$ can be made arbitrarily close to the glued Floer data of the two compositions respectively.
Therefore, the rigid count of the moduli space associated to the pair $(H_{s,r})_{s \in \mathbb{R}_{\ge 0}, r \in \mathbb{R}}$ and
$(H'_{s,r})_{s \in \mathbb{R}_{\le 0}, r \in \mathbb{R}}$ defines a homotopy between the two chain level compositions and hence induces the comutativity of  \eqref{eq:1stt-square}.
This finishes the proof.
\end{proof}

Our next task is to define $\Pi_2$ and $\Pi_3$. 
The pull-back diagram we use this time is
\begin{equation}
\begin{tikzcd}
P_{B_2 \times B_1}=(B_2  \times B_1 \times M \times ET)/T \arrow{r} \arrow{d}
& (B_2 \times ET \times M)/T=P_{B_2} \arrow{d}
\\
(B_2  \times B_1 \times ET)/T \arrow{r} 
& (B_2 \times ET)/T  
\end{tikzcd}
\end{equation}
Let $\tilde{h}_2: P_{B_2 \times B_1} \to P_{B_2}$ be the obvious map.
The bundle in the second column is equipped with the admissible base triple $(\eta_2,g_2,\nabla_2)$.
We can choose a Riemannian metric  on $(B_2  \times B_1 \times ET)/T$
such that the second horizontal map is a Riemannian submersion.
By doing the same Morsification procedure as above to the pull-back of $\eta_2$ on $(B_2  \times B_1 \times ET)/T$, we can obtain an admissible base triple for the bundle in the first column such that the push-forward maps
\[
HF^*(P_{B_2 \times B_1},\tilde{h}_2^*f_2^*A_{f_2}) \to HF^*(P_{B_2},f_2^*A_{f_2})
\]
and
\[
HF^*(P_{B_2 \times B_1},\tilde{h}_2^*A_{f_2}) \to HF^*(P_{B_2},A_{f_2})
\]
are well-defined.
Moreover, we have the following tautological commutative diagram (as in the proof of Proposition \ref{p:correlationproof})
\begin{equation}\label{eq:3rd-square}
\begin{tikzcd}
HF^*(P_{B_2 \times B_1},\tilde{h}_2^*f_2^*A_{f_2})   \arrow{r}{\tilde{h}_2^*} \arrow[swap]{d}
&HF^*(P_{B_2},f_2^*A_{f_2}) \arrow[swap]{d}
\\
HF^*(P_{B_2 \times B_1},\tilde{h}_2^*A_{f_2}) \arrow{r}{\tilde{h}_2^*}
&HF^*(P_{B_2},A_{f_2})
\end{tikzcd}
\end{equation}

Recall the bounds from \eqref{eq:SlopeBound} and \eqref{eq:SlopeBound2}.
By Lemma \ref{l:inde_kappa}, we have the following commutative diagram
\begin{equation}\label{eq:4th-square}
\begin{tikzcd}
HF^*(P_{B_2 \times B_1},F_2^*A_F)   \arrow{r} \arrow[swap]{d}{\cC_{F_2}}
&HF^*(P_{B_2 \times B_1},a_1)   \arrow{r} \arrow[swap]{d}
&HF^*(P_{B_2 \times B_1},\tilde{h}_2^*f_2^*A_{f_2}) \arrow[swap]{d}
\\
HF^*(P_{B_2 \times B_1},A_F)  \arrow{r} 
&HF^*(P_{B_2 \times B_1},a_1)  \arrow{r} 
&HF^*(P_{B_2 \times B_1},\tilde{h}_2^*A_{f_2})
\end{tikzcd}
\end{equation}
By Lemma \ref{l:indepClassifyingModel}, for the second vertical map, we can use either the Borel model $(B_2 \times ET \times B_1 \times M \times ET)/T$ or $(B_2 \times B_1 \times M \times ET)/T$ for $P_{B_2 \times B_1}$.
For the square on the left, we picked the former choice when we define $\cC_{F_2}$. 
However, for the square on the right, we can use the latter model to make it consistent with \eqref{eq:3rd-square}
Now, by combining \eqref{eq:3rd-square} and \eqref{eq:4th-square}, we get $\Pi_2$, $\Pi_3$ as well as the commutativity of the middle right square in \eqref{eq:BigCom}.

By Lemma \ref{l:inde_kappa}, \ref{l:inde_pb} and \ref{l:inde_pf}, we obtain the commutative diagram
\begin{equation}\label{eq:2ndLastSq}
\resizebox{\displaywidth}{!}{
\xymatrix{
 H^*_T(B_2) \otimes HF^*(P_{B_1}, A_{f_1}) \ar[r] \ar[d]
& H^*_T(B_2) \otimes HF^*(P_{B_1}, a_1) \ar[r] \ar[d]
& H^*_T(B_2) \otimes HF^*_T( M,a_1) \ar[dr] \ar[d]
\\
 HF^*(P_{B_2 \times B_1},F_2^*A_F) \ar[r]
& HF^*(P_{B_2 \times B_1},a_1) \ar[r]
& HF^*(P_{B_2},a_1) \ar[r]
& HF^*(P_{B_2},f_2^*A_{f_2})
}}
\end{equation}
which gives the commutativity of the top right square in \eqref{eq:BigCom}.

By Lemma \ref{l:inde_kappa}, \ref{l:inde_pf} and Corollary \ref{c:natural_pf}, we obtain the commutative diagram
\begin{equation}\label{eq:LastSQ}
\xymatrix{
HF^*(P_{B_2 \times B_1},A_F) \ar[d] \ar[r]
& HF^*(P_{B_2},A_{f_2})  \ar[d] \ar[dr]
\\
HF^*(P_{B_2 \times B_1},a_2) \ar[r]
& HF^*(P_{B_2},a_2)  \ar[r]
& HF^*_T(M,a_2) 
}
\end{equation}
which gives the commutativity of the bottom right triangle in \eqref{eq:BigCom}.

\begin{thm}\label{thm:Seideltequiv}
The map
\begin{align*}
\hat{H}^T_*(\Omega G) \to SH_T^*(M) \\
[B,\alpha,f] \mapsto \cS_f(\alpha, e_M)
\end{align*}
is an algebra homomorphism (cf. \eqref{eq:moduleactintroT}).
\end{thm}

\begin{proof}
As explained in the paragraph after \eqref{eq:BigCom}, it suffices to verify the commutativity of \eqref{eq:BigCom} and the compatibility with passing to the direct limit with respect to slope.
The commutativity of \eqref{eq:BigCom} follows from Proposition \ref{p:double_taut}, Lemma \ref{l:trivialCom}, \ref{l:1st-square}, and Equations \eqref{eq:3rd-square}, \eqref{eq:4th-square}, \eqref{eq:2ndLastSq} and \eqref{eq:LastSQ}.
The compatibility of this commutative diagram with respect to increasing the slope is straightforward and left to the readers.
\end{proof}

\begin{thm}[=Theorem \ref{t:SeidelMapEqintro}]\label{t:SeidelMapEqconvex}
There is a ring homomorphism:
\begin{align} \label{eq:moduleactS5}
\mathcal{S}: \hat{H}_*^{G}(\Omega_{poly} G) \to SH_{G}^*(\bar{M}).
\end{align} 
\end{thm}
\begin{proof}  There is a natural isomorphism (\cite[Lemma 5.3]{BFN}) : \begin{align} \label{eq:Weyl} \hat{H}_*^{T}(\Omega_{poly} G,\mathbb{C})^{W} \cong \hat{H}_*^{G}(\Omega_{poly} G,\mathbb{C}).  \end{align} By Theorem \ref{thm:Seideltequiv},  there is a ring map
\begin{align}\label{eq:ring}
\mathcal{S}: \hat{H}_*^{T}(\Omega_{poly} G) \to SH_{T}^*(\bar{M}).
\end{align} 

Note that, by Lemmas \ref{l:normal} and \ref{l:W-inv}, we have
\[
HF^*_G(P,a) \simeq HF^*_T(P,a)^W
\]
so by the taking direct limit, we have
\[
SH^*_G(M)=SH^*_T(M)^W.
\]
By Proposition \ref{prop:SeidelWeyleq},  this map \eqref{eq:ring} is $W$-equivariant and hence taking $W$- invariants induces a map of the form \eqref{eq:moduleactS5}.  \end{proof}

\section{Coulomb branches and symplectic cohomology}\label{s:Coulomb}
\subsection{Background on Coulomb branches}

In this section,  we recall some relevant facts about Coulomb branch algebras.  As in previous sections,  we let $T \subset G$ be a maximal torus and $W$ denote the Weyl group $N(T)/T$.  \vskip 5 pt

\subsubsection{Pure Coulomb branches} 

\begin{defn} The algebra $\co$ is defined to be the vector space $\hat{H}_*^{G}(\Omega_{poly} G,\mathbb{C})$ equipped with the Pontryagin product.  \end{defn} 

The two basic geometric facts concerning $\operatorname{Spec}(\co)$ are the following (\cite{BFM}): 

\begin{enumerate} \item  $\hat{H}_*^{G}(\Omega_{poly} G,\mathbb{C})$ is  a Hopf algebra over $H^*(BG,\mathbb{C}).$ As a consequence,  $\operatorname{Spec}(\co)$ has the structure of a group scheme over $\operatorname{Spec}(H^*(BG,\mathbb{C})).$ 
\item The spectrum $\operatorname{Spec}(\co)$ is a smooth holomorphic symplectic manifold.  Moreover the ``Toda projection" \begin{align} \label{eq:Toda} \pi_{Toda}:  \operatorname{Spec}(\co) \to \operatorname{Spec}(H^*(BG,\mathbb{C})) \end{align}  defines a completely integrable system. 
\end{enumerate} 

 A root $\alpha$ is given by a homomorphism $\mathfrak{t}_\mathbb{C} \to \mathbb{C}$.  We let $D_\mathfrak{t} \subset \mathfrak{t}_{\mathbb{C}}$ denote the union of the kernels of the root homomorphisms. Similarly, for a co-root $\alpha^{\vee}$, its exponential defines a homomorphism $T_\mathbb{C}^{\vee} \to \mathbb{C}^*.$ We let $D_{T^{\vee}} \subset T_{\mathbb{C}}^{\vee}$ denote the union of the kernels of these co-root homomorphisms. There is an isomorphism: \begin{align} \hat{H}_*^{T}(\Omega_{poly} G,\mathbb{C}) \cong \widetilde{\co}, \end{align}  where $\widetilde{\co}$ is the ``affine blowup" of $T^*T_\mathbb{C}^{\vee}$ with discriminant locus $D_{T^{\vee}} \times D_\mathfrak{t} \subset T^*T_\mathbb{C}^{\vee}$ from \cite[\S 2.5]{BFM}. Combining this with the isomorphism \eqref{eq:Weyl}, we obtain that \begin{align} \label{eq:weylhatGhatT} \hat{H}_*^{G}(\Omega_{poly} G,\mathbb{C}) \cong \widetilde{\co}^{W}. \end{align} Note that the isomorphism \eqref{eq:weylhatGhatT} gives rise to a birational map \begin{align} \label{eq:resof} \pi_{T,W}: \operatorname{Spec}(\co) \to \operatorname{Spec}(\hat{H}_*^{T}(\Omega_{poly} T,\mathbb{C})^{W}) \cong T^*T_\mathbb{C}^{\vee}/W,  \end{align} which is an isomorphism away from the Weyl quotient of pre-image of the root hyperplanes under $\pi_{Toda}$.  
\vskip 5 pt

 \begin{example} \label{ex:SU2} Let $G=SU(2)$ and $T$ be its maximal torus.  Take $\mathcal{C}(T,0)$ with coordinates $z^{\pm 1},  \tau.$ Following the calculations from \cite[\S 3.2]{BFM},  we adjoin two more variables: \begin{align} u= \frac{z-1}{\tau},  v= \frac{1-1/z}{\tau} \end{align} to $\mathcal{C}(T,0)$ to obtain the ring $\widetilde{\co}$.  
By a simple calculation, we have 
\[
\widetilde{\co}=\mathbb{C}[\tau,u,v]/(u-v-\tau uv).\]
 The Weyl group $W:=\mathbb{Z}/2\mathbb{Z}$ acts on this ring by $ u \to v , \tau \to -\tau.$ The pure Coulomb branch is then isomorphic to the Weyl invariants: 
\begin{align} 
\co \cong \widetilde{\co}^{W}=\mathbb{C}[\sigma, x, y]/(x^2-\sigma y^2-2y). 
\end{align}  
where $\sigma=\tau^2$, $x=(u+v)/2$ and $y=uv/2$.
\end{example}

\subsubsection{Teleman's construction}

We next review Teleman's description (\cite{Teleman1}) of the Coulomb branch associated to a ``cotangent type" representation $E$ of $G_\mathbb{C}.$ Being of cotangent type means that it comes with a (fixed) decomposition of the form \begin{align} \label{eq:decomposition} E:=\mathbb{V}\oplus \mathbb{V}^{\vee}. \end{align}  To describe the Coulomb branches,  it is convenient to introduce an additional parameter $\mu$.   The Coulomb branch $\cE$ arises as the fiber over $\mu=0$ in a flat family $\rcE$ over $\mathbb{C}[\mu]$.   We set $\rco:=\co[\mu]$ and denote the parameterized (or ``massive") version of  \eqref{eq:Toda} by \begin{align} \label{eq:todacirc} \mathring{\pi}_{Toda}:\operatorname{Spec}(\rco) \to \operatorname{Spec}(H^*(BG,\mathbb{C})[\mu]).  \end{align} Let $\nu \in \mathfrak{t}^{\vee}$ be a weight of $\mathbb{V}$ under the action of $T$.  We let $\psi_\nu: \mathfrak{t}_\mathbb{C} \times \mathbb{C} \to \mathbb{C}$ denote the linear function \begin{align} \psi_\nu(\eta,\mu) = \mu+\langle \nu|\eta \rangle. \end{align} Let $(\mathfrak{t}_\mathbb{C} \times \mathbb{C})^o$ denote the complements of the hyperplanes cut out by $\psi_\nu(\eta,\mu)=0.$ For any $w \in \mathbb{C}^*$ and a weight $\nu$,  we let $w^\nu:=\operatorname{exp}(\nu\operatorname{log}(w)) \in T_\mathbb{C}^{\vee}$.  Consider the following map: $(\mathfrak{t}_\mathbb{C} \times \mathbb{C})^o \to T_\mathbb{C}^{\vee},$ \begin{align}\label{eq:sectiont}(\eta, \mu) \to \prod_\nu (\psi_\nu(\eta, \mu))^{\nu}. \end{align}

Over $(\mathfrak{t}_\mathbb{C} \times \mathbb{C})^o$,  this defines a $W$-equivariant section,  $\epsilon_\mathbb{V}^{\operatorname{pre}}$,  of $$\mathring{\pi}_{Toda}^{T}:\operatorname{Spec}(\mathcal{C}^{\circ}(T;0)) \to \operatorname{Spec}(H^*(BT,\mathbb{C})[\mu]).$$ It therefore descends to a rational section $\hat{\epsilon}_\mathbb{V}$ of $$T^*T_\mathbb{C}^{\vee}/W \times \mathbb{C} \to \mathfrak{t}_\mathbb{C}/W \times \mathbb{C}.$$ 

\begin{defn} \label{defn:eulerLag} We define $\epsilon_\mathbb{V}$ to be the proper transform of $\hat{\epsilon}_\mathbb{V}$ to   $\operatorname{Spec}(\rco).$  For each $\mu \in \mathbb{C}$,  $\epsilon_\mathbb{V}$ is a rational section of the Toda integrable system and thus defines a family of holomorphic Lagrangians.  We refer to $\epsilon_\mathbb{V}$ as the Euler Lagrangian. \end{defn} 

\begin{rem} \label{rem:blowupfirst} Equivalently,  we can first take the proper transform of the rational section $\epsilon_\mathbb{V}^{\operatorname{pre}}$ defined by \eqref{eq:sectiont} to the blow-up $\operatorname{Spec}(\widetilde{\rco}):= \operatorname{Spec}(\widetilde{\co}[\mu])$ and then take $W$-invariants.   \end{rem}

Let $U:= \mathring{\pi}_{Toda}^{-1}((\mathfrak{t}_\mathbb{C} \times \mathbb{C})^o).$ Translation by $\epsilon_\mathbb{V}$  using the group scheme structure on $\operatorname{Spec}(\rco)$ defines a rational fiberwise symplectomorphism \begin{align} \label{eq:eplus} \varepsilon_\mathbb{V}^{+}: \rco_{|U} \cong \rco_{|U} . \end{align}

\begin{thm} \cite[Theorem 1]{Teleman1} \label{thm:Tel}  Let $\mathcal{A}_\mathbb{V}$ denote the scheme given by gluing two copies of $\operatorname{Spec}(\rco)$ by $\varepsilon_\mathbb{V}^{+}.$ Then \begin{align} \rcE \cong \Gamma(\mathcal{A}_\mathbb{V}). \end{align} \end{thm}

In purely algebraic terms,  the description of  $\rcE$ from Theorem \ref{thm:Tel} is equivalent to saying that $\rcE \subset \rco$ is the subring of $\rco$ which remain regular after applying the translation \eqref{eq:eplus}.  

\begin{example} \label{ex:weightpm1} \cite[Example 5.2]{Teleman1} Let $G=U(1)$ and $\mathbb{V}=\mathbb{C}^2$ be a rank 2 representation with weight $(1,-1).$ Let $z^{\pm 1}, \mu, \tau$ be the coordinates on $\operatorname{Spec}(\rco)$.  The rational automorphism given by translation is: \begin{align} \label{eq:translateSU2} \varepsilon_\mathbb{V}^{+}: z \to z(\frac{\mu+\tau}{\mu-\tau}). \end{align} The subring of functions which remain regular after this translation is generated by \begin{align} \mu, \tau, \quad x= z(\mu-\tau),  y= z^{-1}(\mu+\tau). \end{align} The ring $\rcE$ is therefore given by \begin{align} \rcE \cong \mathbb{C}[\mu,\tau, x,y]/(xy=\mu^2-\tau^2). \end{align}    
Note that after setting $\mu=0$, we see that $\Spec(\cE)$ has an $A_1$ singularity.
\end{example}

\begin{example} \cite[Example 5.3]{Teleman1} Let $G=SU(2)$ and $\mathbb{V}= \mathbb{C}^2$ be the standard representation.  The restriction of this representation to the maximal torus $T$ gives the previous example.  The pure Coulomb branches $\rco$ and $\widetilde{\rco}$ have been computed in Example \ref{ex:SU2} (up to adjoining the formal variable $\mu$).  On the Weyl-cover, $\widetilde{\rcE}$, the relevant automorphism is still given by \eqref{eq:translateSU2} and the subring $\widetilde{\rcE} \subset \widetilde{\rco}$ which remains regular after translation by the rational section is generated by \begin{align}x:= \mu u-z, y= \mu v-z^{-1},  w:= \frac{x-y}{\tau}.  \end{align} In this case, we again have $\rcE=\widetilde{\rcE}^{W}.$  \end{example}

 Under the gluing construction, the rational section $\epsilon_{\mathbb{V}}$ is identified with an open subset of the unit section $\mathbb{L}_{id}$ in the other copy of $\operatorname{Spec}(\rco)$ (``chart").  Note that the image of this unit section $\mathbb{L}_{id}$ in Teleman's scheme $\mathcal{A}_\mathbb{V}$ remains a section of the Toda projection \eqref{eq:todacirc} and hence its image is automatically closed inside of the affinization $\operatorname{Spec}(\rcE)$ (because functions pulled back along the Toda projection restrict surjectively and both schemes in question are affine). It follows that the closure of the subvariety $\epsilon_{\mathbb{V}}$  inside of $\operatorname{Spec}(\rcE)$ coincides with the image of $\mathbb{L}_{id}$ and hence defines a section $\bar{\epsilon}_{\mathbb{V}} \subset \operatorname{Spec}(\rcE)$ of the projection \eqref{eq:todacirc} (c.f.  \cite[Theorem 3]{Teleman1}).  Conversely, this property characterizes $\operatorname{Spec}(\rcE)$ as follows: 

\begin{cor} \label{cor:characterization} Let $\mathcal{R}$ be an $H^*(BG)[\mu]$ algebra such that $\operatorname{Spec}(\mathcal{R})$ is a $\operatorname{Spec}(\rco)$-equivariant partial compactification of $\operatorname{Spec}(\rco)$ which compactifies $\epsilon_{\mathbb{V}}$ to a section $\bar{\epsilon}_{\mathbb{V}}$ of the massive Toda projection.  Then $\mathcal{R}$ is a subalgebra of $\rcE$.  \end{cor} 
\begin{proof} The orbit of $\bar{\epsilon}_{\mathbb{V}}$ under the group defines a second chart $\operatorname{Spec}(\rco) \to \operatorname{Spec}(\mathcal{R})$.  Together our two orbits/charts give a map from Teleman's scheme to $\operatorname{Spec}(\mathcal{R})$: \begin{align} \mathcal{A}_\mathbb{V} := \operatorname{Spec}(\rco) \cup \operatorname{Spec}(\rco) \to \operatorname{Spec}(\mathcal{R}).  \end{align} 
There is a pull-back on global functions $\mathcal{R} \to \Gamma(\mathcal{A}_\mathbb{V})= \rcE$ which must be injective.  In other words,  $\mathcal{R}$ is a subalgebra of $\rcE$.  \end{proof}

We conclude this section by mentioning a few further properties of Coulomb branches which are important,  but not needed for our immediate purposes:  \begin{enumerate} \item The Coulomb branch $\cE$ (and indeed the entire family $\rcE$)  can be shown to be independent of the decomposition \eqref{eq:decomposition}.  However,  this requires further analysis (see \cite[\S 6(viii)]{BFN}).    \item Note that in Example \ref{ex:weightpm1},  the Coulomb branch at $\mu=0$ is singular as it is in general a singular (Poisson affine) variety.  Canonical (partial) resolutions of Coulomb branches have been studied in \cite{BFNlinebundles}.  \end{enumerate}  

\subsection{Proof of Theorem \ref{thm:moduleintro}} 
Let $(\bar{M}, \omega,  \theta)$ be a Liouville domain with a convex Hamiltonian $G$-action.  Recall from Remark \ref{rem:LiouvilleC} that $SH_{G}^*(\bar{M},\mathbb{C})$ denotes the version of symplectic cohomology defined over $\mathbb{C}$.  We have the following obvious variant of Theorem \ref{t:SeidelMapEqconvex}:
\begin{cor}\label{t:SeidelMapEq}
 Let $(\bar{M}, \omega,  \theta)$ be a Liouville domain with a convex Hamiltonian $G$-action.  Then there is a ring homomorphism:
\begin{align} \label{eq:moduleact}
\mathcal{S}: \hat{H}_*^{G}(\Omega_{poly} G) \to SH_{G}^*(\bar{M},\mathbb{C}).
\end{align} 
\end{cor}

 Next,  let $\mathbb{V}:=\mathbb{C}^n$ be a unitary representation of $G$.  Introduce complex coordinates $z_i$ on $\mathbb{C}^n$ as well as corresponding polar coordinates $(r_i,\theta_i)$.  Equip this with the standard symplectic form \begin{align} \omega_{\mathbb{C}^n}:= \sum_i r_id r_i\wedge d\theta_i. \end{align} 
The Hamiltonian $K =\pi |z|^2$ generates the diagonal circle action $S^1 \times \mathbb{C}^n \to \mathbb{C}^n$ \begin{align} \label{eq: diagonal} e^{it} \cdot (z_1,\cdots,z_n)=  (e^{2 \pi it}z_1,\cdots, e^{2 \pi i t}z_n).  \end{align} 
The unit ball $\bar{\mathbb{V}}$ is a Liouville domain with contact boundary $S^{2n-1}$.  We consider the $G \times S^1$ equivariant symplectic cohomology $SH^*_{G\times S^1}(\bar{\mathbb{V}},\mathbb{C})$ (as well as $SH^*_{T\times S^1}(\bar{\mathbb{V}},\mathbb{C})$),  where the additional $S^1$-factor corresponds to the diagonal rotation \eqref{eq: diagonal}.  The diagonal rotation defines a Seidel operator $s_\Delta$ which acts on $H^*_{T \times S^1}(\bar{\mathbb{V}},\mathbb{C})$.  This Seidel operator determines symplectic cohomology as follows: 

\begin{lem}[Localization] \label{eq:Localization} There is an isomorphism: 
 \begin{align} SH^*_{T\times S^1}(\bar{\mathbb{V}},\mathbb{C}) \cong H^*_{T\times S^1}(\bar{\mathbb{V}},\mathbb{C})[s_\Delta^{-1}]. \end{align}   \end{lem}

\begin{proof} See \cite{Ritter} for the non-equivariant version and \cite[Section 5.1]{LJ21} for the equivariant version. \end{proof}

Next,  view $\rco:=\hat{H}_{*}^{G}(\Omega_{poly} G)[\mu]$ as a subalgebra of $\hat{H}_*^{G \times S^1}(\Omega_{poly} (G\times S^1)) \cong \hat{H}_{*}^{G}(\Omega_{poly} G)[\mu,w^{\pm}]$ given by those elements which are constant in $w$.  There is therefore a ring homomorphism: \begin{align} \label{eq:moduleactS1} \rco \to SH^*_{G \times S^1}(\bar{\mathbb{V}},\mathbb{C}).  \end{align}

\begin{thm}[=Theorem \ref{thm:moduleintro}] \label{thm:module} The following hold: \begin{enumerate} \item There is a $\mathbb{C}$-algebra isomorphism $\Gamma(\mathcal{O}_{\epsilon_\mathbb{V}}) \cong SH^*_{G\times S^1}(\bar{\mathbb{V}},\mathbb{C}).$ Moreover,  the inclusion $\epsilon_{\mathbb{V}}$ corresponds to the homomorphism \eqref{eq:moduleactS1}.  \item There is a commutative diagram: \begin{equation} \label{eq:diagrams7} \begin{tikzcd}
\rcE \arrow{r}{\mathcal{S}} \arrow[swap]{d}{i} & H^*_{G \times S^1} (\bar{\mathbb{V}},\mathbb{C})  \arrow{d}{\operatorname{ac}} \\
\co[\mu] \arrow{r}{\mathcal{S}} & SH^*_{G \times S^1}(\bar{\mathbb{V}},\mathbb{C}).  \end{tikzcd} \end{equation} \end{enumerate}  \end{thm} 

\begin{proof} \emph{Proof of Claim (1)}: We first calculate the Seidel operator for the diagonal circle action \eqref{eq: diagonal} viewed as an operation on $H^*_{T \times S^1}(\bar{\mathbb{V}},\mathbb{C})$.  Suppose first that $\mathbb{V}:=\mathbb{C}_\nu$,  a rank-one representation of $T$ with weight $\nu$.  The diagonal action acts by the Hamiltonian loop $t \to e^{2 \pi it} z.$ Then by the argument of \cite[Theorem 5.6]{LJ21}\footnote{Theorem 5.6 of \cite{LJ21} assumes that the toric variety is compact and that the torus acting is the maximal torus,  however the proof carries over without change to our setting.}  the Seidel operator for this loop is given by $\psi_\nu(\eta,\mu) := \mu+\langle \nu|\eta \rangle$ viewed as an element of $H^2_{T \times S^1}(\bar{\mathbb{V}})$.
Here, we identify $H^2_{T \times S^1}(\bar{\mathbb{V}})$ with $\operatorname{Hom}(\mathfrak{t}_\mathbb{C} \times \mathbb{C}, \mathbb{C})$.
 In general,  our representation of $T$ decomposes as a direct sum of these one-dimensional representations,  and the Seidel operator $s_\Delta$ for the diagonal circle action is clearly the product of the Seidel operators for these representations.  In other words,  we have \begin{align} s_\Delta = \prod_\nu \psi_\nu(\eta,\mu). \end{align}   By Lemma \ref{eq:Localization},  it follows that \begin{align} (\mathfrak{t}_\mathbb{C} \times \mathbb{C})^o \cong SH^*_{T \times S^1}(\bar{\mathbb{V}},\mathbb{C}).  \end{align} Now trivialize $T \cong (S^1)^r.$ The $i$-th circle will act on each of the one-dimensional representations with some weight $\nu_i$ and the corresponding Seidel operator $s_i$ will be \begin{align} s_i:= \prod_\nu \psi_\nu(\eta,\mu)^{\nu_i}. \end{align} It follows that the  homomorphism  $\mathcal{C}(T,0)[\mu] \to SH^*_{T \times S^1}(\bar{\mathbb{V}},\mathbb{C})$ is given by \eqref{eq:sectiont}.  By (the $T$-equivariant version of) Corollary \ref{t:SeidelMapEq},  this homomorphism extends to \begin{align} \label{eq:lift} \widetilde{\rco}:= \hat{H}_{*}^{T}(\Omega_{poly} G)[\mu] \to SH^*_{T \times S^1}(\bar{\mathbb{V}},\mathbb{C}).  \end{align}  As the subscheme corresponding to \eqref{eq:sectiont} is irreducible and does not lie entirely in the blowup locus,  the subscheme corresponding to \eqref{eq:lift} must coincide with its  proper transform.  By Remark \ref{rem:blowupfirst},  taking Weyl invariants gives the homomorphism \eqref{eq:moduleactS1} as claimed.    \vskip 10 pt

\emph{Proof of Claim (2)}: As in  Corollary \ref{cor:characterization},  we let $\bar{\epsilon}_{\mathbb{V}}$ denote the compactified section.  The image of the pull-back of $\Gamma(\mathcal{O}_{\bar{\epsilon}_{\mathbb{V}}}) \cong H^*(BG)[\mu]$ is characterized intrinsically as the free rank-one $H^*(BG)[\mu]$ submodule of $\Gamma(\mathcal{O}_{\epsilon_{\mathbb{V}}})$ containing the unit.  Under the isomorphism $\Gamma(\mathcal{O}_{\epsilon_\mathbb{V}}) \cong SH^*_{G\times S^1}(\bar{\mathbb{V}},\mathbb{C})$ from  Theorem \ref{thm:module},  this corresponds to the inclusion $H^*_{G \times S^1}(\bar{\mathbb{V}},\mathbb{C}) \to SH^*_{G\times S^1}(\bar{\mathbb{V}},\mathbb{C}).$      \end{proof} 

\begin{cor} \label{cor:characterization2} Let $\mathcal{R}$ be an $H^*(BG)[\mu]$ algebra such that $\operatorname{Spec}(\mathcal{R})$ is a $\operatorname{Spec}(\rco)$-equivariant partial compactification of $\operatorname{Spec}(\rco)$ where $\mathcal{R}$ fits into \eqref{eq:diagrams7}.  Then  $\mathcal{R}$ is a subalgebra of $\rcE.$\end{cor} 
\begin{proof} This is a direct reinterpretation of Corollary \ref{cor:characterization} in view of Theorem \ref{thm:module}. \end{proof} 




\bibliographystyle{amsalpha}
\bibliography{biblio}


{\small

\medskip
\noindent Eduardo Gonz\'alez \\
\noindent University of Massachusetts, Boston, 100 William T, Morrissey Blvd, Boston, MA 02125, US\\
{\it e-mail:} Eduardo.Gonzalez@umb.edu
\medskip

\medskip
\noindent Cheuk Yu Mak\\
\noindent University of Sheffield, Hicks Building, School of Mathematical and Physical Sciences, S10 2TN, UK\\
{\it e-mail:} c.mak@sheffield.ac.uk

\medskip
 \noindent Dan Pomerleano\\
\noindent University of Massachusetts, Boston, 100 William T, Morrissey Blvd, Boston, MA 02125, US\\
 {\it e-mail:} Daniel.Pomerleano@umb.edu

}
\end{document}